\documentclass[12pt]{amsart}
\usepackage{amsfonts,amssymb,amsmath,amscd,amsthm, color}
\usepackage{amstext,amssymb}
\usepackage{mathrsfs}
\usepackage{txfonts}
\usepackage{tikz}

 \usepackage[colorlinks, citecolor=blue,pagebackref,hypertexnames=false]{hyperref}

 \usetikzlibrary{shapes,shadows,calc}
\usepgflibrary{arrows}

\usepackage{enumerate}
\usepackage{fancyhdr}
\usepackage{bm}
\def\cube1[#1,#2,#3,#4,#5]#6{ %
  \node[draw, rectangle, shape aspect=#3, rotate=#2, minimum size=#1, %
 ] (#4) at #5
  {\textcolor{white!53}{ }}; %
  \node at #5 {#6}; %
}

\allowdisplaybreaks[4]

\newtheorem{theorem}{Theorem}[section]
\newtheorem{proposition}[theorem]{Proposition}
\newtheorem{lemma}[theorem]{Lemma}
\newtheorem{corollary}[theorem]{Corollary}

\def\support {{\rm supp}\,}
\def\RR{\Bbb R}
\def \SL {\sqrt{L}}
\def \f {\frac}
\def \RN {\mathbb{R}^n}
\renewcommand{\L}{\mathscr{L}}

\title[Exponential-square integrability, weighted inequalities  for the square function]
{   Exponential-square integrability,  weighted inequalities  for the square functions associated to operators
and applications    }

\author{Peng Chen,\ Xuan Thinh Duong,\ Liangchuan Wu\ and\ Lixin Yan}

\address{Peng Chen, Department of Mathematics, Sun Yat-sen University, Guangzhou, 510275, P.R.~China}
\email{chenpeng3@mail.sysu.edu.cn}
\address{Xuan Thinh Duong, Department of Mathematics, Macquarie University, NSW 2109, Australia}
\email{xuan.duong@mq.edu.au}
\address{Liangchuan Wu, Department of Mathematics, Sun Yat-sen  University,
Guangzhou, 510275, P.R.~China}
\email{wulchuan@mail2.sysu.edu.cn}
\address{Lixin Yan, Department of Mathematics, Sun Yat-sen University, Guangzhou, 510275, P.R.~China}
\email{mcsylx@mail.sysu.edu.cn
}

\subjclass[2010]{42B20, 42B25, 47B38, 58G35.}
\keywords{ Exponential-square integrability, weighted norm inequalities, the square function,
the dyadic square function, semigroups, spectral multiplier, eigenvalue estimate,  space of homogeneous type.}

\date{\today}

\begin{document}

\begin{abstract}
Let $X$ be a metric space with a doubling measure. Let $L$ be a nonnegative self-adjoint operator
  acting on $L^2(X)$, hence $L$ generates an analytic semigroup $e^{-tL}$. Assume that
  the kernels $p_t(x,y)$ of $e^{-tL}$  satisfy
Gaussian upper bounds  and H\"older's continuity in $x$ but  we do not require the semigroup to satisfy
the preservation condition $e^{-tL}1 = 1$.
In this article we aim to establish the exponential-square integrability of a function whose square function
associated to an operator $L$ is bounded, and the proof   is new even for the Laplace operator on
 the Euclidean spaces  ${\mathbb R^n}$.
 We then apply this result to obtain: (i) estimates of the norm on $L^p$   as $p$ becomes large  for operators such as
 the square functions or spectral
 multipliers; (ii) weighted norm inequalities for the square functions; and
 (iii) eigenvalue estimates for   Schr\"odinger operators
 on ${\mathbb R}^n$ or Lipschitz domains of ${\mathbb R}^n$.

\end{abstract}

\maketitle

\tableofcontents


\section{Introduction}
\setcounter{equation}{0}

\noindent
{\bf 1.1. Background.}\
Consider the Laplace operator $\Delta=-\sum_{i=1}^n\partial_{x_i}^2$ on the Euclidean space $\mathbb R^n$.
For   a (suitable) function $f(x)$ on $\mathbb R^n$, the Poisson
integral $u(x,t)=e^{-t\sqrt{\Delta}} f(x)$ is  the harmonic extension of $f$ to
 $\mathbb R^{n+1}_{+}$.
For $\alpha>0$,  define the (classical) square function (or classical Lusin area inegral) by
\begin{eqnarray}\label{e1.1}
S_{\alpha}f(x)=\left(\iint_{|x-y|<\alpha t}
\big| \nabla_y u(y, t)  \big|^2  \ t^{1-n}   \,   dy dt \right)^{1/2}.
\end{eqnarray}
It is  known  that if $S_{\alpha}f$ is a bounded function, then $f$ belongs to the space
 ${\rm BMO}$  and hence the John-Nirenberg inequality implies that
 $f$ is in the exponential  class. In other words,
$\exp \left(c|f|/\|S_{\alpha}f\|_{L^\infty }\right)\in L^1_{\rm loc}(\mathbb R^n)$ for a suitable   constant $c>0.$
(See \cite{St} for the definition  and properties of BMO functions.)
Moreover, a celebrated   theorem of
Chang-Wilson-Wolff (\cite[Theorem 3.2]{CWW}) states that
\begin{eqnarray}\label{e1.2}
\exp \left(c|f|^2/\|S_{\alpha}f\|^2_{L^\infty}\right)\in L^1_{\rm loc}(\mathbb R^n)
\end{eqnarray}
for a suitable constant $c>0$ depending on $n$ and $\alpha.$

This result \eqref{e1.2}  has been useful in the work of R. Fefferman, C. Kenig and J. Pipher; in the study
of Hardy spaces; and also in the work of C. Moore and M. Wilson (see page 486, \cite{CCPMLMS}).
Due to R. Fefferman and J. Pipher, it  was observed  in \cite{FP97}
that   \eqref{e1.2}
implies that for   the Hilbert transform $H$ (or other singular integrals) on ${\mathbb R}$, as $p\to \infty,$
\begin{eqnarray*}
 \big\|Hf \big\|_{L^p(\RR)} \leq Cp \, \big\|f\big\|_{L^p(\RR)},
\end{eqnarray*}
which is well-known to be the sharp dependence on $p$. Analogous sharp estimate  holds for the square
function,
\begin{eqnarray*}
 \big\|f\big\|_{L^p(\RN)} \leq C p^{1/2} \, \big\|S_{\alpha}f\big\|_{L^p(\RN)}
\end{eqnarray*}
as $p\to \infty$ (see page 357, \cite{FP97}).  We note  that $\big\|S_{\alpha}f\big\|_{L^p(\RN)}\leq C p^{1/2} \,
 \big\|f\big\|_{L^p(\RN)}$ follows from Theorem 1.2 in \cite{CWW} with a constant $C$ that
 depends on $n$ and $\alpha$.

The result \eqref{e1.2}  also has other applications; for example, M. Wilson used  it  to derive
  sharp weighted inequalities for the  square function
to obtain Sobolev and singular integral inequalities and
  eigenvalue estimates for degenerate Schr\"odinger operators (see  \cite {Wi87, Wi, Wi08}).
In the last few decades, this  result \eqref{e1.2}   has been  studied extensively,
see for examples, \cite{Fe87, Pe, Swe91,  Wi87, Wi, Wi08}
and the references therein.

\medskip

\noindent
{\bf 1.2. Assumptions, notations and definitions.}\,
Our setting will be the following.
We consider a nonnegative self-adjoint operator  $L$ on $L^2(X)$ where $(X, d, \mu)$ is a metric measure space which
satisfies the  volume doubling  condition
\begin{eqnarray*}
 \mu(B(x, \lambda r)) \le C \lambda^n \mu(B(x,r)), \ \,\;  \forall x \in X,\, \lambda \ge 1, \, r > 0,
 \end{eqnarray*}
where $C$ and $n$ are positive constants and $\mu(B(x,r))$ denotes the volume of the open ball $B(x,r)$  of centre $x$
and radius $r$.
In this article, we assume that  $L$ is a densely-defined operator
on $L^2( X).$ In different sections of the article,   we
assume  that $L$ satisfies some or all the following properties:
\begin{enumerate}
	\item [\textbf{(H1)}] $L$ is a nonnegative self-adjoint operator on $L^2(X)$;
	\item [\textbf{(H2)}] The kernel of $e^{-tL}$, denoted by $p_{t}(x,y)$, is
a measurable function on $X\times X$ and satisfies a Gaussian upper bound, that is
	\begin{eqnarray*}
	\left|p_{t}(x,y)\right| \leq \f{C}{\mu(B(x,\sqrt{t}))}
\exp\left(-\f{d(x,y)^2}{ct}\right)
	\end{eqnarray*}
	for all $t>0$ and $x,y\in X$, where $C$ and $c$ are positive constants.
	\item [\textbf{(H3)}] The kernel $p_{t}(x,y)$ satisfies the following H\"older's
continuity estimate: there exists a constant $\theta\in (0,1]$, such that
	\begin{eqnarray*}
	\big|p_{t}(x,y)-p_{t}(x',y)\big|\leq C\left(\f{d(x,x')}{\sqrt{t}}
\right)^\theta
\f{1}{\mu(B(y,\sqrt{t}))}
\end{eqnarray*}
	  for  $d(x,x')\leq \sqrt{t}$.
\end{enumerate}

The operator $L$ is a   nonnegative  self-adjoint operator on $L^2(X)$, hence $L$ generates the $L$-Poisson semigroup
$\{e^{-t\sqrt{L}}\}_{t>0}$ on $L^2(X).$  Given an operator    $L$ satisfying ${\bf (H1)}$  and ${\bf (H2)}$,
  and  a  function $f\in L^2(X)$,  consider the following
square function   $S_{L, \alpha}f$   associated to the Poisson  semigroup
generated by $L$
\begin{eqnarray}
S_{L, \alpha}f(x):=\left(\int_0^{\infty}\int_{d(y,x)<\alpha t}
|t^2Le^{-t\sqrt{L}} f(y)|^2 {\frac{d\mu(y)}{ \mu(B(x,t))}}{\frac{dt}{ t}}\right)^{1/2},
\quad \alpha>0,
\label{e1.7}
\end{eqnarray}
and set $S_{L}f=S_{L, 1}f$.

It is known that if $L$ satisfies assumptions  ${\bf (H1)}$ and ${\bf (H2)}$, then
the operator $S_{L, \alpha}$  is   bounded on $L^p(X), 1<p<\infty$.
(See for example,  \cite{Au07}, \cite[Proposition 3.3]{GY}.)

\medskip

\noindent
{\bf 1.3. Main results.}\ This article has two aims. The first one is to
address the following question:
Is it possible to establish   harmonic analysis properties of the exponential-square class of \eqref{e1.2}
on a set?  It would be nice  to find a proof which does not require the group structure.
In this article we will give an affirmative answer to this
question by extending the work of  Chang, Wilson and Wolff  (\cite[Theorem 3.2]{CWW})
for the square function \eqref{e1.7} on  a space of homogeneous type.
 We have the following  result.

 \medskip

\noindent{\bf Theorem A.} {\it
Assume that  $X$ satisfies the volume doubling condition.
 Suppose that
$L$ is a densely-defined operator on $L^2(X)$ satisfying
\textbf{(H1)}, \textbf{(H2)} and \textbf{(H3)}.
Let $f\in L^2(X,d\mu)$ and assume that $S_{L, \alpha} f\in L^\infty(X)$. Then there exist
 positive constants $ {\gamma}=\gamma(n,\alpha)$ and $ {C}=C(n, \alpha)$
 independent of $f$ such that
\begin{eqnarray}\label{e1.13}
 \sup_{B}  \frac{1}{\mu(B)} \int_{B} \exp\Bigg( {\gamma}\f{ |f(x)-
f_{B} |^2}
{ \|S_{L, \alpha}f \|_{L^\infty(X)}^2}\Bigg)\  d\mu(x) \leq C,
\end{eqnarray}
where $f_B={1\over \mu(B)}\int_B f(x)d\mu(x)$, and the {\it sup} is taken over all balls in $X.$
 }

 \medskip

We mention that the  proof of Chang-Wilson-Wolff is using reduction to
a related distribution function inequality for dyadic martingales on ${\mathbb R^n}$,
in which   the fundamental identity of sequential analysis in statistics was applied (\cite[Theorem 3.1]{CWW}).
In the proof of our Theorem A, the main device
is to use some dyadic structures of   $X$ to define
a version of  discrete   square functions
associated to an operator $L$
on  a space $X$ of homogeneous type. From this, we are able to  show that   boundedness of the discrete
square function implies its local  exponential-square integrability, and then apply the Calder\'on  machinery (\cite{C})
to  allow  us to translate the ``discrete" results  into the continuous setting. Our proof is different substantially
from the  proof of Chang-Wilson-Wolff which is not applicable  in our setting  since
no group structure is available and
the Fourier transform is missing in the general setting of space of homogeneous type.
Another observation is that  we do   not  assume that the semigroup  $\{e^{-tL}\}_{t>0}$  satisfies
	  the standard preservation condition, so we   may   have
\begin{eqnarray*}
e^{-tL}1 \not=1
\end{eqnarray*}
and this allows us to apply the results to a wide class of operators $L$.

We also note that the kernel for $e^{-tL}$  is not assumed to be translation invariant and several
techniques which can be used with
 the Poisson kernel of the Laplacian on ${\mathbb R}^n$
 are not applicable here.
This is indeed one of the  main obstacles in this article and makes the theory
quite subtle and delicate. We
  overcome  this problem in the proof in Section 4 by using
some estimates on heat kernel bounds,
finite  propagation speed of solutions to the wave equations and
spectral theory of nonnegative self-adjoint operators.

\smallskip

The second aim of this article is to apply \eqref{e1.13} to
give various applications.
The first is to obtain
estimates of the $L^p$ operator norms as $p \rightarrow \infty $ for operators such as
 the square functions or spectral
 multipliers. We have the following result.

\medskip
\noindent{\bf Theorem B.} \ {\it Assume that  $X$ satisfies the  volume doubling  condition.
 Assume that $L$ is a
densely-defined operator
on $L^2(X)$
satisfying \textbf{(H1)}, \textbf{(H2)} and \textbf{(H3)}.

\begin{itemize}
\item[(i)]
There exists a constant $C=C(n, \alpha)>0$   such that as $p\to \infty,$
\begin{eqnarray*}
\|f\|_{L^p(X)} \leq C  p^{1/2}\  \|S_{L, \alpha}f\|_{L^p(X)}.
\end{eqnarray*}

  \item[(ii)]
 Let $\beta>{n/2}$ and assume that  $F: [0, \infty)\rightarrow \mathbb{C}$ is a bounded Borel function such that
 $F(0)=0,$
$  \sup_{t>0}\|\eta\delta_tF\|_{W^{\infty}_\beta(\RR)}<\infty$
where $ \delta_t F(\lambda)=F(t\lambda)$ and $\eta\in C_0^{\infty}({\mathbb R_+})$
 is a
fixed function, not identically zero. Then there exists a constant $C=C(n, \beta)>0$  such that
 the operator $F(L)$ satisfies
$$
 \|F(L)f\|_{L^p(X)} \leq Cp\,  \sup_{t>0}\|\eta\delta_tF\|_{W^{\infty}_\beta(\RR)}   \|f\|_{L^p(X)}
$$
as $p\to \infty.$
\end{itemize}
 }

  We remark that  under the assumptions
\textbf{(H1)} and \textbf{(H2)} of the operator $L,$ it is known that there exists a constant $C=C(n, \alpha)>0$  such that
\begin{eqnarray*}
\|S_{L, \alpha}f\|_{L^p(X)}\leq C p^{1/2}\|f\|_{L^p(X)}
\end{eqnarray*}
 as $p\to \infty$ (see \cite{GY, GX}).
Combining this with (i) of Theorem B, we then obtain  the   estimate
\begin{eqnarray*}
c p^{-1/2}\|f\|_{L^p(X)} \leq \|S_{L, \alpha}f\|_{L^p(X)}\leq C p^{1/2}\|f\|_{L^p(X)}
\end{eqnarray*}
 as $p\to \infty$.

 \smallskip

Another application  is to apply \eqref{e1.13} to obtain
 a two weighted inequality for the square function $S_{L, \alpha}$
associated to an operator $L$ (see Theorem 5.6 below for precise statement).

\medskip
\noindent{\bf Theorem C.} \ {\it Assume that  $X$ satisfies the volume doubling condition
and  $L$ is a
densely-defined operator
on $L^2(X)$
satisfying \textbf{(H1)}, \textbf{(H2)} and \textbf{(H3)}.  Let $0<p<\infty$, $\tau>p/2$,
and let $V$ and $W$ be two weights such that
$$
\int_Q V(x)\left(\log \left(e+\f{V(x)}{V_Q}\right)\right)^{\tau} d\mu(x)\leq \int_Q W(x)d\mu(x)
$$
for all $Q\in   \mathcal{D}$.  Then for all $\alpha>0,$ there exists  a constant $C=C(p, n, \tau, \alpha)>0$ such that for every
 $f\in L_0^\infty(X)$,
$$
\int_X |f(x)|^p V(x)d\mu(x)\leq C\int_X   \big( S_{L,\alpha} f(x)\big)^p W(x)d\mu(x).
$$
 }
\medskip

 Finally, we study the eigenvalue problem of Schr\"odinger operators on ${\mathbb R^n}$ or a Lipschitz domain
 of ${\mathbb R}^n$.  In the following, we let $\Omega={\mathbb R}^n$  or
  $\Omega$ be a special Lipschitz domain of ${\mathbb R}^n.$
 Let $A(x)$ be an $n\times n$
matrix function with real symmetric, bounded measurable  entries  on
${\mathbb R}^n$ satisfying the ellipticity condition
\begin{eqnarray*}
\|A\|_{\infty}\leq \lambda^{-1}\ \ \ {\rm and}\ \ \
\  A(x)\xi\cdot\xi\geq \lambda |\xi|^2
\end{eqnarray*}
for some constant $\lambda\in (0,1)$,  for all $\xi\in \Omega$
and for almost all $x\in\Omega$.
For  a nonnegative  $V$ in $L^1_{\rm loc}(\Omega)$, we consider the Schr\"odinger operator $\mathcal{L}$,
 defined by
 \begin{eqnarray}
 \label{e1.15}
 \mathcal{L}u=-{\rm{div}}(A(x)\nabla u)-V u.
 \end{eqnarray}
We note that the spectrum of the operator  $L =-{\rm{div}}A \nabla  $ is contained in $[0, \infty)$.
However, the effect of the potential $V \geq 0$ is that the spectrum of  $\mathcal{L}$ might contain some negative values.
We have the following result.

\medskip
\noindent{\bf Theorem D.} {\it
 When $\Omega={\mathbb R}^n$  or
  $\Omega$ is a special Lipschitz domain of ${\mathbb R}^n$, let
 $ \mathcal{L}=-{\rm div} (A\nabla)-V$ be  a  real symmetric operator
(under Dirichlet boundary condition or Neumann boundary condition) with nonnegative
 $V\in L^1_{\rm loc}(\Omega)$ as in (\ref{e1.15}).
 Let $\sigma(\mathcal{L})$ be the lowest nonpositive eigenvalue of $\mathcal{L}$ on $\Omega$. Fix $\tau>1$.
 For every ball $B=B(x_B, r_B)$ in $\Omega$,
 define
$$
\Lambda(B,V)=\int_B V(x)\left(\log \left(e+\f{V(x)}{V_B}\right)\right)^{\tau} d\mu(x).
$$
Then there exist two positive constants $c_1$ and  $c_2$  which depend on $n$, $\tau$ such that
$$
\sigma(\mathcal{L}) \geq -\sup_{B}
c_1\left[|B|^{-1}\Lambda(B,V)-c_2   r_B^{-2}|B|\right],
$$
where $|B|$ denotes the Lebesgue measure of $B$ in $\Omega$.

As a consequence, if
\[  r_B^2 \Lambda(B,V) \leq c_2 |B|^2 \]
for all balls $B$ in $\Omega$,
 then $\mathcal{L}$ is nonnegative.
 }

 \smallskip

 We mention that in the Euclidean spaces ${\mathbb R}^n$,
  such operators have received a great deal of attention in the
past decades.  In particular, the case $A=I$, the identity, was studied, see for example,
 in \cite{CW, CWW, Fe, KS, Pe} and the references therein; while
the case of nonconstant $A$ was treated in \cite{Wi87, Wi}.

The layout of the article is as follows. In Section 2 we recall some basic properties of space of homogeneous type,
heat kernels and finite propagation speed for the wave equation, and build the necessary kernel estimates
for functions of an operator. In Section 3 we introduce  the discrete square function
on space of homogeneous type.
   In Section 4 we give a proof of Theorem A  by using reduction to  a version of  discrete  square function
on  space   of homogeneous type.   Theorems B, C and D  will be  proved  in
 Sections 5 and 6, which include
 estimates of the $L^p$ operator norms   as $p \rightarrow \infty$ for operators such as
 the square functions or spectral
 multipliers,   weighted norm inequalities for the square functions and
  eigenvalue estimates for   Schr\"odinger operators on ${\mathbb R}^n$ or
 on a special Lipschitz domain  of ${\mathbb R}^n$.

\bigskip

\section{Notations and preliminaries on operators }
\setcounter{equation}{0}

We start by introducing  some   notations and assumptions.  Throughout this article,
unless we mention the contrary, $(X,d,\mu)$ is a metric measure  space, that is, $\mu$
is a Borel measure with respect to the topology defined by the metric $d$.
We denote by
$B(x,r)=\{y\in X,\, {d}(x,y)< r\}$  the open ball
with centre $x\in X$ and radius $r>0$. We   often just use $B$ instead of $B(x, r)$.
Given $\lambda>0$, we write $\lambda B$ for the $\lambda$-dilated ball
which is the ball with the same centre as $B$ and radius $\lambda r$.
We say that $(X, d, \mu)$ satisfies
 the doubling property (see \cite[Chapter 3]{CW71})
if there  exists a constant $C>0$ such that
\begin{eqnarray}
\mu(B(x,2r))\leq C \mu(B(x, r)),\quad \forall\,r>0,\,x\in X. \label{e2.1}
\end{eqnarray}
If this is the case, there exist  $C, n$ such that for all $\lambda\geq 1$ and $x\in X$
\begin{equation}
\mu(B(x, \lambda r))\leq C\lambda^n \mu(B(x,r)). \label{e2.2}
\end{equation}
In the Euclidean space with Lebesgue measure, $n$ corresponds to
the dimension of the space.

  Throughout the article,  we will suppose that $\mu(X)=\infty$ and $\mu(\{x\})=0$ for all $x\in X$.
  For $1\le p\le+\infty$, we denote the
norm of a function $f\in L^p(X,{\rm d}\mu)$ by $\|f\|_p$, by $\langle \cdot,\cdot \rangle$
the scalar product of $L^2(X, {\rm d}\mu)$, and if $T$ is a bounded linear operator from $
L^p(X, {\rm d}\mu)$ to $L^q(X, {\rm d}\mu)$, $1\le p, \, q\le+\infty$, we write $\|T\|_{p\to q} $ for
the  operator norm of $T$. Given a subset $E$, we denote by $\chi_E$ the characteristic function of $E.$
For every $f\in L^1_{\rm loc}(X)$ and a subset $E$, we write
$$ f_E=(f)_E=\frac{1}{\mu(E)}\int_E f(x)d\mu(x), \ \ \ {\rm and}\ \ \ \
|f|_E=\frac{1}{\mu(E)}\int_E |f(x)|d\mu(x).
$$
 For $F\in L^1({\mathbb R})$, the Fourier transform of $F$ is
given by
\begin{eqnarray*}
    {\widehat F} (t)
 =\frac{1}{2\pi}  \int_{-\infty}^{+\infty}
    F(\lambda)   e^{-it\lambda}\, d\lambda.
\end{eqnarray*}
Finally, $C$ and $c$ denote two generic constants, not necessarily the same at each occurrence, which, in the course of a
proof, may be taken to depend on any of the non-essential quantities assumed to be bounded.

\medskip

\medskip
\noindent{\bf 2.1. System of dyadic cubes.} \ In a metric space $(X, d)$, a countable family
$$
\mathscr{D}=\bigcup_{k\in \mathbb{Z}} \mathscr{D}_k,\qquad  \mathscr{D}_k=
\{Q_\beta^k:\
\beta\in \mathscr{J}_k\},
$$
of Borel sets $Q_\beta^k\subset X$  is called {\it a system of dyadic cubes} with parameter
$\delta\in (0, 1)$ if it satisfies the following properties:
\begin{enumerate}
	\item [(i)] $X=\bigcup_{\beta\in \mathscr{J}_k} Q_\beta^k$
	(disjoint union) for
	all $k\in \mathbb{Z}$;
	\item [(ii)] if $\ell\geq k$, then either $Q_{\beta_2}^\ell\subseteq
	Q_{\beta_1}^k$ or $Q_{\beta_1}^k\cap Q_{\beta_2}^\ell=\emptyset$;
	\item [(iii)] for each $(k,\beta_1)$ and each $\ell\leq k$,
	there exists a unique $\beta_2$
	such that $Q_{\beta_1}^k\subseteq Q_{\beta_2}^\ell$;
	\item [(iv)] there exists a positive constant
	$M=M(n,\delta)$
	such that for each $(k,\beta_1)$,
\begin{eqnarray}\label{e2.3}
	1\leq \# \left\{\beta_2:\  Q_{\beta_2}^{k+1}\subseteq Q_{\beta_1}^k\right\}\leq M  \quad
	\text{and }\quad
	Q_{\beta_1}^k= \underset{ \beta_2:\  Q_{\beta_2}^{k+1}\subseteq Q_{\beta_1}^k
		 }{\bigcup} Q_{\beta_2}^{k+1} \ ;
\end{eqnarray}
	\item [(v)] for each $(k,\beta)$,
    \begin{equation}
    \label{eqn:dycube1}
    B(x_\beta^k,\delta^k/3)\subseteq Q_\beta^k\subseteq
	B(x_\beta^k,2\delta^k)=:B(Q_\beta^k)
    \end{equation}
     for  some $x_\beta^k\in Q_\beta^k$;
	\item [(vi)]if $\ell\geq k$ and $Q_{\beta_2}^\ell\subseteq
Q_{\beta_1}^k$, then
	$B(Q_{\beta_2}^\ell)\subset B(Q_{\beta_1}^k)$.
\end{enumerate}

The set $Q_\beta^k$ is called a dyadic cube of generation $k$ with
center point
 $x_\beta^k\in Q_\beta^k$. For $Q\in \mathscr{D}_{k}$, for every $\ell\geq 0$
we denote by
$$Ch^{(\ell)}(Q)=\left\{P\subset Q:\ P\in
\mathscr{D}_{k+\ell}\right\}.
$$
Obviously, $Ch^{(0)}(Q)=\{Q\}$ and $Ch^{(1)}(Q)=Ch(Q)$, which is called the children of $Q$.

Next, we recall the following construction due to T. Hyt\"onen and A. Kairema~\cite{HK}, which is a slight
elaboration
of
seminal work by M. Christ \cite{Ch}, as well as E. Sawyer and R.L. Wheeden \cite{SW}.

\begin{proposition}\label{prop2.1}
Let $(X, d)$ be  a metric space. Then there exists a system of dyadic cubes with parameter $0<\delta<1/144$.
	The construction only depends on a
fixed set of countably many points
$x_\beta^k$, having the properties that
	\begin{equation}
	\label{e2.4}
	d(x_{\beta_1}^k,x_{\beta_2}^k)\geq \delta^k \ \  (\beta_1\neq \beta_2), \qquad
\underset{\beta}{\min}
	\ d(x,x_\beta^k)<\delta^k \  \  \text{for all } \ x\in X,
	\end{equation}
and  a partial order ``$\leq $'' among the pairs $(k,\beta)$ such that
\begin{enumerate}
	\item [$\bullet$]if $d(x_{\beta_2}^{k+1},x_{\beta_1}^k)<\delta^k/6$, then $(k+1,
\beta_2)\leq (k,\beta_1)$;
	\item [$\bullet$]if $(k+1,\beta_2)\leq (k,\beta_1)$, then $d(x_{\beta_2}^{k+1},x_{\beta_1}^k)<
2\delta^k$;
 \item [$\bullet$]for every $(k+1,\beta_2)$, there is exactly one $(k,\beta_1)\geq (k+1,\beta_2)$, called its parent;
  \item [$\bullet$]for every $(k,\beta_1)$, there are between $1$ and $M$ pairs $(k+1, \beta_2)\leq (k, \beta_1)$, called
  its children.
\end{enumerate}
	\end{proposition}

 We will refer the sets $\{x_\beta^k\}_{k, \beta\in{\mathscr J}_k}$ of dyadic points as the set of {\it reference points}.
Now fix a set of reference points $\{x_\beta^k:\  k\in{\mathbb Z}, \beta\in{\mathscr J}_k \}$ and
 some   $0<\delta<1/144$.
  From \cite[Section 4]{HK}, reference points $x_{\beta_1}^k$ and $x_{\beta_2}^k$,
$\beta_1\neq \beta_2$,
	of the same generation are  {\it in conflict} if
	$ d(x_{\beta_1}^k, x_{\beta_2}^k)<\delta^{k-1}.$
  Reference points $x_{\beta_1}^k$ and $x_{\beta_2}^k$, $\beta_1\neq \beta_2$,
  of  the same
generation
	are { \it neighbors} if there occurs a conflict between their children.
It is known that
for every reference point $x_\beta^k$, the number of neighbors of $x_\beta^k$ is bounded from above by a fixed
constant, say $N$.
Then every point
$x_{\beta_1}^k$
gets a primary label
\begin{equation}
\label{e4.1}
{\bm{label_1}}(k,\beta_1):=\ell
\end{equation}
not bigger than $N=N(n,\delta)$, such that  if $(k,\beta_1)$ and $(k,\beta_2)$, $\beta_1
\neq \beta_2$, have the
same label
$\ell\in \{0,1,\cdots, N\}$, they are not neighbors.
Next we label the reference points $x_\beta^{k+1}$ of the next generation
$k+1$ with
duplex labels: If ${\bm{label_1}}(k,\beta_1)=\ell$, each of its children
$(k+1,\beta_2)
\leq (k,\beta_1)$
gets a
different duplex label
\begin{equation}
\label{e4.2}
{\bm{label_2}}(k+1,\beta_2):=(\ell,m), \  \  \  \  m(\beta_2)\in \{1,2,\cdots, M\}.
\end{equation}
Then adjacent dyadic systems are constructed by the following specific
{\it selection rule}:
 Fix $(\ell,m)\in \{0,\cdots, N\}$
 $\times \{1,
\cdots, M\}$. For
	every index pair $(k,\beta_1)$, check whether there exists $(k+1,\beta_2)\leq
(k,\beta_1)$ with
	label pair $(\ell,m)$. If so, decree that $z_{\beta_1}^k:=x_{\beta_2}^{k+1}$.
Otherwise, pick some
	$(k+1,\beta_2)\leq (k,\beta_1)$ with $\rho(x_{\beta_2}^{k+1},x_{\beta_1}^k)<\delta^{k+1}$
and
	decree $z_{\beta_1}^k:=x_{\beta_2}^{k+1}$.
Define
\begin{eqnarray}
\label{e2.7a}
\Lambda: \{0,\cdots, N\}\times \{1,\cdots, M\}&\to& \{1,\cdots, K\}\subset
\mathbb{N} \notag\\
(\ell,m) &\mapsto& b
\end{eqnarray}
which is a bijection. We identify $b=\Lambda(\ell,m)$ with $(\ell,m)$. Each $b$ gives rise
to a set $ \Big\{{^b z_\beta^k}:\ k\in \mathbb{Z},\ \beta\in \mathscr{J}_k \Big\}$ of new
dyadic
points associated with the duplex label $(\ell,m)=b$.
 From these new dyadic points $\left\{{^b z_\beta^k}:\ k\in \mathbb{Z},\ \beta\in \mathscr{J}_k\right\}$,
  we can construct adjacent
dyadic systems
\begin{eqnarray}
\label{e4.33b}
\left\{\mathcal{D}^{b}: b=1,2,\cdots, K\right\},
\end{eqnarray}
where $\mathcal{D}^b=\bigcup_{k\in \mathbb{Z}}\mathcal{D}_k^b$,
  $\mathcal{D}_k^b=\{{^b}Q_\beta^k:\  \beta\in \mathscr{J}_k\}$, and
\begin{equation}
\label{eqn:dycube2}
B(^b z_\beta^k, \delta^k/12)\subset {^b Q_\beta^k} \subset B({^b z_\beta^k},4\delta^k).
\end{equation}
For more details, we refer the reader to \cite[Section 4]{HK}, and in the sequel, we assume $0<\delta<1/144$.

Throughout the article,   if $^bQ_{\beta_1}^k\in \
\mathcal{D}^b$ has only one child, namely $^bQ_{\beta_2}^{k+1}$, although $^bQ_{\beta_1}^k$
 and $^bQ_{\beta_2}^{k+1}$ are the same set of points,  $^bQ_{\beta_1}^k$ and
$^bQ_{\beta_2}^{k+1}$ are   seen  as  two  different cubes. Then we have the following proposition.

\begin{proposition}
\label{prop4.2}
There exists an injective mapping
$$E_2:\  \mathscr{D}\to  \mathcal{D}=\bigcup_{b=1}^K \mathcal{D}^b,
$$
such that for every $Q_\beta^k\in \mathscr{D}$,
 $E_2(Q_\beta^k)\in \mathcal{D}_{k-1}^{b}$ for some
$b\in \{1,2,\cdots, K\}$ with the property:
\[
3Q_\beta^k:=\left\{y\in X:\   \
d(y,Q_\beta^k)\leq \delta^k\right\}\subset E_2(Q_\beta^k),\  \   \  \mu(E_2(Q_\beta^k))\leq C\mu(Q_\beta^k)
\]
for some  constant $C=C(n)$ independent of $(\beta,k)$.
\end{proposition}

\begin{proof}
 Observe that for every $Q_{\beta_1}^k\in \mathscr{D}$,  $(k,\beta_1)$ has
 a unique
${\bm{label_2}}(k,\beta_1)=(\ell,m)$,
and there exists a unique  $\beta_2\in \mathscr{J}_{k-1}$
such that $(k,\beta_1)
\leq (k-1,\beta_2)$. We  set $b={\Lambda({\bm{label_2}}(k,\beta_1))}$, where $\Lambda$ is the mapping in \eqref{e2.7a}.
According to the selection rule,
there exists ${^b Q_{\beta_2}^{k-1}}\in \mathcal{D}^{b}$ with center point ${^b z_{\beta_2}^{k-1}}=x_{\beta_1}^k$. Define
\begin{eqnarray}
\label{e4.4}
E_2: \mathscr{D} &\to& \bigcup_{b=1}^K \mathcal{D}^{b}=\bigcup_{b=1}^K
\bigcup_{k\in \mathbb{Z}} \mathcal{D}_k^b \notag \\
Q_{\beta_1}^k  &\mapsto&   {^b}Q_{\beta_2}^{k-1},\ \ \  b={\Lambda({\bm{label_2}}(k,\beta_1))},
 \  \   (k,\beta_1)\leq (k-1,\beta_2).
\end{eqnarray}

Obviously,  for every $Q_{\beta_1}^k\in \mathscr{D}$ in level $k$ with center $x_{\beta_1}^k$,
there exists a unique
$b\in \{1,2,\cdots, K\}$ such that $E_2(Q_{\beta_1}^k)={^b Q_{\beta_2}^{k-1}}\in \mathcal{D}_{k-1}^{b}$.
It follows from \eqref{eqn:dycube2} that
$$
B(x_{\beta_1}^k,\delta^{k-1}/12)\subseteq E_2(Q_{\beta_1}^k)={^b Q_{\beta_2}^{k-1}}\subseteq
	B(x_{\beta_1}^k,4\delta^{k-1}),
$$
and one can apply \eqref{eqn:dycube1} and $\delta<1/144$ to deduce
\[
B(x_{\beta_1}^k, \delta^k/3)\subset Q_{\beta_1}^k \subset 3Q_{\beta_1}^k=\left\{y\in X:\   \
d(y,Q_{\beta_1}^k)\leq \delta^k\right\}\subset B(x_{\beta_1}^k, 3\delta^k)\subset B(x_{\beta_1}^k ,\delta^{k-1}/12).
\]
The doubling property \eqref{e2.1} implies that there exists a constant $C=C(n)$
such that $\mu(E_2(Q_{\beta_1}^k))\leq C\mu(Q_{\beta_1}^k)$ for every $Q_{\beta_1}^k\in \mathscr{D}$.

Next,  we show that  $E_2$ is injective.
Indeed, for any  $Q_{\beta_3}^k,Q_{\beta_4}^k\in \mathscr{D}$ with
  $\beta_3\neq \beta_4$,
 then  $E_2(Q_{\beta_3}^k)\neq E_2(Q_{\beta_4}^k)$ since their center points are
 different. For any $Q_{\beta_5}^k,Q_{\beta_6}^\ell \in \mathscr{D}$
 with $k\neq \ell$, even though $E_2(Q_{\beta_5}^k)$ and $E_2(Q_{\beta_6}^\ell)$
  are both in some $\mathcal{D}^{b}$, $b\in \{1,2,\cdots, K\}$, one can
see $E_2(Q_{\beta_5}^k)\in \mathcal{D}_{k-1}^b$, while
  $E_2(Q_{\beta_6}^\ell)\in \mathcal{D}_{\ell-1}^b $.
Hence,    the proof of Proposition~\ref{prop4.2} is end.
\end{proof}

\medskip

We would like to mention that in the case when $X=\mathbb{R}^n$,
a cube $Q\subset {\mathbb R}^n$ is dyadic if it is of the form
$$
  Q=\left({j_1\over 2^k}, {j_1+1\over 2^k}\right) \times \cdots \times
\left({j_n\over 2^k}, {j_n+1\over 2^k}\right)
$$
for some integers $k$ and $j_i, i=1, \cdots, n.$ We denote   the collection of all dyadic cubes by $\mathscr{D}$.
Let $\mathcal{D}$ be the collection of all triples of dyadic cubes.
It is known (\cite[Lemma~2.1]{Wi}) that there exist disjoint families
$\mathcal{D}^1,\cdots, \mathcal{D}^{3^n}$ such that $\mathcal{D}=
\bigcup_{b=1}^{3^n} \mathcal{D}^b$. In this case,
one can define  the mapping $E_2: \mathscr{D}   \to  \bigcup_{b=1}^{3^n}\mathcal{D}^b$ in \eqref{e4.4} as follows:
\begin{eqnarray}\label{e4.5}
E_2(Q)=    3Q
\end{eqnarray}
for every $Q\in \mathscr{D}$. Here $3Q$ denotes the cube concentric with $Q$ while with sidelength three times as big.


\medskip
\noindent{\bf 2.2.\ Finite propagation speed property for wave equation.}
Suppose that $L$ is a densely-defined operator on $L^2(X)$.
It is known (see for instance \cite{CGT, CS, Si04} that if $L$ satisfies
\textbf{(H1)} and \textbf{(H2)}, then $L$ satisfies the finite propagation speed
property, that is, there exists a finite, positive constant
$c_0$ with the property that the Schwartz kernel $K_{\cos(t\sqrt{L})}$ of
$\cos(t\sqrt{L})$ satisfies
$$\leqno{\rm (FS)} \hspace{3cm}
\text{supp }K_{\cos(t\sqrt{L})} \subset \{(x,y)\in X\times X:\ d(x,y)
\leq c_0 t\}.
$$
The precise value of $c_0$ is non-essential, and throughout the article we will
choose $c_0=1$. By the Fourier inversion formula, whenever $F$ is an even bounded
Borel function with $\widehat{F}\in L^1(\mathbb{R})$, we can write $F(\sqrt{L})$ in
terms of $\cos(t\sqrt{L})$. Concretely, we have
\[F(\sqrt{L})=(2\pi)^{-1} \int_{-\infty}^\infty \widehat{F}(t)\cos(t\sqrt{L})dt,\]
which, combined with condition  (FS), gives
\begin{equation}
\label{e2.5}
K_{F(\sqrt{L})}(x,y)=(2\pi)^{-1} \int_{d(x,y)\leq |t|}
\widehat{F}(t)K_{\cos(t\sqrt{L})}(x,y)dt.
\end{equation}

\smallskip

\begin{lemma}
	\label{le2.2}
	Suppose that $L$ is a
densely-defined operator on $L^2(X)$ satisfying \textbf{(H1)} and \textbf{(H2)}.
Let $\rho>0$, $\phi\in C_0^\infty(\mathbb{R})$ be even, ${\rm{supp }}\
	\phi \in (-\rho,\rho)$ and $\varphi={\widehat\phi}$. For every $t>0$ and
	$k=0,1,2,\cdots$, $\psi_{k}(t\sqrt{L})=(t^2 L)^k\varphi(t\sqrt{L})
	$
	is defined by the spectral
theory. Then for every
	$k=0,1,2,\cdots$, we have
	
	\begin{itemize}
\item[(i)]
The kernel $K_{ \psi_k(t\sqrt{L})}(x,y)$
of  $\psi_k(t\sqrt{L}) $   satisfies
$$
	{\rm{supp }}\  K_{\psi_k(t\sqrt{L})} \subset \big\{(x,y)\in X
\times X :\ d(x,y)\leq \rho t  \big\}
$$
and
$$
	\left|K_{  \psi_{k}(t\sqrt{L})}(x,y)\right| \leq \f{C(k,\rho)}{\mu(B(x,t))}, \ \ \ \forall t>0, \ \ \forall x, y\in X;
$$

	\item[(ii)]
The kernel $K_{\psi_k(s\sqrt{L})\psi_k(t\sqrt{L}) }(x,y)$ of $\psi_k(s\sqrt{L})\psi_k(t\sqrt{L})$
satisfies
$$\big|K_{\psi_k(s\sqrt{L})\psi_k(t\sqrt{L})}(x,y)\big|
\leq   \min\left(\frac{s}{t}, \frac{t}{s}\right)\, \frac{C(k,n,\rho)}{ \mu(B(x,\max{(s,t)}))},
 \ \ \ \forall s, t>0, \ \ \forall x, y\in X;
$$

	\item[(iii)]
 In addition, we assume that $L$  satisfies condition   \textbf{(H3)} with some index $\theta\in (0, 1]$.
 Then the kernel $K_{ \psi_k(t\sqrt{L}) }(x,y)$ of $ \psi_k(t\sqrt{L})$
satisfies
\begin{eqnarray}\label{e2.6}
\big|K_{ \psi_k(t\sqrt{L})}(x,y)- K_{ \psi_k(t\sqrt{L})}(x',y)\big|
\leq   C(k,n,\rho) \min \left\{ \Big(\f{d(x,x')}{ {t}}\Big)^\theta, 1\right\}
\f{1}{\mu(B(y, {t}))}
 \end{eqnarray}
 for all $x, x', y\in X$ and all $t>0$.
\end{itemize}
\end{lemma}

\begin{proof}  	The proof of (i) is standard; see for example, \cite{HLMMY, Si04}.
The proof of (ii) can be obtained by a minor modification with that of \cite[Lemma 2.3]{SY16}, and we skip it here.

We now prove (iii). For every $m\geq  n/2+1$ and every $k\in{\mathbb N}$, we write $\eta_{m,k}(u)=(1+u^2)^m\psi_k(u)$.
Then the kernel $K_{ \eta_{m,k}(t\sqrt{L}) }(x,y)$ of $ \eta_{m,k}(t\sqrt{L})$
satisfies property (i) above.  On the other hand, we rewrite
$\psi_k(u) = (1+u^2)^{-m} \eta_{m,k}(u).$ Observe that
for any $m\in {\Bbb N}$, we have   the relationship
\begin{eqnarray}\label{mmmm}
  (I+ t^2L)^{-m}={1\over  (m-1)!} \int\limits_{0}^{\infty}e^{- st^2L}e^{-s} s^{m-1} ds
\end{eqnarray}

 To show \eqref{e2.6}, we consider two cases.

 \smallskip

 \noindent
 {\it Case (1): $d(x, x')\leq  t$.} \

  \smallskip
 In this case, we use   formula  \eqref{mmmm},
 together with property (i)   of this lemma, \eqref{e2.2} and  condition   \textbf{(H3)} to obtain
\begin{eqnarray*}
& &\hspace{-1.2cm}\left|K_{\psi_k(t\sqrt{L})} (x,y)-K_{\psi_k(t\sqrt{L})} (x',y)\right| \notag\\
&\leq& C  \int_{X} \left( \int_0^\infty e^{-s}s^{m-1}
\left|p_{t^2 s} (x,z)-p_{t^2 s}(x',z)\right|ds\right)\left|K_{(I+t^2 L)^m
	\psi_k(t\sqrt{L})} (z,y)\right|d\mu(z)\notag\\
&\leq & C(k,\rho) \int_{{d(z,y)\leq \rho t}} \f{1}{\mu(B(z,t))} \int_0^{\left(d(x,x')/t\right)^2}
 \f{1}{\mu(B(z,\sqrt{s}t))}
e^{-s}s^{m-1} dsd\mu(z)\notag\\
& & +\,  C(k,\rho) \int_{{d(z,y)\leq \rho t}} \f{1}{\mu(B(z,t))} \int_{\left(d(x,x')/t\right)^2}^\infty
\Big(\f{d(x,x')}{\sqrt{s}t}\Big)^\theta \f{1}{\mu(B(z,\sqrt{s}t))} e^{-s}
s^{m-1}dsd\mu(z) \\
  &=&I+II.
\end{eqnarray*}
For term $I$, we use \eqref{e2.2}  and elementary integration,
to verify that for $m\geq  n/2+1$,
\begin{eqnarray*}
I&\leq & C(k,\rho) \f{1}{\mu(B(y,t))} \int_{{d(z,y)\leq \rho t}} \f{1}{\mu(B(z, t))}d\mu(z)  \int_0^{\left(d(x,x')/t\right)^2}
s^{-n/2}
e^{-s}s^{m-1} ds\\
 &\leq&   C(k,n,\rho)  \left(\f{d(x,x')}{t}\right)^{2m-n}\f{1}
{\mu(B(y,t))}.
\end{eqnarray*}
Now for the term $II$,
\begin{eqnarray*}
II&\leq & C(k,\rho)\left(\f{d(x,x')}{t}\right)^{\theta} \f{1}{\mu(B(y,t))} \int_{{d(z,y)\leq \rho t}} \f{1}{\mu(B(z, t))} d\mu(z)
  \int_{\left(d(x,x')/t\right)^2}^\infty
e^{-s}s^{m-1} ds \\
&\leq & C(k,n,\rho)  \left(\f{d(x,x')}{t}\right)^{\theta}\f{1}
{\mu(B(y,t))},
\end{eqnarray*}
which yields  \eqref{e2.6} in this case $d(x, x')\leq  t$.

 \smallskip

 \noindent
 {\it Case (2): $d(x, x')> t$.} \

 \smallskip

In this case we  apply  property (i)   of this lemma  to see that
\begin{eqnarray*}
\left|K_{\psi_k(t\sqrt{L})} (x,y)-K_{\psi_k(t\sqrt{L})} (x',y)\right|
&\leq& |K_{\psi_k(t\sqrt{L})} (x,y)|+|K_{\psi_k(t\sqrt{L})} (x',y)|\\
&\leq& \f{C(k,\rho)}{\mu(B(y,t))},
\end{eqnarray*}
which gives our desired estimate.  This proves (iii) and completes the proof of Lemma~\ref{le2.2}.
 \end{proof}

 \medskip

In the end of this section, we  mention that our assumptions
\textbf{(H1)}, \textbf{(H2)} and \textbf{(H3)} on $L$
 hold for large variety of second order self-adjoint operators.
We list some of them:

1) 
 Let $A(x)$ be an $n\times n$
matrix function with real symmetric, bounded measurable  entries  on
${\mathbb R}^n$ satisfying the ellipticity condition
\begin{eqnarray*}
\|A\|_{\infty}\leq \lambda^{-1}\ \ \ {\rm and}\ \ \
\  A(x)\xi\cdot\xi\geq \lambda |\xi|^2
\end{eqnarray*}
for some constant $\lambda\in (0,1)$,  for all $\xi\in{\mathbb R}^n$
and for almost all $x \in \mathbb{R}^n$. Let
$L$ be a second order elliptic operator
in divergence form $L=-{\rm div} (A\nabla)$ on $L^2({\mathbb R}^n).$
Then properties  \textbf{(H1)}, \textbf{(H2)} and \textbf{(H3)}
 for operator $L$  are always satisfied, see \cite{AR, Da,  Ou}.

\medskip

2)   Let $L=-\Delta+V$ be some Schr\"odinger operator on
$\mathbb{R}^n$, $n\geq 3$, where $V$ is a nonnegative potential and  belongs to  the reverse H\"older
class $(RH)_{q}$ for some $q\geq n/2$, i.e.,
\begin{equation*}
\left(\frac{1}{|B(x, r)|}\int_{B(x, r)} V(y)^{q}~dy\right)^{1/q}\leq
\frac{C}{|B(x, r)|}\int_{B(x, r)}V(y)~dy,
\end{equation*}
where $|B(x,r)|$ denotes the Lebesgue measure of ball $B(x,r)$.
It's well known that $L$ satisfies \textbf{(H1)}, \textbf{(H2)} and \textbf{(H3)}, see \cite{DZ03,    Sh95}.

\medskip

3)\ Consider a complete Riemannian
manifold  $M$ with the
doubling property. One defines $\Delta$, the Laplace-Beltrami operator, as a
self-adjoint positive operator on $L^2(M)$.
If $M$ satisfies
suitable geometric
conditions, e.g. $M$ has nonnegative Ricci curvature,
 then $\Delta$ satisfies properties  \textbf{(H1)}, \textbf{(H2)} and \textbf{(H3)}, see \cite{ACDH, LY, Ou}.

 \bigskip

\section{The discrete square function}
\setcounter{equation}{0}

In this section, we assume that  $X$ satisfies the volume doubling condition, and $L$ is a
densely-defined operator
on $L^2(X)$
satisfying \textbf{(H1)}, \textbf{(H2)} and \textbf{(H3)}.
For a fixed   $\rho>0$,  we  let  $\phi\in C_0^\infty(\mathbb{R})$ be real-valued and even, ${\rm{supp }}\
	\phi \in (-\rho,\rho)$, and set
 $\psi(x):=x^2\big(\widehat\phi\big)(x),\
x\in \mathbb{R}.
 $  By the spectral theory (see for example \cite{Yo}), for every
$f\in L^2(X)$, we have
the following Calder\'on reproducing formula
\begin{eqnarray*}
f  &=&c_{\psi} \lim_{\substack{\varepsilon\to 0+ \\  N\to \infty}}
 \int_{\varepsilon}^N  \psi(t\sqrt{L})
 t^2L e^{-t \sqrt{L}}f \f{dt}{t}
\end{eqnarray*}
with the limit converging in $L^2(X)$.

 Fix
$\delta\in (0, 1/144)$   as in Proposition~\ref{prop2.1}. There exists  a system of dyadic cubes  {with parameter $\delta$}
such that
$$
\mathscr{D}=\bigcup_{k\in \mathbb{Z}} \mathscr{D}_k,\qquad  \mathscr{D}_k=
\left\{Q_\beta^k:\
\beta\in \mathscr{J}_k\right\}.
$$
For every  $Q_\beta^k\in \mathscr{D}_k,\ k\in \mathbb{Z},\ \beta\in
\mathscr{J}_k$,
define
\begin{equation*}
T(Q_\beta^k) :=\left\{(x,t)\in X\times (0, \infty):\ \ x\in Q_\beta^k,\
 {\rho^{-1}\delta^{k+1} \leq t <\rho^{-1}\delta^{k}} \right\},
\end{equation*}
{where $\rho>0$ is the parameter  appeared in the support of $\phi$.}
We obtain a decomposition for $X\times (0, \infty)$ as follows:
$$
X\times (0, \infty)=\bigcup_{k\in \mathbb{Z}} \bigcup_{\beta\in \mathscr{J}_k}
  T(Q_\beta^k).
$$
Then  $T(Q_{\beta_1}^k)\cap T(Q_{\beta_2}^\ell)=\emptyset$
for arbitrary $Q_{\beta_1}^k,Q_{\beta_2}^\ell\in
\mathscr{D}$ with $\beta_1\neq \beta_2$ or $k\neq \ell$.
Therefore, for every $f\in L^2(X)$,  we use A. Calder\'on trick (\cite{C}) to write
\begin{eqnarray}\label{e3.1}
f(x)&=& c_{\psi}  \iint_{X\times (0, \infty)} K_{\psi(t\sqrt{L})}(x,y)
\left( t^2L e^{-t \sqrt{L}}f(y) \right) \f{d\mu(y)dt}{t} \notag \nonumber\\
&=&c_{\psi} \sum_{k\in \mathbb{Z}}\sum_{\beta\in \mathscr{J}_k}
 \iint_{T(Q_\beta^k)}
K_{\psi(t\sqrt{L})} (x,y) \left(t^2Le^{-t \sqrt{L}}f(y)\right)\f{d\mu(y)dt}{t}\nonumber\\
&=:&\sum_{k\in \mathbb{Z}}\sum_{\beta\in \mathscr{J}_k}  a_{Q_\beta^k}(f)(x),
\end{eqnarray}
where
\begin{equation}
\label{e3.3}
a_{Q_\beta^k}(f)(x)=c_\psi  \iint_{T(Q_\beta^k)}
K_{\psi(t\sqrt{L})} (x,y) \left(t^2Le^{-t \sqrt{L}}f(y)\right)\f{d\mu(y)dt}{t}.
\end{equation}
Set
\begin{equation}
\label{e3.2}
\lambda_{Q_\beta^k} =  \left( \iint_{T(Q_\beta^k)} \left|t^2
Le^{-t \sqrt{L}}f(y)
\right|^2 \f{d\mu(y)dt}{t}\right)^{1/2}.
\end{equation}
We have the following result.

\begin{lemma}
\label{le3.1}
Suppose  that $L$ is a
densely-defined operator on $L^2(X)$ satisfying conditions \textbf{(H1)}, \textbf{(H2)} and \textbf{(H3)}
with some  $\theta\in (0, 1]$.
Then for every $k\in {\mathbb Z}$ and $ \beta_1\in \mathscr{J}_k$, the function
  $a_{Q_{\beta_1}^k}(f)$  satisfies
the following properties:
\begin{enumerate}[(i)]
  \item $\displaystyle {\rm{supp }}\  a_{Q_{\beta_1}^k}(f)
  \subset 3Q_{\beta_1}^k
  \subset   B(x_{\beta_1}^k,3\delta^k)\subset E_2(Q_{\beta_1}^k)$;
  \vspace{1ex}

  \item $\|a_{Q_{\beta_1}^k}(f)\|_\infty \leq C \lambda_{Q_{\beta_1}^k}
  \mu(Q_{\beta_1}^k)^{-1/2}$; \vspace{1ex}

  \item $\displaystyle \left|a_{Q_{\beta_1}^k}(f)(x)-a_{Q_{\beta_1}^k}(f)(x')\right|
  \leq  C  \lambda_{Q_{\beta_1}^k} \min\left\{\Big(\f{d(x,x')}{ \delta^k} \Big)^\theta,1
    \right\}   \mu(Q_{\beta_1}^k)^{-1/2}$ for  all\, $x,x'\in X$; \vspace{1ex}

  \item For every $\ell \leq k$, if $3Q_{\beta_1}^k\cap
  3Q_{\beta_2}^\ell \neq \emptyset$, then
$$
\left|\int a_{Q_{\beta_1}^k}(f)(x)\overline{a_{Q_{\beta_2}^\ell}(f)(x)}d\mu(x)\right|
\leq C \lambda_{Q_{\beta_1}^k}\lambda_{Q_{\beta_2}^\ell}\delta^{k-\ell}
\left(\f{\mu(Q_{\beta_1}^k)}{\mu(Q_{\beta_2}^\ell)}\right)^{1/2}.
$$
\end{enumerate}
\end{lemma}

\begin{proof}
The result of (i) and (ii) is a
 straightforward application of (i)   of Lemma~\ref{le2.2}.

  For every $x,x'\in X$,
\begin{eqnarray*}
 \left|a_{Q_{\beta_1}^k}(f)(x)-a_{Q_{\beta_1}^k}(f)(x')\right|
&\leq&  C \iint_{T(Q_{\beta_1}^k)}
\left|K_{\psi(t\sqrt{L})} (x,y)-K_{\psi(t\sqrt{L})} (x',y)\right|
 \big|t^2Le^{-t \sqrt{L}}f(y) \big|\f{d\mu(y)dt}{t}.
\end{eqnarray*}
From this, we  apply   (iii) of Lemma~\ref{le2.2}  and the Cauchy-Schwarz inequality to obtain  (iii).

Let us prove (iv).  For every $\ell\leq k$, we have that $\delta^k\leq \delta^{\ell}$.
Note that $\psi(t\sqrt L)$ is self-adjoint and so $K_{\psi(t\sqrt L)}(x,y)= {K_{\psi(t\sqrt L)}(y,x)}$. Thus
\begin{eqnarray*}
& &\hspace{-1.2cm}\left|\int  a_{Q_{\beta_1} ^k}(f)(x) \overline{a_{Q_{\beta_2}^\ell}(f)(x)} d\mu(x) \right|\\
&\leq&C  \iint_{T(Q_{\beta_1}^k)}  \left|t^2Le^{-t\sqrt{L}}
f(y)\right | \iint_{T(Q_{\beta_2}^\ell)} \left |K_{\psi(s\sqrt{L})
\psi(t\sqrt{L})}(y,z) \right |    \big |s^2L
e^{-s \sqrt{L}}f(z) \big| \f{d\mu(z)ds}{s}\f{d\mu(y)dt}{t}.
\end{eqnarray*}
By (ii) of Lemma~\ref{le2.2},
$
 |K_{\psi(s\sqrt{L})
\psi(t\sqrt{L})}(y,z)  | \leq C t/(s\mu(B(y,s))).
$
By the Cauchy-Schwarz inequality,
\begin{eqnarray*}
 \left|\int  a_{Q_{\beta_1}^k}(f)(x) \overline{a_{Q_{\beta_2}^\ell}(f)(x)} d\mu(x) \right|
&\leq&C   \lambda_{Q_{\beta_1}^k}\lambda_{Q_{\beta_2}^\ell}\delta^{k-\ell}
\Bigg(\f{\mu(Q_{\beta_1}^k)}{\mu(Q_{\beta_2}^\ell)}\Bigg)^{1/2},
\end{eqnarray*}
which implies  (iv).
The proof of Lemma~\ref{le3.1} is end.
\end{proof}

\smallskip

Let us now introduce the discrete square function associated to an operator $L.$
For a given $f\in L^2(X)$,   by \eqref{e3.1} we have that  $f=
\underset{Q\in \mathscr{D}}{\sum} a_Q(f)$ where $a_Q(f)$ is given in \eqref{e3.3}.
It follows from \eqref{e3.1} and Proposition~\ref{prop4.2} that
\begin{equation}
\label{e4.8}
f(x)=\sum_{b=1}^K   f_b(x), \  \    \text{where} \  \ f_{b}(x)=\sum_{Q\in \mathscr{D},
\  E_2(Q)\in \mathcal{D}^{b}}  a_Q(f)(x).
\end{equation}
Define  the {\it discrete square function associated to $L$}:
\begin{eqnarray}\label{e4.7}
{\mathscr A}_L (f )(x):=\left(\sum_{  Q\in \mathscr{D} }
\f{|\lambda_{Q }|^2}{\mu(Q )}\chi_{E_2(Q)}(x)\right)^{1/2},
\end{eqnarray}
where $\lambda_Q$ is given in \eqref{e3.2}.

\begin{lemma}\label{le4.3}
 Fix
$\delta\in (0, 1/144)$   as in Proposition~\ref{prop2.1} and let $\alpha>0$.
Let    $S_{L,\alpha}f$ and ${\mathscr A}_L (f )$ be given in \eqref{e1.7}  and \eqref{e4.7}, respectively.
Then there exists  $\rho=\rho(\alpha,\delta)$   in Lemma~\ref{le2.2}
 such that
\begin{equation}
\label{e4.9}
c^{-1} S_{L,\delta  \alpha /{75}}f(x)\leq {\mathscr A}_L (f )(x)\leq cS_{L, \alpha}f(x)
\ \ \  {\rm a.e.}\  x\in X
\end{equation}
for some  positive constant    $c=c(n,\alpha)$.

As a consequence,   we have
$$
\|f\|_p \simeq \|S_{L, \alpha}f\|_p\simeq  \|{\mathscr A}_Lf\|_p, \ \ \ \ 1<p<\infty.
 $$
\end{lemma}

\begin{proof}
Let  $\rho=\delta^2 \alpha/5$.
It follows from \eqref{e4.7}, \eqref{eqn:dycube1} and the doubling property \eqref{e2.1}   that
\[\mathscr{A}_L(f)(x)\  {\simeq } \ \left(\sum_{k\in \mathbb{Z}}{\int_{\rho^{-1}\delta^{k+1} }^{\rho^{-1}\delta^{k} }}
 \int_{\Omega_k(x)} \left|t^2 L e^{-t \sqrt{L}} f(y)\right|^2\frac{d\mu(y)}{\mu(B(x,t))}\frac{dt}{t}\right)^{1/2},\]
where $\displaystyle \Omega_k(x)=\bigcup_{Q\in \mathscr{D}_k ,\  x\in E_2(Q)}Q$.


For any $y\in \Omega_k(x)$, there exists a unique $Q\in \mathscr{D}_k$ with $x\in E_2(Q)$
such that $y\in Q$. Denote the center point of $Q$ by $x_Q$, one can apply
Proposition~\ref{prop4.2} and \eqref{eqn:dycube2} to see
$ E_2(Q)\subset B(x_Q,4\delta^{k-1})$, and so
  $d(x,x_Q)\leq 4\delta^{k-1}$.  Note
that $Q\subset B(x_Q,2\delta^k)$. By \eqref{eqn:dycube1}, we have
\[d(x,y)\leq d(x,x_Q)+d(y,x_Q)\leq 4\delta^{k-1}+2\delta^{k}<5\delta^{k-1}.\]
 {Recall that $\rho=\delta^2 \alpha/5$ so $[\rho^{-1}\delta^{k+1},\rho^{-1}\delta^k)= [5\delta^{k-1}/\alpha,5\delta^{k-2}/\alpha)$. Therefore,
for every $(y,t)\in \Omega_k(x)\times [\rho^{-1}\delta^{k+1},\rho^{-1}\delta^k)$, we have $d(x,y)<5\delta^{k-1}\leq \alpha t$ .
This, in combination with the definition of $S_{L,\alpha}(f)$, deduces that
${\mathscr A}_L (f )(x)\leq cS_{L, \alpha}(f)(x)$.}

On the other hand, recall that $\delta<1/{144}$ and $\rho=\delta^2 \alpha/5$.
 For each $z\in B(x,\delta  \alpha t/{75})$, there exists
a unique $P\in \mathscr{D}_k$ with center point $x_P$  such that $z\in P$. Then for $t\in [{5\delta^{k-1}/\alpha,5\delta^{k-2}/\alpha})$,
 $d(x,x_P)\leq d(z,x_P)+d(x,z)\leq 2\delta^k+ \delta\alpha t/{75}<\delta^{k-1}/{12}$, which implies  $x\in E_2(P)$,
   $z\in \Omega_k(x)$ and so $B\left(x,{\delta  \alpha t/{75}}\right)\subset \Omega_k(x)$
 , hence $S_{L,\delta \alpha/{75}}(f)(x)\leq c{\mathscr A}_L (f )(x).$
  The proof of Lemma~\ref{le4.3} is end.
\end{proof}

\medskip

\section{Exponential-square integrability}
\setcounter{equation}{0}
In this section  we aim to establish   the local exponential-square integrability of functions whose square functions
associated to an operator $L$ are bounded, and
extend a result of  Chang-Wilson-Wolff \cite[Theorem 3.2]{CWW}.
Based on  Lemma~\ref{le4.3},
the proof is using  reduction to the discrete  square function   associated to an operator $L$
on  space $X$ of homogeneous type.

Throughout this section  we fix a   $Q\in \mathcal{D}^b$ for some $b\in
\{1,2,\cdots,K\}$. Let $\mathcal{G}\subset \bigcup_{|k|\leq
K_{\mathcal{G}}}\mathcal{D}_k^{b}$  be a subset such that
$\bigcup_{P\in \mathscr{D}, E_2(P) \in \mathcal{G} } E_2(P) \subset Q$   and some $K_{\mathcal{G}}<\infty$.
  For every $f\in L^2(X)$,  it follows from \eqref{e3.1} that  $f =\sum_{P\in \mathscr{D}}  a_P(f)$.
For such $\mathcal{G}$, we define
\begin{eqnarray}\label{e4.10}
f_{\mathcal{G}}=\sum_{P\in \mathscr{D}, \  E_2(P)\in \mathcal{G}}   a_P(f)
\end{eqnarray}
and
\begin{eqnarray}\label{e4.66}
{\mathscr A}_L (f_\mathcal{G})(x)=\left(\sum_{P\in \mathscr{D}, \   E_2(P)\in \mathcal{G}}
\f{|\lambda_{P }|^2}{\mu(P )}\chi_{E_2(P)}(x)\right)^{1/2}
\end{eqnarray}
where  $\lambda_P$ is given in \eqref{e3.2}, and $E_2(P)$ is   as in Proposition~\ref{prop4.2}.
Obviously, ${\mathscr A}_L (f_\mathcal{G})\in L^{\infty}(X)$. Moreover, it can be verified  that
  there exists a  constant $C_1>0$, independent of $f_\mathcal{G}$ and $Q$ such that
\begin{equation}
\label{e4.13}
\displaystyle \left|
f_{\mathcal{G}}\right|_{Q}^2\leq \f{\|f_\mathcal{G}\|_2^2} {\mu({Q})}
\leq C_1^2  \|\mathscr A_L(f_\mathcal{G}) \|_{L^\infty({Q})}^2.
\end{equation}

Indeed, we assume that $Q\in \mathcal{D}_{k_0}^b$ for some $k_0\in \mathbb{Z}$.
 Since $E_2(P)\subseteq {Q}$, there exists
some $j\geq k_0$ such that
$E_2(P)\in \mathcal{D}_{j}^b$. Note that $\support a_P(f)\subset E_2(P)$.
Thus we have
\begin{eqnarray*}
\f{\|f_\mathcal{G}\|_2^2} {\mu({Q})}
&\leq &  \f{1}{\mu({Q})} \sum_{P\in \mathscr{D}, \  E_2(P)\in \mathcal{G}}
\int_{Q} \left|a_P(f)(x)\right|^2 d\mu(x) \\
& &+ \f{2}{\mu({Q})}     \sum_{j\geq k_0}
\sum_{\substack{P_1\in \mathscr{D}\\  E_2(P_1)\in \mathcal{G}\cap
\mathcal{D}_j^b}  }  \sum_{\ell>0}
\sum_{\substack{P_2\in \mathscr{D}, \  E_2(P_2)\in \mathcal{G}\\
E_2(P_2)\in Ch^{(\ell)} (E_2(P_1))}} \left| \int_{Q}  a_{P_1}(f)(x) \overline{a_{P_2}(f)(x)} d\mu(x)\right|.
\end{eqnarray*}

For $E_2(P_1)\in \mathcal{G}\cap \mathcal{D}_j^b $, and  $E_2(P_2)\in
Ch^{(\ell)}(E_2(P_1))$ mentioned above, we have $E_2(P_2)\in
\mathcal{D}_{j+\ell}^b$, and the definition of $E_2$ implies that
$P_1\in \mathscr{D}_{j+1}$ and $P_2\in \mathscr{D}_{j+\ell+1}$.
Hence it follows from (iv) of Lemma~\ref{le3.1} and Proposition~\ref{prop4.2}
that {there exists a constant $C=C(n,\rho)$ such that }
\begin{eqnarray*}
\left|\int_{Q}  a_{P_1}(f)(x) \overline{a_{P_2}(f)(x)} d\mu(x) \right|
\leq  C  \sqrt{\mu(E_2(P_2))\over \mu(E_2(P_1))}\delta^\ell
\left|\lambda_{P_1}\right| \left|\lambda_{P_2}\right|
\leq C \delta^\ell\Bigg( |\lambda_{P_2}|^2+\f{\mu(E_2(P_2))}{\mu(E_2(P_1))} |\lambda_{P_1} |^2
 \Bigg).
\end{eqnarray*}
This, together with   (ii) of Lemma~\ref{le3.1}, yields
\begin{eqnarray}
\label{e4.14}
\left|f_{\mathcal{G}}\right|_{Q}^2\leq \f{\|f_\mathcal{G}\|_2^2} {\mu({Q})}
&\leq & \f{C}{\mu({Q})}
\sum_{P\in
\mathscr{D},   \  E_2(P) \in \mathcal{G}}  |\lambda_P|^2
 +\f{C}{\mu(Q)}\sum_{P_2\in \mathscr{D}, E_2(P_2)\in \mathcal{G}}
 \sum_{\substack{\ell>0,  P_1\in \mathscr{D}, E_2(P_1)\in
 \mathcal{G}\notag\\ E_2(P_2)\in Ch^{(\ell)}(E_2(P_1))}}
 \delta^\ell \left|\lambda_{P_2}\right|^2 \notag \\
& +&   \f{C}{\mu({Q})}
\sum_{P_1\in \mathscr{D}, \  E_2(P_1)\in \mathcal{G}  }  \sum_{\ell>0}
\delta^\ell  |\lambda_{P_1}|^2
  \sum_{\substack{P_2\in \mathscr{D}, \  E_2(P_2)\in
 \mathcal{G}\notag\\
 E_2(P_2)\in Ch^{(\ell)} (E_2(P_1))}}   \f{\mu(E_2(P_2))}{\mu(E_2(P_1))}
  \notag\\
 &\leq & \f{C}{\mu({Q})}
\sum_{P\in \mathscr{D}, \  E_2(P)\in \mathcal{G}  }
\left|\lambda_{P}\right|^2\notag\\
&\leq& \f{C}{\mu({Q})}  \int_{Q}  \sum_{P\in \mathscr{D}, \
E_2(P)\in \mathcal{G}  }    \f{\left|\lambda_{P}\right|^2}
{\mu(P)}\chi_{E_2(P)}(x)d\mu(x) \notag\\
&\leq&   C_1^2 \left\|\mathscr A_L(f_\mathcal{G})\right\|_{L^\infty({Q})}^2
\end{eqnarray}
for some $C_1>0.$ This shows \eqref{e4.13} and this
  estimate
will be used often in the sequel.

 \subsection{John-Nirenberg type inequality for the discrete square function
 }

In this subsection, we will  give some preliminary result about
a  John-Nirenberg type inequality for the   function $f_{\mathcal{G}}$.
\begin{proposition}
\label{le4.4}
 Suppose that
$L$ is a densely-defined operator on $L^2(X)$ satisfying
\textbf{(H1)}, \textbf{(H2)} and \textbf{(H3)}.  Let $f\in L^2(X,d\mu)$. For a fixed  $Q\in \mathcal{D}^b$ for some $b\in
\{1,2,\cdots,K\}$, let $f_{\mathcal{G}}$  be given in \eqref{e4.10}.
  Then
	there exists two constants {$ c_1=c_1(n,\rho)$, $C_2=C_2(n,\rho)$}
 independent of $\mathcal{G}$, $K_\mathcal{G}$, $Q$ and $f$, with
$0<{c_1}<1<C_2<\infty$, such that for all $\lambda>0$,
	\begin{equation}	\label{e4.11}
	\mu\left(\left\{x\in {Q}:\     |f_\mathcal{G}(x)|>\lambda
\right\}\right)
  \leq   C_2 \exp\left(- \f{{c_1}\lambda}
  {\|{\mathscr A}_L(f_\mathcal{G})
  \|_{L^\infty({Q})}}\right) \mu({Q}).
	\end{equation}
	
	As a consequence, we have	
\begin{eqnarray}\label{e4.12}
 \int_{Q} \exp \left(\f{{c_1}|f_\mathcal{G}(x)|}
{2\|\mathscr A_L(f_\mathcal{G})\|_{L^\infty({Q})}}\right)
d\mu(x) \leq C_2\, \mu(Q).
\end{eqnarray}
\end{proposition}
\begin{proof} One  writes $E_\lambda=\{x\in {Q}:   \
|f_\mathcal{G}(x)|>\lambda\}$  and $\lambda_0={\mathscr N}\left\|\mathscr A_L
(f_\mathcal{G})\right\|_{L^\infty({Q})}$ for some constant ${\mathscr N}>C_1$, which will be chosen later.
Let us  consider   two cases.

\smallskip

\noindent
{\it Case (1):    $\lambda\geq  \lambda_0$}.

In this case, noting that by \eqref{e4.13}  that  $\left|f_{\mathcal{G}} \right|_{Q}\leq \lambda$, we can construct a subset $\{Q_{\lambda,i}:i\in \mathbb{N}\}$ of dyadic cubes $Q_{\lambda,i}\in \mathcal{D}^{b}$
by selecting all the maximum dyadic cubes  $Q_{\lambda,i}\subset Q$ in $\mathcal{D}^{b}$ such that:
 1)
  $Q_{\lambda,i}$ satisfies
$\left|f_{\mathcal{G}} \right|_{Q_{\lambda,i}}> \lambda$;
 2)
 $\widetilde{Q}_{\lambda,i}$ satisfies
$\left|f_{\mathcal{G}} \right|_{\widetilde{Q}_{\lambda,i}}\leq \lambda$, where $\widetilde{Q}_{\lambda,i}$
is the parent of $Q_{\lambda,i}$.
Moreover,  we  have the following properties:
\begin{enumerate}[(1)]
  \item[ 3)] Up to a subset of measure zero,
   \[E_\lambda \subset \bigcup_i
Q_{\lambda,i} =:\Omega_\lambda;\]
  \item[ 4)] If $\lambda_1\geq \lambda_2$, then $\Omega_{\lambda_1}\subset \Omega_{\lambda_2}$.
\end{enumerate}

Indeed,  we note that for every
$Q_{\lambda_1,i}\subset \Omega_{\lambda_1}$,
$\left|f_{\mathcal{G}} \right|_{Q_{\lambda_1,i}}> \lambda_1\geq\lambda_2$.
There is  some $Q_{\lambda_2,j}$ such that
$Q_{\lambda_1,i}\subset Q_{\lambda_2,j}\subset \Omega_{\lambda_2}$, and so   4) follows.
 For  3), denote the maximal operator
\[M^{\mathcal{D}^b}(f)(x):=\sup_{x\in P\in \mathcal{D}^b} \f{1}{\mu(P)}
\int_{P} |f(y)|d\mu(y).\]
It follows from \cite[Theorem~15.1]{LN15} that
 $M^{\mathcal{D}^b} $ is of weak
type $(1,1)$ uniformly in $b\in \{1,2,\cdots,K\}$. Therefore, a standard argument deduces the corresponding
Lebesgue differentiation theorem
\[\lim_{\substack{k\to +\infty\\ x\in P\in \mathcal{D}_k^b}} \f{1}{\mu(P)}
\int_P f(y)d\mu(y) =f(x) \  \  \   a.e.\  x\in X. \]
This shows
$
E_\lambda\subset \left\{x\in Q:\  \  M^{\mathcal{D}^b}\left(f_\mathcal{G}\right)(x)>\lambda\right\}\subset
\bigcup_i Q_{\lambda,i}
$
up to a subset of measure zero. This finishes the proof of  3) above.

\smallskip

Now, let us fix $i$. To  estimate $\big|(f_\mathcal{G})_{Q_{\lambda,i}}\big|$,
for $x\in Q_{\lambda,i}$ we  rewrite
$f_\mathcal{G}$ by
\[f_{\mathcal{G}}(x)=\sum_{\substack{P\in \mathscr{D},\, E_2(P)\in
\mathcal{G}\\ \widetilde{Q}_{\lambda,i}\subset E_2(P) }} a_P(f)(x)
+ \sum_{\substack{P\in \mathscr{D},\,
E_2(P)\in \mathcal{G}\\  E_2(P)\subset Q_{\lambda,i} }} a_P(f)(x) =:
f_{\mathcal{G},1}(x)+ f_{\mathcal{G},2}(x).
\]
Since $\mathcal{G}\subset \bigcup_{|k|\leq K_{\mathcal{G}}}\mathcal{D}_k^{b}$
with $K_{\mathcal{G}}<\infty$,     (iii) of Lemma~\ref{le3.1} implies
that $f_{\mathcal{G},1}$ is a finite sum of continuous functions and so is
continuous on $Q$ .
Therefore, for every $Q_{\lambda,i}$
there exists a point $x_i\in Q_{\lambda,i}$ such that
\[(f_{\mathcal{G},1})_{Q_{\lambda,i}}= f_{\mathcal{G},1} (x_i).\]
Similarly, there exists a point $\tilde{x}_i\in \widetilde{Q}_{\lambda,i}$
such that
$(f_{\mathcal{G},1})_{\widetilde{Q}_{\lambda,i}}= f_{\mathcal{G},1}
(\tilde{x}_i)$.
Denote that $\widetilde{Q}_{\lambda,i}\in \mathcal{D}_{\ell_0}^b$ and  $E_2(P)\in \mathcal{D}_{\ell_P}^b$.
Then one can apply (iii) of Lemma~\ref{le3.1}, Proposition~\ref{prop4.2}
and H\"older's inequality to obtain
\begin{eqnarray}
\left|f_{\mathcal{G},1} (x_i)- f_{\mathcal{G},1} (\tilde{x}_i)
\right| &\leq &
\sum_{\substack{P\in \mathscr{D},\, E_2(P)\in
\mathcal{G} \\ \widetilde{Q}_{\lambda,i}\subset E_2(P) }}
\left|a_P(f)(x_i)-a_P(f)(\tilde{x}_i) \right| \notag\\
&\leq& C \sum_{\substack{P\in \mathscr{D},\, E_2(P)\in
\mathcal{G} \\ \widetilde{Q}_{\lambda,i}\subset E_2(P) }}
\left(\f{d(x_i,\tilde{x}_i)}{\delta^{\ell_P+1}}\right)^{\theta}
 \f{|\lambda_P|}{
\mu(P)^{1/2} }\notag\\
&\leq &   C
\sum_{\substack{P\in \mathscr{D},\, E_2(P)\in
\mathcal{G} \\ \widetilde{Q}_{\lambda,i}\subset E_2(P) }} \delta^{-\theta}\left(\frac{\delta^{\ell_0}}
{\delta^{\ell_P}}\right)^{\theta} \frac{|\lambda_P|}{\mu(P)^{1/2}}
\notag\\
&\leq& C \delta^{-\theta}   \left(\sum_{\substack{P\in \mathscr{D}, \, E_2(P)\in
\mathcal{G} \\ \widetilde{Q}_{\lambda,i}\subset E_2(P) }}  \Bigg(\delta^{\ell_0-\ell_P}\Bigg)^{2\theta} \right)^{1/2}
\left(\sum_{\substack{P\in \mathscr{D},\, E_2(P)\in
\mathcal{G} \\ \widetilde{Q}_{\lambda,i}\subset E_2(P) }}
 \f{|\lambda_P|^2}{
\mu(P) }\right)^{1/2}\notag\\
&\leq &
 \f{C}{\delta^\theta \sqrt{1-\delta^{2\theta}}}  \left(
\sum_{\substack{P\in \mathscr{D},\, E_2(P)\in
\mathcal{G} \\  \widetilde{Q}_{\lambda,i}\subset E_2(P) }}  \f{|\lambda_P|^2}{\mu(P)}
\right)^{1/2},  \notag
\end{eqnarray}
{where constants $C$ above are dependent on $n$ and $\rho$.}
 Note that $\widetilde{Q}_{\lambda,i}\subset E_2(P) $ implies
$\chi_{E_2(P)}(x_i)\equiv 1$. Then we have
\begin{eqnarray}
\label{e4.15}
\left|f_{\mathcal{G},1} (x_i)- f_{\mathcal{G},1} (\tilde{x}_i)
\right| &\leq & \f{C}{\delta^\theta \sqrt{1-\delta^{2\theta}}}  \Bigg(
 \sum_{ P\in \mathscr{D},E_2(P)\in \mathcal{G}  }  \f{|\lambda_P|^2}{
 \mu(P)}\chi_{E_2(P)}(x_i)
\Bigg)^{1/2}  \notag\\
&\leq&   C_3\|\mathscr A_L(f_{\mathcal{G}})\|_{L^\infty({Q})}
\end{eqnarray}
for some $\displaystyle C_3=\f{C}{\delta^\theta \sqrt{1-\delta^{2\theta}}}$.

Note that $\support  f_{\mathcal{G},2} \subset Q_{\lambda,i} $. Thus
$$
\big|(f_{\mathcal{G},2})_{Q_{\lambda,i}} -(f_{\mathcal{G},2})_{
\widetilde{Q}_{\lambda,i}}\big|= \left| \f{1}{\mu(Q_{\lambda,i})}-\f{1}{
\mu(\widetilde{Q}_{\lambda,i})} \right|
\left| \int_{Q_{\lambda,i}}f_{\mathcal{G},2}(x)d\mu \right|\leq |f_{\mathcal{G},
2}|_{Q_{\lambda,i}}.
$$
Taking $Q_{\lambda,i}$ as $Q$ and $f_{\mathcal{G},2}$ as $f_{\mathcal{G}}$
in \eqref{e4.14},
we have
\begin{eqnarray}
\label{e4.16}
\big|(f_{\mathcal{G},2})_{Q_{\lambda,i}} -(f_{\mathcal{G},2})_{
\widetilde{Q}_{\lambda,i}}\big|   &\leq &  |f_{\mathcal{G},
2}|_{Q_{\lambda,i}}\leq C_1 \|\mathscr A_L(f_{\mathcal{G},2})
 \|_{L^\infty({Q_{\lambda,i}})}
\leq  C_1  \|\mathscr A_L(f_\mathcal{G}) \|_{L^\infty({Q})},
\end{eqnarray}
where $C_1$ is the constant in \eqref{e4.14}.

Therefore, when $\lambda\geq \lambda_0$,
\begin{eqnarray}
\label{e4.17}
\big|(f_\mathcal{G})_{Q_{\lambda,i}}\big|  &\leq &
\big|(f_\mathcal{G})_{Q_{\lambda,i}} - (f_\mathcal{G})_{
\widetilde{Q}_{\lambda,i}}\big| +\big|(f_\mathcal{G})_{
\widetilde{Q}_{\lambda,i}}\big|  \notag\\
&\leq & \big|(f_{\mathcal{G},1})_{Q_{\lambda,i}} - (f_{\mathcal{G},
1})_{\widetilde{Q}_{\lambda,i}}\big|  +\big|(f_{\mathcal{G},2})_{
Q_{\lambda,i}} - (f_{\mathcal{G},2})_{\widetilde{Q}_{\lambda,i}}
\big| +  \big|(f_\mathcal{G})_{\widetilde{Q}_{\lambda,i}}\big|
\notag\\
&\leq & C_4  \|\mathscr A_L
(f_\mathcal{G}) \|_{L^\infty({Q})}
 +\lambda
\end{eqnarray}
with $C_4=C_1+C_3.$

\smallskip

Now, take a parameter $\lambda'>0$ which will be chosen later. Similarly, let
$\{Q_{\lambda+\lambda',j}\}$,
$Q_{\lambda+\lambda',j}\in \mathcal{D}^{b}$, be the maximal dyadic cubes
such that
$|f_{\mathcal{G}} |_{Q_{\lambda+\lambda',j}}>\lambda+\lambda'$. Then for all $Q_{\lambda+\lambda',j}$, we have
$\mu(Q_{\lambda+\lambda',j})\leq \f{1}{\lambda+\lambda'}\int_{Q_{\lambda+\lambda',j}}
|f_{\mathcal{G}}| d\mu(x)$. From the selection rule of $Q_{\lambda+\lambda',j}$ and $Q_{\lambda,i}$,
we know that if $Q_{\lambda+\lambda',j}\cap Q_{\lambda,i}\neq \emptyset$
then $Q_{\lambda+\lambda',j}\subset Q_{\lambda,i}$.
For every  fixed $Q_{\lambda,i}$, by \eqref{e4.17}
\begin{eqnarray}\label{em1}
\mu\left(\Omega_{\lambda+\lambda'}\cap Q_{\lambda,i}\right)
&\leq & \f{1}{\lambda+\lambda'} \sum_{j:\  Q_{\lambda+\lambda',j}
\subset Q_{\lambda,i}}  \int_{Q_{\lambda+\lambda',j}}
|f_{\mathcal{G}}| d\mu(x)\nonumber\\
&\leq &  \f{1}{\lambda+\lambda'}  \int_{Q_{\lambda,i}} \left|f_{\mathcal{G}}
-(f_\mathcal{G})_{Q_{\lambda,i}}\right|  d\mu(x) +\f{ \left|
(f_\mathcal{G})_{Q_{\lambda,i}} \right|}{\lambda+\lambda'}
\mu(\Omega_{\lambda+\lambda'}\cap Q_{\lambda,i})\nonumber\\
&\leq & \f{1}{\lambda+\lambda'}  \int_{Q_{\lambda,i}} \left|f_{\mathcal{G},1}
-(f_{\mathcal{G},1})_{Q_{\lambda,i}}\right|d\mu(x)+ \f{2}{\lambda+
\lambda'}  \left|f_{\mathcal{G},2}\right|_{Q_{\lambda,i}} \mu(Q_{\lambda,i})\nonumber\\
& &+  \f{C_4  \|\mathscr A_L(
f_\mathcal{G}) \|_{L^\infty({Q})}
 +\lambda}{\lambda+\lambda'} \mu(\Omega_{\lambda+\lambda'}\cap Q_{\lambda,i}).
\end{eqnarray}
For $x\in Q_{\lambda,i}$, we apply the argument in  \eqref{e4.15} and \eqref{e4.16} to obtain
\begin{eqnarray}\label{em2}
\left|f_{\mathcal{G},1}(x)
-(f_{\mathcal{G},1})_{Q_{\lambda,i}}\right|\leq C_3 \|\mathscr A_L
(f_{\mathcal{G}})
\|_{L^\infty({Q})}
\end{eqnarray} and
\begin{eqnarray}\label{em3}
\left|f_{\mathcal{G},2}\right|_{Q_{\lambda,i}}
\leq C_1 \|\mathscr A_L(f_{\mathcal{G}})
\|_{L^\infty({Q})}.
\end{eqnarray}
 Putting    \eqref{em2} and \eqref{em3} into   \eqref{em1}, we have
\begin{eqnarray*}
\mu\left(\Omega_{\lambda+\lambda'}\cap Q_{\lambda,i}\right)
&\leq & \f{C_5\|\mathscr A_L(f_{\mathcal{G}})
\|_{L^\infty({Q})}}
{\lambda+\lambda'} \mu(Q_{\lambda,i})
 +\f{C_4 \|\mathscr A_L(f_\mathcal{G})
 \|_{L^\infty({Q})}
 +\lambda}{\lambda+\lambda'} \mu(\Omega_{\lambda+\lambda'}\cap Q_{\lambda,i})
\end{eqnarray*}
with $C_5=2C_1+C_3.$
Therefore,
\[ \mu(\Omega_{\lambda+\lambda'}\cap Q_{\lambda,i})  \leq \f{C_5\|\mathscr A_L(f_{\mathcal{G}})\|_{
L^\infty({Q})}}{
\lambda'-C_4\|\mathscr A_L(f_\mathcal{G})
\|_{L^\infty({Q})} }
 \mu(Q_{\lambda,i}).\]
This tells us that for every $\lambda\geq
\lambda_0$,
\[\mu(\Omega_{\lambda+\lambda'}) \leq \f{C_5
\|\mathscr A_L(f_{\mathcal{G}})\|_{L^\infty({Q})}}{\lambda'-
 C_4
\|\mathscr A_L(f_\mathcal{G})\|_{L^\infty({Q})} }
\mu(\Omega_{\lambda}).\]

Let
$\lambda'=\lambda_0$.
For $\lambda\geq \lambda_0$, there exists a unique
$k\geq 1$ such that $k\lambda_0 \leq   \lambda < (k+1)\lambda_0$. Therefore,
\begin{eqnarray*}
\mu(E_\lambda)&\leq& \mu(\Omega_\lambda) \leq \mu(\Omega_{k
\lambda_0})\\
&\leq&\f{C_5
\|\mathscr A_L(f_{\mathcal{G}})\|_{L^\infty({Q})}}{\lambda_0-
C_4
\|\mathscr A_L(f_\mathcal{G})\|_{L^\infty({Q})} }
\mu(\Omega_{(k-1)\lambda_0}).
\end{eqnarray*}
Now  choose the constant ${\mathscr N}=4C_4$, and then $\lambda_0=4C_4\|\mathscr A_L(f_\mathcal{G})\|_{L^\infty({Q})}$.
Iterating  $k-1$ times,  we notice that $k-1>\lambda/\lambda_0 \, -2$,
\begin{eqnarray}\label{eooo}
\mu(E_\lambda)&\leq & \exp\left\{-(k-1) \log \left(\f{\lambda_0-
C_4 \|\mathscr A_L(f_\mathcal{G})\|_{L^\infty({Q})} }
{C_5\|\mathscr A_L(f_{\mathcal{G}})\|_{L^\infty({Q})}}
\right) \right\}
\mu(\Omega_{\lambda_0}) \nonumber\\
&\leq& C_2 \exp\left(-\f{{c_1} \lambda}{\|\mathscr A_L(f_\mathcal{G})
\|_{L^\infty({Q})}}\right) \mu({Q}),
\end{eqnarray}
where
\[{c_1}=\f{1}{4C_4} \log\left(\f{3C_4}{C_5}
\right),\  \  C_2=\left(\f{3C_4}{C_5}
\right)^2=\exp(8C_4{c_1}).
\]

\smallskip

\noindent
{\it Case (2):    $\lambda< \lambda_0$.}

In this case,  we write
\begin{eqnarray*}
\mu(E_\lambda)&<&\exp\left(\f{c_1 \lambda_0}{\|\mathscr A_L(
f_\mathcal{G})\|_{L^\infty({Q})}}
\right) \exp\left(-\f{c_1 \lambda}{\|\mathscr A_L(
f_\mathcal{G})\|_{L^\infty({Q})}}
\right)\mu(Q)\\
&=&\sqrt{C_2} \exp\left(-\f{c_1 \lambda}{\|
\mathscr A_L(f_\mathcal{G})\|_{L^\infty({Q})}}
\right) \mu({Q}),
\end{eqnarray*}
which, in combination with estimate \eqref{eooo} of {\it Case (1)}, yields  \eqref{e4.11}.

Estimate \eqref{e4.12} is a consequence of \eqref{e4.11}.
The proof of Proposition~\ref{le4.4} is complete.
\end{proof}

\smallskip

 \subsection{Exponential-square estimates}
To prove our main  Theorem~\ref{th4.7} below, we first
 show that the boundedness of the discrete   square function
implies its  local exponential-square integrability. Recall that
 $c_1$ and  $C_2$ are   two  constants  given in \eqref{e4.11}.
Then we have the following result.

\begin{theorem}	\label{prop4.5}
 Suppose that
$L$ is a densely-defined operator on $L^2(X)$ satisfying
\textbf{(H1)}, \textbf{(H2)} and \textbf{(H3)}.  Let $f\in L^2(X,d\mu)$. For a fixed  $Q\in \mathcal{D}^b$ for some $b\in
\{1,2,\cdots,K\}$, and $f_{\mathcal{G}}$  be given in \eqref{e4.10}.
  Then there exists   a constant
{$c_2=c_2(n,\rho)\in (0, 1)$} independent of $\mathcal{G}$, $K_\mathcal{G}$, $Q$ and $f$
  such that for all $\lambda>0$,
	\begin{equation}
	\label{e4.18}
	\mu\big(\left\{x\in {Q}:\     |f_\mathcal{G}(x)|>\lambda
\right\}\big)
  \leq   C_2 \exp\left(-{c_2} \f{\lambda^2}
  {\|\mathscr A_L(f_\mathcal{G})
  \|_{L^\infty({Q})}^2}\right)\, \mu({Q}).
	\end{equation}
	
	As a consequence, there exist two positive constants {$c_3=c_3(n,\rho)\in (0, c_2)$} and ${C}_6=C_6(n,\rho)$
	independent of  $\mathcal{G}$, $K_\mathcal{G}$, $Q$ and $f$ such that 	
\begin{equation}
\label{e4.19}
\f{1}{\mu(Q)} \int_{Q} \exp\left(\f{{c_3}|f_{\mathcal{G}}(x)|^2}
{\|\mathscr A_L(f_{\mathcal{G}})\|_{L^\infty(Q)}^2}\right)\, d\mu(x)
\leq C_6.	
\end{equation}

\end{theorem}

 \begin{proof}
To prove  \eqref{e4.18}, it is enough  to  show  that for {$c_4= c_1/4$}, there exists  a positive constant {$C_7=C_7(n,\rho)$}
independent of  $f$
 such that
 \begin{equation}
\label{e4.20}
\int_{Q} \exp \left( {c_4} \left(|f_\mathcal{G}|-C_7
\|\mathscr A_L(f_\mathcal{G})\|_{L^\infty({Q})}^2\right)
\right) d\mu(x)    \leq   C_2
\mu({Q}).
\end{equation}
 Once \eqref{e4.20} is proven, \eqref{e4.18} follows readily.
Indeed,  we take   $g=\displaystyle
 {\lambda } f/{[ 2C_7
 \|\mathscr A_L(f_{\mathcal{G}})\|_{L^\infty(Q)}^2]}
\in L^2(X)$.   Replacing $f_\mathcal{G}$ by
$g_\mathcal{G}$ in \eqref{e4.20} and using the fact $|(Cf)_{\mathcal{G}}|
=C|f_{\mathcal{G}}|$ and
$\mathscr A_L((Cf)_{\mathcal{G}})=C\mathscr A_L(f_{\mathcal{G}}) $ for any
constant $C>0$, we have
$$
\int_{Q} \exp\Bigg\{{c_4} \Bigg(  \f{\lambda }
{2C_7 \|
\mathscr A_L(f_{\mathcal{G}})
\|_{L^\infty(Q)}^2   } \Big( |f_{\mathcal{G}} | -
\f{\lambda}{2}\Big)
 \Bigg)\Bigg\} \, d\mu(x)  \leq C_2\, \mu(Q).
 $$
 From this, it is easy to see that  \eqref{e4.18} holds for ${c_2=c_4}/(4C_7).$

We now turn to verify  \eqref{e4.20}. Observe that if $ \|\mathscr A_L(f_\mathcal{G}) \|_{L^\infty({Q})} \leq 2$,
we then apply \eqref{e4.12}  to obtain
\begin{eqnarray}\label{e4.21}
\int_{Q} \exp \left( {c_4} |f_\mathcal{G}|
\right) d\mu(x) \leq  C_2\mu({Q}).
\end{eqnarray}
Then it  reduces to show  \eqref{e4.20} for  the case   $\left\|\mathscr A_L
(f_\mathcal{G})\right\|_{L^\infty(Q)}
>2$.
Define
 $
\displaystyle \mathcal{G}_1=\{E_2(P)\in \mathcal{G}:\    \
 |\lambda_P|^2/\mu(P) <1\}
  $ and  $\mathcal{G}_2=\mathcal{G}\setminus
 \mathcal{G}_1$.
 We decompose $f_\mathcal{G}$ into two parts:
\[f_\mathcal{G}=\sum_{ E_2(P)\in \mathcal{G}_1}  a_P(f)+
\sum_{E_2(P)\in \mathcal{G}_2}  a_P(f)   =:
f_{{\mathcal{G}}_1}+f_{{\mathcal{G}}_2}.
\]
Notice that $\|a_P(f)\|_{L^\infty({Q})} \leq C|\lambda_P|/{\mu(P)^{1/2}}
\leq C|\lambda_P|^2 /\mu(P)$ for $E_2(P)\in \mathcal{G}_2$,
 where $C$ is in (ii) of Lemma~\ref{le3.1} and $0<C\leq C_1$ and
 $C_1$ is the constant in \eqref{e4.14}.
 This tells us that
$$
|f_{\mathcal{G}_2}(x)|\leq C_1\sum_{  E_2(P)\in
\mathcal{G}_2}   \frac{\left|\lambda_P\right|^2}{\mu(P) } \chi_{E_2(P)}(x) .
$$
From this, we know
\begin{eqnarray}\label{em5}
|f_\mathcal{G}(x)| -C_1\mathscr A_L^2(
f_{\mathcal{G}}) (x)
 \leq  |f_{\mathcal{G}_1}(x)|.
\end{eqnarray}

To continue, we consider the following two cases.

\smallskip

\noindent
{\it Case (1)}:  $\left\|\mathscr A_L(f_{\mathcal{G}_1})\right\|_{
L^\infty({Q})} \leq 2$.

\smallskip

In this case, we  apply  \eqref{e4.21}  for $f_{\mathcal{G}_1}$ and
\eqref{em5} to obtain
\begin{eqnarray}
\label{e4.22}
 \int_{Q} \exp\left\{{c_4}\left(|f_\mathcal{G}(x)| -C_1\|\mathscr A_L(
f_{\mathcal{G}})\|^2_{L^\infty(Q)}
\right)\right\} d\mu(x)
&\leq& \int_{Q} \exp\left\{{c_4}\left(|f_\mathcal{G}(x)| -C_1\mathscr A_L^2(
f_{\mathcal{G}}) (x)
\right)\right\} d\mu(x)\nonumber\\
&\leq&   \int_{Q} \exp\left({c_4} \left|f_{\mathcal{G}_1}(x)\right|\right)  d\mu(x)
\nonumber\\
&\leq&    C_2 \mu({Q}).
\end{eqnarray}

\smallskip

\noindent
{\it Case (2)}: $\left\|\mathscr A_L(f_{\mathcal{G}_1})\right
\|_{L^\infty({Q})} > 2$.

 In this case, we assume $Q\in \mathcal{D}_{k_0}^b$ for some $k_0\leq K_\mathcal{G}$.
 Let
\[
\{R_1,R_2,\cdots ,R_{M_0}\}=Ch^{(K_{\mathcal{G}}-k_0)}(Q).
\]
Then it follows that for every $R_j$, $1\leq j\leq M_0$,
and $E_2(P)\in \mathcal{G}$, either $R_j\cap E_2(P)=\emptyset$ or $R_j\in Ch^{(\ell)}(E_2(P))$
for some $\ell\geq 0$. For every $j=1, 2, \cdots, M_0,$ set
$$
\L_j=\left\{ E_2(P)\in {\mathcal G}_1:   R_j\subset  E_2(P) \right\}.
$$

Let us  consider $\L_1$. Note that when  $\left\|\mathscr A_L(f_{\mathcal{G}_1})\right\|_{L^\infty(R_1)} = 0$,
there does not exist   $E_2(P)\in \mathcal{G}_1$ such that $R_1\in Ch^{(\ell)}(E_2(P))$ for some $\ell\geq 0$;
when $\left\|\mathscr A_L(f_{\mathcal{G}_1})\right\|_{L^\infty(R_1)} \neq 0$,
there are a constant   $K_1$ with $1\leq K_1\leq \min\{2K_{\mathcal{G}}+1, K_\mathcal{G}-k_0+1 \}$ and
 a chain $\L_1:$
\begin{equation*}
\L_1:\ \ \    \   E_2(P_1)\supset  E_2(P_2)  \supset E_2(P_3)
\supset \cdots \supset   E_2(P_{K_1}), \
 \   \   \    E_2(P_k)\in \mathcal{G}_1, \  R_1\subset E_2(P_k),   \  1\leq k\leq K_1.
\end{equation*}
In the chain $\L_1$, if there exist  $E_2(P_k)$ and $E_2(P_{k+1})$  with
$E_2(P_k)=E_2(P_{k+1})$,  then    $E_2(P_{k+1})\in Ch^{(\ell)}
E_2(P_{k})$ for some $\ell\geq 1$.
Note that $E_2(P)\in \mathcal{G}_1$
means that $|\lambda_P|^2/\mu(P) <1$. We use the
following  method to classify cubes in the chain $\L_1$ in the following way:

\smallskip

\noindent
{  (i):}
 If
  \[\left(\sum_{k=1}^{K_1} \f{|\lambda_{P_k}|^2}{\mu(P_k)}\right)^{1/2}\leq 2,\]
   set $\textbf{flag}(E_2(P_k))=1$ for $1\leq k\leq K_1$;

 \medskip

 \noindent
{ (ii):}
  Otherwise, there exists a sequence   of $N_1+1$ natural numbers  $0=c_0<2\leq c_1<c_2<\cdots <c_{N_1}\leq K_1-1$ such that
      \[1<\Bigg(\sum_{k=c_i+1}^{c_{i+1}} \f{|\lambda_{P_k}|^2}{\mu(P_k)}\Bigg)^{1/2}
  \leq 2, \   \  \  \Bigg(\sum_{k=c_i+1}^{c_{i+1}+1} \f{|\lambda_{P_k}|^2}{\mu(P_k)}\Bigg)^{1/2}
 >2,  \  \  \   i=0,1,\cdots, {N_1 -1},\]
 and
 \[\Bigg(\sum_{k=c_{N_1}+1}^{K_1} \f{|\lambda_{P_k}|^2}{\mu(P_k)}\Bigg)^{1/2}
  \leq 2.\]
  For every $i$ with $0\leq i\leq N_1-1$, set $\textbf{flag}(E_2(P_k))=i+1$ if $c_i +1\leq k\leq c_{i+1}$.
  When $k\in [c_{N_1}+1, K_1]$, set $\textbf{flag}(E_2(P_k))=N_1+1$.

  To understand the above algorithm, we take an example of the case $K_1=8$,  $c_1=3$, $c_2=6$   and $N_1=2$
    to classify cubes in the chain $\L_1$ as follows.

\smallskip

\begin{center}
\begin{tikzpicture}
    \small
   \cube1[70,0,4,p1,(5,12.4)] {\scriptsize  $ E_2(P_1)$ };
   \cube1[60,0,4,p2,(5,10)] {\scriptsize  $ E_2(P_2)$ };
   \cube1[55,0,4,p3,(5,7.9)] {\scriptsize  $ E_2(P_3)$ };
   \cube1[48,0,4,p4,(5,5.9)] { \scriptsize $E_2(P_4)$ };
   \cube1[41,0,4,p5,(5,4.2)] {\scriptsize  $E_2(P_5)$ };
   \cube1[35,0,4,p6,(5,2.7)] { \scriptsize $E_2(P_6)$ };
   \cube1[30,0,4,p7,(5,1.3)] { \scriptsize $E_2(P_7)$ };
   \cube1[26,0,4,p8,(5,0)] { \scriptsize $E_2(P_8)$ };

   \draw[-latex] (p1) .. node[left=1.6cm, above=0.5mm] {  $
    \Bigg(\sum\limits_{k=1}^{3} \frac{|\lambda_{P_k}|^2}{\mu(P_k)}\Bigg)^{1/2}
  \leq 2, $}  node[left=1.6cm, below=0.5mm] {  $   \Bigg(
  \sum\limits_{k=1}^{4} \frac{|\lambda_{P_k}|^2}{\mu(P_k)}\Bigg)^{1/2}
  >2.  $} controls (2.2,12.3) and (2.2,7.8) ..  (p3);

  \draw[-latex] (p1) .. node[right=1.6cm,above=0.5mm] { $
  \textbf{flag}(E_2(P_k))=
  \textbf{1},$}   node[right=1.6cm,below=0.5mm] { $k=1,2,3.$}
  controls (7.8,12.3) and (7.8,7.8) ..  (p3);

     \draw[-latex] (p4) .. node[left=1.6cm, above=0.5mm] {  $
     \Bigg(\sum\limits_{k=4}^{6} \frac{|\lambda_{P_k}|^2}{\mu(P_k)}\Bigg)^{1/2}
  \leq 2, $}  node[left=1.6cm, below=0.5mm] {  $   \Bigg(\sum\limits_{k=4}^{7}
  \frac{|\lambda_{P_k}|^2}{\mu(P_k)}\Bigg)^{1/2}
  >2.  $} controls (2.8,5.8) and (2.8,2.8) ..  (p6);

  \draw[-latex] (p4) .. node[right=1.6cm,above=0.5mm] { $
  \textbf{flag}(E_2(P_k))=
  \textbf{2},$} node[right=1.6cm,below=0.5mm] { $k=4,5,6.$}
  controls (7.2,5.8) and (7.2,2.8) ..  (p6);

     \draw[-latex] (p7) .. node[left] { $ \Bigg(\sum\limits_{k=7}^{8} \frac{
     |\lambda_{P_k}|^2}{\mu(P_k)}\Bigg)^{1/2}
  \leq 2.$} controls (3.6,1.2) and (3.6,0.1) ..  (p8);
  \draw[-latex] (p7) .. node[right=1.6cm,above=0.5mm] { $
  \textbf{flag}(E_2(P_k))=
  \textbf{3},$}node[right=1.6cm,below=0.5mm]{$k=7,8.$}
  controls (6.4,1.2) and (6.4,0.1) ..  (p8);

\end{tikzpicture}\\
\textbf{Figure 1.} The case $K_1=8$,  $c_1=3$, $c_2=6$ and $N_1=2$.
\end{center}

\smallskip

Now we  estimate $N_1$. Note that  for each $0\leq i\leq N_1-1$
\begin{equation*}
  1<\sum_{k=c_i+1}^{c_{i+1}} \f{|\lambda_{P_k}|^2}{\mu(P_k)}
  \leq 4,
  \end{equation*}
which  gives
\begin{equation*}
  N_1<\sum_{i=0}^{N_1-1}\sum_{k=c_i+1}^{c_{i+1}} \f{|\lambda_{P_k}|^2}{\mu(P_k)}
  =\sum_{k=1}^{c_{N_1}} \f{|\lambda_{P_k}|^2}{\mu(P_k)} \chi_{E_2(P)}(x_0) \leq
   \sum_{k=1}^{K_1} \f{|\lambda_{P_k}|^2}{\mu(P_k)} \chi_{E_2(P)}(x_0) \leq 4N_1+4
  \end{equation*}
for some $x_0\in R_1$, and so
\begin{equation*}
  N_1<  \|\mathscr A_L(
f_{\mathcal{G}_1})
 \|_{L^\infty({R_1})}^2\leq  \|\mathscr A_L(
f_{\mathcal{G}_1})
 \|_{L^\infty({Q})}^2.
  \end{equation*}
Then we have
\[\max\left\{\textbf{flag}(E_2(P_k)):\   k=1,2,\cdots, K_1\right\}=\textbf{flag}(E_2(P_{K_1})) =N_1+1< \|\mathscr A_L(
f_{\mathcal{G}_1})
 \|_{L^\infty({Q})}^2+1.\]

\smallskip

Next we consider  the chain $\L_j$,  $j=2,3,\cdots, M_0$.
When $\left\|\mathscr A_L(f_{\mathcal{G}_1})\right\|_{L^\infty(R_j)} \neq 0$,
we use the algorithm as in $\L_1$ above to classify   cubes in the
 chain {$\L_j$}. For each $E_2(P)$ in the chain $\L_j$,    $\textbf{flag} (E_2(P))$ is obtained.
 Observe that if $E_2(P_k)$ is a
common one contained in two chains, then all $E_2(P_\ell)$,  $\ell=1,\ldots,k$,
are common one. This fact, together with   our algorithm above, determines that one can
verify that
the function
$\textbf{flag}:\  \mathcal{G}_1\to \mathbb{N}_+$ is well-defined and single-valued.

From the above algorithm, we know   that for every
$E_2(P)\in \mathcal{G}_1$, it  can  be assigned a label ${\textbf{flag}}(E_2(P))$, and
\begin{equation}
\label{e4.23}
N_0:=\max\left\{\textbf{flag}(E_2(P)):\    E_2(P)\in \mathcal{G}_1\right\} <  \|\mathscr A_L(
f_{\mathcal{G}_1})
 \|_{L^\infty({Q})}^2+1.
\end{equation}

For $ 1\leq i\leq  N_0$, denote
\[\mathcal{G}_{(i)}=\left\{E_2(P)\in \mathcal{G}_1:\ \ \
\textbf{flag}(E_2(P))=i\right\} .
\]
Define
\[f_{\mathcal{G}_{(i)}}=\sum_{E_2(P)\in \mathcal{G}_{(i)}} a_P(f),\   \
1\leq i\leq  N_0,\]
which yields
\[
f_{\mathcal{G}_1} =\sum_{i=1}^{N_0}  f_{\mathcal{G}_{(i)}}.
\]
Set
\[\mathcal{B}_{(j)}=\left\{E_2(P)\in \mathcal{G}_{(j)}:\   \forall
\ E_2(R)\in \mathcal{G}_{(j)}, \ \text{either}\   E_2(R)\in Ch^{(\ell)}
(E_2(P)) \     \text{for}\  \ell\geq 0\
   \text{or}\    E_2(R)\cap
E_2(P)=\emptyset\right\}.
\]
That is, $\mathcal{B}_{(j)}$ is the set of maximal cubes in $\mathcal{G}_{(j)}$.
With the notation above, we  claim that:

\smallskip

\noindent
{\bf (1)}
 For every $ 1\leq j\leq  N_0,$
	\begin{eqnarray}\label{ecc}
	\left\|\mathscr A_L(f_{\mathcal{G}_{(j)}}) \right\|_{
		L^\infty({Q})}\leq 2;
	\end{eqnarray}
	
\smallskip

\noindent
{\bf (2)}  There exists constant {$C_8=C_8(n,\rho)>0$} such that for every $R\in
	\mathcal{B}_{(k)}$, $2\leq k\leq N_0$,
	\begin{equation}
	\label{e4.24}
	\Bigg|   \Big(  \sum_{j=1}^{k-1}  f_{\mathcal{G}_{(j)}}(x)\
	\Big)  -  \Big( \sum_{j=1}^{k-1}  f_{\mathcal{G}_{(j)}}\
	\Big)_{{R}}  \Bigg|     \leq C_8, \    \  a.e.\
	x\in {R}\in  \mathcal{B}_{(k)}.
	\end{equation}

The  proof of  \eqref{ecc} is simple.   Indeed,  since for every $x\in Q$,
 there exists a unique $R_{i}\in Ch^{(K_{\mathcal{G}}-k_0)}(Q) $, $1\leq i\leq M_0$
such that $x\in R_{i}$, and
$\mathscr A_L(f_{\mathcal{G}_{(j)}})(x)=\|\mathscr A_L(f_{\mathcal{G}_{(j)}})\|_{L^\infty(R_i)}$.
It follows from the above  algorithm that
$$
\mathscr A_L(f_{\mathcal{G}_{(j)}})(x)=\Bigg(\sum_{\textbf{flag}(E_2(P))=j,\  R_i\subset E_2(P)}
 \f{|\lambda_P|^2 }{\mu(P)}\Bigg)^{1/2}\leq 2
$$
as desired.
We now verify  \eqref{e4.24}.
From  (iii) of Lemma~\ref{le3.1}  and
$K_{\mathcal{G}}<\infty$, there exists a point $\tilde{x}\in R$ such that $
  \ \sum\nolimits_{j=1}^{k-1}  f_{\mathcal{G}_{(j)}}
 (\tilde{x})\
  =   \big( \ \sum\nolimits_{j=1}^{k-1}  f_{\mathcal{G}_{(j)}}\
\big)_{{R}} $.
Notice that for every $E_2(P)\in \mathcal{G}_{(j)}$ with $R\subset E_2(P)$,
there exists some $\ell$ with $ k-j\leq \ell\leq \min\{2K_{\mathcal{G}},K_{\mathcal{G}}-k_0\}$ such that
$R\in Ch^{(\ell)}(E_2(P))$.   Denote $R\in \mathcal{D}_{\ell_0}^b$ and so $E_2(P)\in \mathcal{D}_{\ell_0-\ell}^b$.
 Note that $E_2(P)\in \mathcal{G}_1$
means that $|\lambda_P|^2/\mu(P) <1$. Therefore, one can apply  (iii) of
Lemma~\ref{le3.1}  to show that  for every $j=1,2,\cdots, k-1$ and $x\in R$,
\begin{eqnarray*}
\left|f_{\mathcal{G}_{(j)}}(x)-f_{\mathcal{G}_{(j)}}(\tilde{x})\right| &\leq&
\sum_{\substack{E_2(P)\in \mathcal{G}_{(j)}\\  R\subset E_2(P)}} \big|a_P(f)(x)-
a_P(f)(\tilde{x})\big|\\
&\leq & C \sum_{\substack{k-j\leq \ell\leq 2K_{\mathcal{G}}\\ E_2(P)\in
\mathcal{G}_{(j)},  R\in Ch^{(\ell)}(E_2(P))    }}  \Big(\f{d(x,x')}{
  \delta^{\ell_0-\ell+1}}\Big)^{\theta} \f{\lambda_P}{\mu(P)^{1/2}}\\
&\leq & C\sum_{\ell\geq k-j}   \delta^{
\theta(\ell-1)}  \\
&\leq & \f{C}{\sqrt{1-\delta^{2\theta}}}\delta^{\theta (k-j-1)},
\end{eqnarray*}
where the third inequality holds by using the fact that for $R\in Ch^{(\ell)}(E_2(P))$,
it follows from Proposition~\ref{prop4.2} that  for $x,x'\in R$,
\[\f{d(x,x')}{\delta^{\ell_0-\ell+1}}\leq C\delta^{-1} \f{\delta^{\ell_0}}{\delta^{\ell_0-\ell}} \leq C\delta^{\ell-1}.\]
\noindent Hence,
\begin{eqnarray*}
 \text{LHS of}\  \eqref{e4.24}   &\leq &   \sum_{j=1}^{k-1}  \left|f_{
 \mathcal{G}_{(j)}}(x)-f_{\mathcal{G}_{(j)}}(\tilde{x})\right|
 \leq  \f{C}{\sqrt{1-\delta^{2\theta}}} \sum_{j=1}^{k-1} \delta^{\theta(k-j-1)}
 \leq {C\over (1-\delta^{\theta})\sqrt{1-\delta^{2\theta}}  }.
\end{eqnarray*}
This  finishes the proof of   \eqref{e4.24} with $\displaystyle C_8={C\over (1-\delta^{\theta})
\sqrt{1-\delta^{2\theta}}  }$.

\smallskip

Next we  apply \eqref{ecc} and  \eqref{e4.24}  to derive  the   estimate \eqref{e4.20} by the iteration method.
	When ${R_{(k)}}\in \mathcal{B}_{(k)}, k=1,2, \cdots, N_0$,
we apply \eqref{e4.12} and \eqref{ecc} to obtain
\begin{eqnarray}\label{e4.25}
\int_{R_{(k)}} \exp \Big( {c_4} |f_{\mathcal{G}_{(k)}}|
\Big) d\mu(x) \leq  C_2\, \mu(R_{(k)}).
\end{eqnarray}
This, in combination with  \eqref{e4.24}, implies that for any $k\in \{2,3,\cdots,N_0\}$,
\begin{eqnarray}\label{e4.26}
&&\hspace{-1.2cm}\int_{{R_{(k)}}}    \exp\Big({c_4} \big|\sum _{j=1}^{k}
f_{\mathcal{G}_{(j)}}(x)\big| \Big)d\mu(x)\nonumber \\
&\leq&
\exp\Big({c_4} \big|  \sum_{j=1}^{k-1}  f_{\mathcal{G}_{(j)}}
\big|_{{R_{(k)}}} \Big)\int_{{R_{(k)}}}    \exp\Big({c_4} \big|
\big(\sum _{j=1}^{k-1}
f_{\mathcal{G}_{(j)}}(x)\big) - \big( \ \sum_{j=1}^{k-1}  f_{\mathcal{G}_{(j)}}\
\big)_{{R_{(k)}}}  \big| +{c_4} \left|f_{\mathcal{G}_{(k)}}(x)\right|\Big)
 d\mu(x) \notag\\
&\leq &   C_2\exp\big({c_4}
C_8\big)  \exp\Big({c_4} \big| \ \sum_{j=1}^{k-1}
f_{\mathcal{G}_{(j)}}\
\big|_{{R_{(k)}}}\Big)\mu(R_{(k)}).
\end{eqnarray}
For $x\in Q
\setminus \bigcup_{R_{(k)}\in \mathcal{B}_{(k)} }R_{(k)}$, we have that $f_{\mathcal{G}_{(k)}}(x)=0$
since  $\support a_P(f)\subset E_2(P)$. Then \eqref{e4.26}, together with  the convexity of the exponential
function, implies
\begin{eqnarray}\label{e4.27}
 \int_{{Q}} \exp\Big(c_4  \big|\sum _{j=1}^{k}
f_{\mathcal{G}_{(j)}}(x) \big|\Big)d\mu(x)
&=&  \sum_{{R_{(k)}}\in \mathcal{B}_{(k)} }
\int_{{R_{(k)}}}    \exp\Big(c_4 \big|\sum _{j=1}^{k}
f_{\mathcal{G}_{(j)}}(x)\big| \Big)d\mu(x)\nonumber\\
&&  + \int_{\displaystyle Q
\setminus \cup_{R_{(k)}\in \mathcal{B}_{(k)} }R_{(k)}}
\exp\Big(c_4 \big|\sum _{j=1}^{k-1}
f_{\mathcal{G}_{(j)}}(x)\big| \Big)d\mu(x) \nonumber\\
&\leq & C_2\exp\left(c_4 C_8\right)  \int_{Q} \exp
\Big(c_4 \big|\sum _{j=1}^{k-1}
f_{\mathcal{G}_{(j)}}(x)\big|\Big)d\mu(x).
\end{eqnarray}
Recall that $C_4$ is the constant given in \eqref{e4.17},
$C_2=\exp\left(8C_4 {c_1}\right)$, and   $ \left\|\mathscr A_L(f_{\mathcal{G}_1})
 \right\|_{L^\infty(Q)}^2> N_0-1.$
  By iteration, we have
\begin{eqnarray}\label{em4}
\int_{{Q}} \exp\left(c_4 |f_{\mathcal{G}_1}(x)|\right)d\mu(x)
&=&\int_{{Q}} \exp\Big(c_4 \big|\sum_{j=1}^{N_0}
f_{\mathcal{G}_{(j)}}(x)\big|\Big)d\mu(x)\nonumber\\
&\leq & C_2^{N_0 -1} \exp\left(c_4(N_0 -1)C_8\right)
\int_Q \exp\left(c_4 \left|f_{\mathcal{G}_{(1)}} (x)
\right|\right)d\mu(x)\nonumber\\
&\leq & C_2^{N_0} \exp(c_4 (N_0-1) C_8)  \mu(Q)\nonumber\\
&\leq&   C_2\exp\left\{c_4  \left(32C_4+C_8\right)(N_0-1)
\right\}\mu(Q)\nonumber\\
&\leq&  C_2\exp\left(c_4 (32C_4 +C_8) \|\mathscr A_L(f_{\mathcal{G}_1})
 \|_{L^\infty(Q)}^2
  \right)\mu(Q).
\end{eqnarray}
Finally, we set $C_7=32C_4 +C_8 +C_1 $ to   apply   \eqref{em5} and \eqref{em4}
to obtain
\begin{eqnarray}\label{222}
 \int_{Q} \exp\left(c_4 |f_{\mathcal{G}}(x)| \right)d\mu(x)
 & \leq& \int_{Q} \exp\left(c_4 \left(|f_{\mathcal{G}_1}(x)| +C_1\|\mathscr A_L(f_{\mathcal{G}})
 \|_{L^\infty(Q)}^2\right)\right)d\mu(x) \nonumber\\
  & \leq&   C_2 \exp\left(c_4  C_7\|\mathscr A_L(f_{\mathcal{G}})
 \|_{L^\infty(Q)}^2\right)  \mu(Q).
 \end{eqnarray}

From {\it Cases (1)} and {\it  (2)}, we have  obtained  \eqref{e4.20}.
 The proof of Theorem~\ref{prop4.5} is complete.
\end{proof}

As a consequence, we have the following corollary.

\begin{corollary}
\label{cor4.6}
For a fixed  $Q\in \mathcal{D}^b$ for some $b\in
\{1,2,\cdots,K\}$, and $f_{\mathcal{G}}$  be given in \eqref{e4.10}.
Let $0<p<\infty$. {There exists a positive constant $C=C(p,n,\rho)$ independent of $\mathcal{G}$, $K_\mathcal{G}$, $Q$ and $f$
  such that} for every nonnegative
$ V\in L^1_{\rm loc}(X)$ with $\int_{Q} V(x)d\mu(x)\neq  0$,
\begin{eqnarray}\label{e4.29}
\int_{Q}  |f_\mathcal{G}(x)|^p V(x)d\mu(x) \leq C  \left\|\mathscr A_{L}
(f_{\mathcal{G}})\right\|_{L^\infty
(Q)}^p   \int_{Q} V(x)\left(\log\left(e +\f{V(x)}{V_{Q}}\right)\right)^{p/2}d\mu(x).
\end{eqnarray}
\end{corollary}

\begin{proof}
From \eqref{e4.19}, we have that  $\displaystyle  \Bigg(\f{\sqrt{c_3}
|f_{\mathcal{G}}|}{ \|\mathscr A_L(f_{\mathcal{G}})
 \|_{L^\infty(Q)}}\Bigg)^{p}
\in {\rm{Exp}}_{C_6}(Q,2/p)$ for any $0<p<\infty$, {where $c_3$ and $C_6$ are   constants} in \eqref{e4.19} and
\[{\rm{Exp}}_{C_6}(Q,1/\gamma)=\left\{\psi:\    \f{1}{\mu(Q)} \int_Q
\exp\left(\psi^{1/\gamma}(x)\right)d\mu(x) \leq C_6+1 \right\}.\]

We follow  \cite[Theorem~11.2]{Wi08} to see that if $V_Q\neq 0$,
then  for all $\gamma>0$,
\begin{equation}
\label{e4.30}
\int_Q   V(x)\left(\log\left(C_6 +\f{V(x)}{V_Q}\right)\right)^{\gamma}d\mu(x)
 \sim \sup\left\{\int_Q V(x)\phi(x)d\mu(x):\  \  \phi\in
 {\rm{Exp}}_{C_6}(Q,1/\gamma)
  \right\},
\end{equation}
with comparability constants that only depend on $\gamma$.

By \eqref{e4.30} with $\gamma=2/p$,
\begin{eqnarray*}
\int_{Q} \Bigg(\f{\sqrt{c_3}
|f_{\mathcal{G}}|}{ \|\mathscr A_L(f_{\mathcal{G}})
 \|_{L^\infty(Q)}}\Bigg)^{p} V(x)d\mu(x) \leq C_p
\int_{Q} V(x)\left(\log\left(C_6 +\f{V(x)}{V_{Q}}\right)\right)^{p/2}d\mu(x).
\end{eqnarray*}
Meanwhile, there exists constant $C=C(C_6,p)$ independent of $V$ and $Q$, such that
\[\int_{Q} V(x)\left(\log\left(C_6 +\f{V(x)}{V_{Q}}\right)\right)^{p/2}d\mu(x)
 \leq C \int_{Q} V(x)\left(\log\left(e +\f{V(x)}{V_{Q}}\right)\right)^{p/2}d\mu(x). \]
Hence \eqref{e4.29} holds, and  we complete the proof of Corollary~\ref{cor4.6}.
\end{proof}

\medskip

We are now ready to prove   the following theorem.

 \begin{theorem}
\label{th4.7}
Suppose that
$L$ is a densely-defined operator on $L^2(X)$ satisfying
\textbf{(H1)}, \textbf{(H2)} and \textbf{(H3)}. Let $f\in L^2(X,d\mu)$ and assume $S_{L,\alpha} f \in L^\infty(X)$.
Then there exist two positive constants {$ c_5=c_5(n,\alpha)$, $ {C_9}=C_9(n,\alpha)$}
 independent of $f$ such that
\begin{eqnarray}\label{e4.31}
 \sup_{B}  \frac{1}{\mu(B)}   \int_{B} \exp\Bigg( {{c_5}}\f{ |f(x)-
f_{B} |^2}
{ \|S_{L,\alpha} f \|_{L^\infty(X)}^2}\Bigg) \, d\mu(x) \leq {C_9},
\end{eqnarray}
where the supremum  takes for all balls $B$.
\end{theorem}

\begin{proof}
To prove Theorem~\ref{th4.7}, {one can apply Lemma~\ref{le4.3} and fix the parameter $\rho=\delta^2 \alpha /5$ to obtain that  it suffices to show}
\begin{eqnarray}\label{e4.311}
 \sup_{B}  \frac{1}{\mu(B)}  \int_{B} \exp\Bigg( {c^2{c_5}}\f{ |f(x)-
f_{B} |^2}
{ \|\mathscr A_L(f) \|_{L^\infty(X)}^2}\Bigg) \, d\mu(x) \leq {C_9},
\end{eqnarray}
where $c$ is the constant in \eqref{e4.9}.

Let us show \eqref{e4.311}.
Since $f\in L^2(X)$, the Calder\'on reproducing formula \eqref{e3.1} tells us
$$
f
=\lim_{m\to \infty}\sum_{b=1}^K \sum_{Q\in \mathscr D,E_2(Q)\in \bigcup_{-m\leq s\leq m}\mathcal{D}_s^{b}} a_Q(f)(x)
=\lim_{m\to \infty}\sum_{b=1}^K F_{b,m}=\lim_{m\to \infty} F_m,
$$
where the limits converge in $L^2(X)$.
Then one can apply  the Riesz theorem
to see there exists some subsequence $\{m_i\}$ such that
$ \underset{i\to \infty}{\lim}F_{m_i}(x)=f(x)$, a.e. $x\in X$.
%
Also for every ball $B$,
\[\lim_{m\to \infty} \left|f_B- \left(F_{m}\right)_B\right|\leq \lim_{m\to \infty}
\Big(\frac{1}{\mu(B)} \int_B \left|f(x)-F_{m}(x)\right|^2 d\mu(x)\Big)^{1/2} =0.\]
It follows from the fact that $\mathscr{A}_L (F_{b,m})(x)\leq \mathscr{A}_L (F_m)(x)
\leq \mathscr{A}_L (f)(x)$,  a.e. $x\in X$,
\begin{eqnarray*}
\f{|f(x)-f_B|^2}{\|\mathscr A_L(f)\|_{L^\infty(X)}^2}
&\leq&3\f{|f(x)-F_{m_i}(x)|^2}{\|\mathscr A_L(f)\|_{L^\infty(X)}^2}+3\f{|F_{m_i}(x)-(F_{m_i})_B|^2}
{\|\mathscr{A}_L (F_{m_i})\|_{L^\infty(X)}^2}+3\f{|(F_{m_i})_B-f_B|^2}{\|\mathscr A_L(f)\|_{L^\infty(X)}^2},
\end{eqnarray*}
which implies
\begin{eqnarray*}
\f{|f(x)-f_B|^2}{\left\|\mathscr A_L(f)\right\|_{L^\infty(X)}^2}
&\leq& 3\underset{i\to \infty}{\underline{\lim}}\f{|F_{m_i}(x)-(F_{m_i})_B|^2}
{\left\|\mathscr{A}_L (F_{m_i})\right\|_{L^\infty(X)}^2}.
\end{eqnarray*}
From Fatou's lemma,
\begin{eqnarray}
\label{e4.32}
 \int_{B} \exp\Bigg( {c^2 {c_5}}\f{|f(x)-
f_{B}|^2}
{\|\mathscr A_L(f)\|_{L^\infty(X)}^2}\Bigg)d\mu(x)
 &\leq & C\underset{i\to \infty}{\underline{\lim}}     \int_{B}  \exp\Bigg( {3c^2 { c_5}}\f{|F_{m_i}(x)-
(F_{m_i})_{B}|^2}
{\|\mathscr A_L(F_{m_i})\|_{L^\infty(X)}^2}\Bigg)\,  d\mu(x) \notag\\
&\leq &   \underset{i\to \infty}{\underline{\lim}} \frac{C}{K}
\sum_{b=1}^K    \int_{B}  \exp\Bigg( {3K^2 c^2 {c_5}  }\f{|F_{b,m_i}(x)-
(F_{b,m_i})_{B}|^2}
{\|\mathscr A_L(F_{b,m_i})\|_{L^\infty(X)}^2}\Bigg)\, d\mu(x).
\end{eqnarray}
Therefore, it suffices to show for every ball $B$ and $F_{b,m}$,
$b\in \{1,2,\cdots, K\}$, $m>0$, there exist two positive constants {$ c_6$ and $C_{10}$}
 independent of  $F_{b,m}$ and $B$, such that
\begin{equation}
\label{e4.33}
\frac{1}{\mu(B)}   \int_{B} \exp\Bigg( {{c_6}}\f{|F_{b,m}(x)-
(F_{b,m})_{B}|^2}
{\|\mathscr A_L(F_{b,m})\|_{L^\infty(X)}^2}\Bigg)\, d\mu(x) \leq {C_{10}}.
\end{equation}

Let us prove \eqref{e4.33}. Assume $\mbox{diam}\,B\sim \delta^{k_0}$.
Define $\mathcal{G}_{b,m}=\left\{Q\in \mathcal{D}_{k_0}^b:\  \ Q\cap B\neq \emptyset \right\}$.
Obviously,  $\#\mathcal{G}_{b,m}\leq M_1$ for some  $M_1$  depending
 on $n$ and $\delta$ only.
Then {for $ x\in B$}, by $\support a_P \subset E_2(P)$, we write
\begin{eqnarray*}
F_{b,m}(x)&=&\sum_{Q\in \mathcal{G}_{b,m}} \sum_{\substack{E_2(P)\subset Q\\E_2(P)\in
\bigcup_{-m\leq s\leq m}\mathcal{D}_s^{b}}} a_P(f)(x)+
\sum_{Q\in \mathcal{G}_{b,m}} \sum_{\substack{ Q\subsetneqq E_2(P)\\E_2(P)\in
 \bigcup_{-m\leq s\leq m}\mathcal{D}_s^{b}}} a_P(f)(x)\\
&=:&\sum_{Q\in \mathcal{G}_{b,m}} F_{Q,1}(x)+\sum_{Q\in \mathcal{G}_{b,m}}F_{Q,2}(x).
\end{eqnarray*}
 Note that $\mathscr{A}_L (F_{Q,1})(x)\leq \mathscr{A}_L (F_{b,m})(x)$ and
  $\mathscr{A}_L (F_{Q,2})(x)\leq \mathscr{A}_L (F_{b,m})(x)$. By the convexity of the exponential function,
\begin{eqnarray*}
{\rm LHS\ of } \ \eqref{e4.33}
&\leq&
\frac{1}{2M_1} \sum_{Q\in \mathcal{G}_{b,m}} \sum_{j=1}^2 \frac{1}{\mu(B)} \int_{B}
\exp\Bigg( 4M^2_1 {{c_6}}\f{|F_{Q,j}(x)-
(F_{Q,j})_{B}|^2}
{\|\mathscr A_L(F_{Q,j})\|_{L^\infty(X)}^2}\Bigg)\, d\mu(x).
\end{eqnarray*}
It reduces  to show for every ball $B$ and $F_{Q,j}$,$j=1,2$,
there exist two positive constants {$ c_7$ and $ C_{11}$}
 independent of  $F_{Q,j}$ and $B$, such that
\begin{equation}
\label{e4.333}
\frac{1}{\mu(B)} \int_{B} \exp\Bigg( { {c_7}}\f{|F_{Q,j}(x)-
(F_{Q,j})_{B}|^2}
{\|\mathscr A_L(F_{Q,j})\|_{L^\infty(X)}^2}\Bigg) \leq {C_{11}}.
\end{equation}

Consider  the case $j=1$. \
We note that $\support F_{Q,1}\subset Q$. A similar approach to derive \eqref{e4.13} gives
\[\left|F_{Q,1}\right|_B^2 \leq \frac{C}{\mu(Q)}\int_{Q} |F_{Q,1}(x)|^2 d\mu(x)
 \leq CC^2_1\|\mathscr{A}_L(F_{Q,1} )\|_{L^\infty(Q)}^2, \]
where $C_1$ is the  constant in  \eqref{e4.13}. Then again note that $\support F_{Q,1}\subset Q$,
\begin{eqnarray*}
 &&\hspace{-1.2cm}\frac{1}{\mu(B)} \int_{B} \exp\Bigg( { c_7}\f{|F_{Q,1}(x)-
(F_{Q,1})_{B}|^2}
{\|\mathscr A_L(F_{Q,1})\|_{L^\infty(X)}^2}\Bigg)\, d\mu(x)
 \\
&\leq &  \frac{C\exp(CC_1^2)}{\mu(Q)}  \sum_{Q'\in \mathcal{G}_{b,m}}  \int_{Q'}
\exp\Bigg( {2 {c_7}}\f{|F_{Q,1}(x)|^2}
{\|\mathscr A_L(F_{Q,1})\|_{L^\infty(Q)}^2}\Bigg)\, d\mu(x)\\
&\leq&   \frac{C\exp(CC_1^2)}{\mu(Q)}   \int_{Q}
\exp\Bigg( {2  c_7}\f{|F_{Q,1}(x)|^2}
{\|\mathscr A_L(F_{Q,1})\|_{L^\infty(Q)}^2}\Bigg)\, d\mu(x)+(M_1-1)C\exp(CC_1^2).
\end{eqnarray*}
This, together with Theorem~\ref{prop4.5}, deduces that there exists a constant $C_{10}$
 independent on $F_{Q,1}$ and $B$, such that
\[\frac{1}{\mu(B)} \int_{B} \exp\Bigg( { c_7}\f{|F_{Q,1}(x)-
(F_{Q,1})_{B}|^2}
{\|\mathscr A_L(F_{Q,1})\|_{L^\infty(X)}^2}\Bigg)\, d\mu(x) \leq C_6M_1C\exp(CC_1^2)
\leq C_{11}\]
for  $ c_7\leq c_3/2$, where $c_3$ is the constant in \eqref{e4.19}.
This proves \eqref{e4.333} for the case $j=1$.
\smallskip

Next we work the case $j=2$. Note that $F_{Q,2}$ is continuous
on $B$, so there exists $x'\in B$ such that $(F_{Q,2})_B=F_{Q,2}(x')$.
A similar argument as in the proof of \eqref{e4.15} shows that
\[|F_{Q,2}(x)-(F_{Q,2})_B|=|F_{Q,2}(x)-F_{Q,2}(x')|
 \leq C_3\|\mathscr{A}_L(F_{Q,2})\|_{L^\infty(Q)}, \ \  x\in B,\]
where $C_3$ is the constant in \eqref{e4.15}. Then
\[\frac{1}{\mu(B)} \int_{B} \exp\Bigg( { c_7}\f{|F_{Q,2}(x)-
(F_{Q,2})_{B}|^2}
{\|\mathscr A_L(F_{Q,2})\|_{L^\infty(X)}^2}\Bigg)\, d\mu(x) \leq \exp( \gamma_3 C_3^2)
\leq C_{11},\]
which proves \eqref{e4.333} for the case $j=2$ and finishes the proof of Theorem~\ref{th4.7}. 
 \end{proof}

\medskip
\noindent
{\bf Remarks.}
We would like to comment on the possibility of several generalizations and
open problems related to the results of Theorem~\ref{th4.7}.


\smallskip

(1) The first one is the extension of   the work of
Chang-Wilson-Wolff (\cite[Theorem 3.2]{CWW}) to non-homogeneous metric spaces,
in place of ${\mathbb R^n}$ (see \cite{T, V}
 for the definition and properties of non-homogeneous metric spaces),  and this question will be considered in the future.

\smallskip

(2)
 In the  proof of  Theorem~\ref{th4.7},
a key step is to use Theorem~\ref{prop4.5}, which provides an algorithm to obtain a direct proof of
its  exponential-square integrability of   a function whose discrete square function
associated to an operator $L$ is bounded. In the proof, we  used the assumption  \textbf{(H3)}
of H\"older's continuity in $x$ of
the kernels $p_t(x,y)$ of $e^{-tL}$.  We may ask whether  Theorem~\ref{th4.7} still holds
 assuming merely  that an operator $L$ satisfies  ${\bf (H1)}$ and ${\bf (H2)}$.
 This problem continues a line of the study of harmonic analysis properties  of abstract selfadjoint
operators   whose kernel   satisfies weaker conditions than standard regularity
estimates,  and there was lots of success in solving this problem in the last few years
see for example, in  \cite{Au07, ACDH, AM, BFP, DM, DOS, DY1, DY, HLMMY} and the
references therein.

\smallskip

\bigskip

\section{Singular integral operators}
\setcounter{equation}{0}

In this section, we apply the previous results in Section  4 to obtain
some   estimates of the norms on $L^p$ as $p$ becomes large for operators such as
 the square functions or spectral
 multipliers,  and weighted norm inequalities for the square functions.

 \subsection{Square function results} In this subsection
we establish sharp lower bounds
 for the operator norm  of  the square functions  on $L^p(X,d\mu)$ as $p$ becomes large.
\begin{proposition}
\label{th5.1}
Let $(X,d,\mu)$ be a metric measure space endowed with a distance $d$ and a
nonnegative Borel
doubling measure $\mu$ on $X$.  Assume that $L$ is a
densely-defined operator
on $L^2(X)$
satisfying \textbf{(H1)}, \textbf{(H2)} and \textbf{(H3)}. { Then there exists a positive constant $C=C(n,\rho)$
  such that}
\begin{equation}
\label{e5.1}
\|f\|_{L^p(X,d\mu)} \leq C p^{1/2}\  \|{\mathscr{A}_L}(f)\|_{L^p(X,d\mu)} \  \
\text{as} \ \   p\to \infty.
\end{equation}
\end{proposition}
\begin{proof}  

The proof of \eqref{e5.1} is  based on combing a technique of \cite[Theorem~11.3]{Wi08}
and the
Fefferman-Pipher's method \cite{FP97}.  Notice that
\[ \|f\|_{L^p(X,d\mu)}=\sup_{\|V\|_{ L^{p'}(X,d\mu)} =1} \int_X
f(x)V(x)d\mu(x),\ \  \    \  1/p+1/{p'}=1,\]
it suffices to show  that there exists a constant $C>0$ such that
\[\left|\int_X f(x)V(x)d\mu(x) \right|  \leq  Cp^{1/2}
\|{\mathscr{A}_L}(f)\|_{L^p(X,d\mu)}, \ \   \|V\|_{ L^{p'}(X,d\mu)} =1 \]
as $p\to \infty$.

Next we decompose $f$ and $V$ to $f_k$ and $V_k$ such that $f_k,V_k\in L^2(X)$ and so
Calder\'on reproducing formula can be applied.
Fix a sequence of sets $E_1\subset E_2\subset \cdots \subset E_k\subset E_{k+1}\subset \cdots
\subset X$ satisfying $\mu(E_k)<\infty$ for every $k\geq 1$ and $X=\bigcup_k
E_k$. For every $f\in L^p(X,d\mu)$, $2\leq p<\infty$, let
$f_k=f\chi_{E_k}$, $k=1,2,
\cdots$. It's clear that $f_k\in L^2(X,d\mu)\cap L^p(X,d\mu)$ and
$\displaystyle \lim_{k\to \infty}\|f-f_k\|_{
L^p(X,d\mu)}=0$.  Notice  that $L_0^\infty(X)$ is densely contained in $L^q(X)$ for
every $1<q<\infty$, where  $L_0^\infty(X)$ denotes the space of bounded
functions with bounded support. Therefore for  every $V\in L^{p'}(X,d\mu)$ with
$\|V\|_{L^{p'}(X,d\mu)}=1$, there exists  $\{V_k\}_{k=1}^\infty\subset
L_0^\infty(X,d\mu)$ such that $\displaystyle \lim_{k\to \infty}\|V-V_k
\|_{L^{p'}(X,d\mu)}=0$. In the sequel, we may assume
$\|V_k\|_{L^{p'}(X,d\mu)}\leq 2$ for all $k\geq 1$.

For $f\in L^p(X)$ and $V\in L^{p'}(X)$, note that $f_k,V_k\in L^2(X,d\mu)$,
one can apply  Calder\'on reproducing formula \eqref{e3.1}
to deduce
\begin{eqnarray}
\label{e5.2}
\int_X f(x)V(x)d\mu(x) &=&\lim_{k\to \infty} \int_X f_k(x)V_k(x)d\mu(x) \notag \\
&=& \lim_{k\to \infty} \lim_{m\to \infty}   \sum_{b=1}^K  \int_X
F_{b,m}(x) V_k(x)d\mu(x),
\end{eqnarray}
where
\[F_{b,m}(x):=\sum_{\substack{ Q\in \mathscr{D} \\
E_2(Q)\in \bigcup_{-m\leq s\leq m}\mathcal{D}_s^{b}}} a_Q(f_k)(x).\]

For $F_{b,m}$ and $V_k$ defined above,  it turns to estimate $\int
|F_{b,m}(x)| |V_k(x)|d\mu(x)$. To this end, we apply  the technique in
\cite[Theorem~11.3]{Wi08}. For every given $b\in \{1,2,\cdots,K\}$ and $m\geq 1$,
set $E_{b,m}^j=\left\{x\in X:\   \mathscr A_L\left(F_{b,m}\right)(x)
>2^j\right\}$ and
\[D_{b,m}^{j}=\left\{E_2(Q):\ \ Q\in \mathscr{D},  E_2(Q)\in \bigcup_{-m\leq s
\leq m}\mathcal{D}_s^{b} , \  \  E_2(Q)\subset E_{b,m}^j,\  E_2(Q)\not\subset
E_{b,m}^{j+1}\right\}, \  \   j\in \mathbb{Z}.\]
Note that function $\mathscr A_L
 \left(F_{b,m}\right)(x)$ just takes finite values on each dyadic cube and so
 can take minimal value on each dyadic cube.
It is easy to see that
\[D_{b,m}^{j}=\left\{E_2(Q):\ \ Q\in \mathscr{D}, E_2(Q)\in \bigcup_{-m\leq s
\leq m}\mathcal{D}_s^{b},   \displaystyle 2^j< \min_{x\in E_2(Q)}\mathscr A_L
 \left(F_{b,m}\right)(x)
\leq 2^{j+1} \right\}, \  \   j\in \mathbb{Z}.\]
Sets $D_{b,m}^{j}$ for different $j$
are disjoint
(some of them might be empty) and
$$
\left\{E_2(Q):\ \ Q\in \mathscr{D}, E_2(Q)\in \bigcup_{-m\leq s
\leq m}\mathcal{D}_s^{b}\right\}=\bigcup_j D_{b,m}^{j}.
$$

Set
\[F_{b,m}^j(x) =\sum_{E_2(Q)\in D_{b,m}^j} a_Q(f_k)(x),\ \  \  j\in \mathbb{Z}.\]
Then $F_{b,m}= \sum_{j} F_{b,m}^j$.

\smallskip

We claim  that $\mathscr A_L(F_{b,m}^j)\leq 2^{j+1}$ everywhere. Indeed,
 if  $\mathscr A_L(F_{b,m}^j)(x_0) > 2^{j+1}$  for
some $x_0\in X$, then there exists $E_2(Q)\in D_{b,m}^j$ such that $x_0
\in E_2(Q)$;  otherwise $F_{b,m}^j(x_0)=0$.
Let $E_2(Q_0)$ be the minimal cube in $\left\{E_2(Q)\in D_{b,m}^j, \
x_0\in E_2(Q)\right\}$, it's clear that $\mathscr A_L(F_{b,m}^j
)(x)=\mathscr A_L (F_{b,m}^j
 )(x_0)>2^{j+1}$ for all  $x\in E_2(Q_0)$. Therefore, $E_2(Q_0)\subset
E_{b,m}^{j+1}$ follows from  $\mathscr A_L\left(F_{b,m}\right)\geq
\mathscr A_L (F_{b,m}^j )$ ,  which is   in conflict with
$E_2(Q_0)\in D_{b,m}^j$.

Set
\[ \mathcal{B}_{b,m}^j=\left\{Q_{b,m}^{j}: \ \ Q_{b,m}^{j} \ \mbox{is the maximal
dyadic cube in}\ D_{b,m}^j\right\}, \]
 It's clear that
any two cubes in $\mathcal{B}_{b,m}^j$ are disjoint.
Besides it is easy to check
 $$
 {\rm{supp}}\ F_{b,m}^j\subset \bigcup_{Q_{b,m}^{j}\in \mathcal{B}_{b,
 m}^j}Q_{b,m}^{j} \subset  E_{b,m}^j.
 $$

\smallskip

Let $M$ be the classical Hardy-Littlewood
maximal operator.
Following R. Fefferman-Pipher's method \cite{FP97}, set
\[\widetilde{V}_k=|V_k|+\f{M(V_k)}{2\|M\|_{L^{p'}(X,d\mu)}}+\f{M\circ M(V_k)}
{(2\|M\|_{L^{p'}(X,d\mu)})^2}+\cdots,\]
where $\|M\|_{L^{p'}(X,d\mu)}$ denotes the operator norm of the Hardy-Littlewood
maximal operator on $L^{p'}(X,d\mu)$. Then $\|\widetilde{V}_k\|_{L^{p'}(X,d\mu)}
\leq 4$ and $\|\widetilde{V}_k\|_{A_1}\leq 2\|M\|_{L^{p'}(X,d\mu)}$, where $A_1$
is the classical Muckenhoupt weight class.

\smallskip

Without  loss of generality, we  assume that $(\widetilde{V}_k)_{Q_{b,m}^{j}}
\neq 0$
for all $Q_{b,m}^{j}\in \mathcal{B}_{b,m}^j$. Note that $\|
\mathscr A_L(F_{b,m}^j)\|_{L^\infty}\leq 2^{j+1}$.
Then one can use
Corollary~\ref{cor4.6} with $p=1$ to see
\begin{eqnarray}\label{e5.3}
\int_X  |F_{b,m}(x)|\   |V_k(x)|d\mu(x) &\leq &
  \sum_{j\in \mathbb{Z}}   \sum_{Q_{b,m}^{j}\in \mathcal{B}_{b,m}^j}
\int_{Q_{b,m}^{j}}  |F_{b,m}^j (x)|\ \widetilde{V}_k(x)
d\mu(x)\nonumber\\
&\leq & C\sum_{j\in \mathbb{Z}}  2^j  \sum_{Q_{b,m}^{j}\in \mathcal{B}_{b,m}^j}
\int_{Q_{b,m}^{j}} \widetilde{V}_k(x) \Bigg(\log
\Bigg(e+\f{\widetilde{V}_k(x)}{(\widetilde{V}_k)_{Q_{b,m}^{j}}}\Bigg)\Bigg)^{1/2}d\mu(x),
\end{eqnarray}

By a similar approach to prove (2.15) in \cite[pp. 17]{Wi08}, we can prove there
exists constant $C>0$ such that
for any dyadic cube $Q$ and function $f$ defined on $Q$,
$$
\int_{Q} |f(x)| \log\Bigg(e+\f{|f(x)|}
{|f|_{Q}}\Bigg)d\mu(x)
\leq C\int_{Q} M\left(f \chi_{Q}
\right) d\mu(x).
$$

By H\"older's inequality,
\begin{eqnarray}\label{e5.4}
& &\int_{Q_{b,m}^{j}} \widetilde{V}_k(x) \Bigg(\log \Bigg(e+\f{
\widetilde{V}_k(x)}{(\widetilde{V}_k)_{Q_{b,m}^{j}}}\Bigg)\Bigg)^{1/2}\, d\mu(x)\nonumber\\
&\leq& \Bigg(\int_{Q_{b,m}^{j}} \widetilde{V}_k(x) d\mu(x)\Bigg)^{1/2}
\Bigg(\int_{Q_{b,m}^{j}} \widetilde{V}_k(x) \log\Bigg(e+\f{\widetilde{V}_k(x)}
{(\widetilde{V}_k)_{Q_{b,m}^{j}}}\Bigg)d\mu(x)\Bigg)^{1/2}\nonumber\\
&\leq & C  \Bigg(\int_{Q_{b,m}^{j}} \widetilde{V}_k(x) d\mu(x)\Bigg)^{1/2}
\Bigg(\int_{Q_{b,m}^{j}} M(\widetilde{V}_k\, \chi_{Q_{b,m}^{j}}
) d\mu(x)\Bigg)^{1/2}\nonumber\\
&\leq & C \|\widetilde{V}_k\|_{A_1}^{1/2} \int_{Q_{b,m}^{j}} \widetilde{V}_k(x)
d\mu(x)\nonumber\\
&\leq & C  \|M\|_{L^{p'}(X,d\mu)}^{1/2} \int_{Q_{b,m}^{j}} \widetilde{V}_k(x)
d\mu(x).
\end{eqnarray}

Therefore,
\begin{eqnarray}\label{e5.5}
& &\int_X  |F_{b,m}(x)|\   |V_k(x)|d\mu(x) \nonumber\\
&\leq & C \|M\|_{L^{p'}(X,d\mu)}^{1/2} \sum_{j\in \mathbb{Z}} 2^j
\sum_{Q_{b,m}^{j}\in \mathcal{B}_{b,m}^j} \int_{Q_{b,m}^{j}}
\widetilde{V}_k(x) d\mu(x)\nonumber\\
&\leq & C \|M\|_{L^{p'}(X,d\mu)}^{1/2} \sum_{j\in \mathbb{Z}}
\sum_{Q_{b,m}^{j}\in \mathcal{B}_{b,m}^j} \int_{Q_{b,m}^{j}}
\mathscr A_L(F_{b,m})(x)\widetilde{V}_k(x) d\mu(x)\nonumber\\
&\leq & C \|M\|_{L^{p'}(X,d\mu)}^{1/2}  \int_X
\mathscr A_L(F_{b,m})(x)\widetilde{V}_k(x)
d\mu(x)\nonumber\\
&\leq & C \|M\|_{L^{p'}(X,d\mu)}^{1/2} \int_X {\mathscr{A}_L}(f_k)(x)
\widetilde{V}_k(x)d\mu(x) \nonumber\\
&\leq & C \|M\|_{L^{p'}(X,d\mu)}^{1/2}  \int_X \left[\sqrt{2}
{\mathscr{A}_L}(f-f_k)(x)+(1+
\sqrt{2}){\mathscr{A}_L}(f)(x)\right]\widetilde{V}_k(x)d\mu(x)\nonumber\\
&\leq &  C \|M\|_{L^{p'}(X,d\mu)}^{1/2} \left( \|{\mathscr{A}_L}(f)\|_{L^p(X,d\mu)} +
\|{\mathscr{A}_L}(f-f_k)\|_{L^p(X,d\mu)} \right).
\end{eqnarray}
By \eqref{e4.9},
$$
\|{\mathscr{A}_L}(f-f_k)\|_{L^p(X,d\mu)} \leq c \|S_{L, 5\delta^{-2}}(f-f_k)\|_{L^p(X,d\mu)}
\leq C_p
\|f-f_k\|_{L^p(X,d\mu)},
$$
where  $c$ be the constant in \eqref{e4.9}. This, together with \eqref{e5.2}, \eqref{e5.3}, \eqref{e5.4}and  \eqref{e5.5},
 yields that  for any $f\in L^p(X,d\mu)$, $2\leq p<\infty$, and $\|V\|_{L^{p'}
(X,d\mu)}=1$,
\begin{eqnarray*}
\left|\int_X f(x)V(x)d\mu(x)\right| &\leq & C   \|M\|_{L^{p'}(X,d\mu)}^{1/2}
\|{\mathscr{A}_L}(f)\|_{L^p(X,d\mu)} +  C_p \|M\|_{L^{p'}(X,d\mu)}^{1/2}  \lim_{k\to \infty}
\|f-f_k\|_{L^p(X,d\mu)}\\
&=& C   \|M\|_{L^{p'}(X,d\mu)}^{1/2} \|{\mathscr{A}_L}(f)\|_{L^p(X,d\mu)}.
\end{eqnarray*}

By \cite[Theorem~1.3]{Hy2012}, we have
\[\|M\|_{L^q(X,d\mu)} \leq C_{\mu} \left(\f{1}{q-1}\right)^{1/q},\
 \   1<q<\infty,\]
where $C_{\mu}$ depends only on the doubling constant of the measure
$\mu$ (see also \cite[pp. 7]{St70}).

Therefore,
\[\|f\|_{L^p(X,d\mu)} \leq C p^{1/2}\  \|{\mathscr{A}_L}(f)\|_{L^p(X,d\mu)},\  \
\text{as}\   \   p\to \infty.\]
This finishes the proof of \eqref{e5.1}, and finish the proof of Proposition~\ref{th5.1}.
\end{proof}  

We now apply Proposition~\ref{th5.1} {and Lemma~\ref{le4.3}}  to obtain the following result readily.

\begin{theorem}
\label{th5.2}
 Assume that $L$ is a
densely-defined operator
on $L^2(X)$
satisfying \textbf{(H1)}, \textbf{(H2)} and \textbf{(H3)}.  The area function  $S_{L,\alpha}$ is
 defined in \eqref{e1.7}.
Then
  {there exists a constant $C=C(n,\alpha)>0$ such that}
\begin{equation}
\label{e5.6}
\|f\|_{L^p(X,d\mu)} \leq C  p^{1/2}\  \|S_{L,\alpha}(f)\|_{L^p(X,d\mu)} \  \
\text{as} \ \   p\to \infty.
\end{equation}
\end{theorem}

\bigskip

\subsection{$L^p$ bounds for spectral multipliers}

Suppose that
  $L$ is  a  nonnegative self-adjoint
     operator  acting  on $L^2({X})$. Let $E(\lambda)$  be the spectral resolution
	  of  $L$. By the spectral theorem, for any bounded Borel function   $F: [0, \infty)\rightarrow {\Bbb C}$,
one can define the operator
\begin{eqnarray}
\label{e5.7}
F(L)=\int_0^{\infty} F(\lambda) dE(\lambda),
\end{eqnarray}
which is  bounded on $L^2(X)$.
In \cite[Theorem 3.1]{DOS}, X.T. Duong, E.M. Ouhabaz and A. Sikora  obtain
  the following    H\"ormander-type
spectral multiplier result for the special case $m=2$.

\begin{proposition}\label{prop5.3}
 Assume that $L$ is a
densely-defined operator
on $L^2(X)$
satisfying \textbf{(H1)}  and \textbf{(H2)}.
 Let $\beta>{n/2}$ and assume that for any $R>0$ and all Borel functions $F$  such that\, {\rm supp} $F\subseteq [0, R]$,
 \begin{eqnarray}\label{e5.8}
\int_X |K_{F(\sqrt{L})}(x,y)|^2 d\mu(x) \leq {C\over \mu(B(y, R^{-1}))} \|\delta_R F\|^2_{L^q(\mathbb R)}
\end{eqnarray}
for some $q\in [2, \infty].$ Next suppose that $F: [0, \infty)\rightarrow \mathbb{C}$ is a bounded Borel function such that
$F(0)=0,\ \sup_{t>0}\|\eta\delta_tF\|_{W^q_\beta(\RR)}<\infty$
where $\delta_t F(\lambda)=F(t\lambda)$,
   $\| F \|_{W^q_\beta(\mathbb R)}=\|(I-d^2/d x^2)^{\beta/2}F\|_{L^q(\mathbb R)}$
   and $\eta \in {C}_0^{\infty}({\mathbb R}_+)$
   is a fixed function, not identically zero.   Then
 the operator $F(L)$ is bounded on $L^p(X)$ for all $1<p<\infty$.
\end{proposition}

It should be noted that  Gaussian bounds \textbf{(H2)} implies estimates (\ref{e5.8}) for $q=\infty$.
 Hypothesis (\ref{e5.8})
is called   the Plancherel estimate  or the Plancherel
condition. For the standard Laplace operator on Euclidean spaces $\mathbb R^n$,
this is equivalent to
$(1,2)$ Stein--Tomas restriction theorem (which is also the Plancherel estimate
of the Fourier transform) (see \cite{DOS, COSY}).

In this section we will apply the result from the previous section in order to
derive certain estimates on the $L^p$ operator norm $p\to \infty$ of spectral multipliers
$F(L)$.

\begin{theorem}\label{th5.4}
	Under the assumptions of Proposition~\ref{prop5.3}
	and the assumption in addition that $L$  satisfies  \textbf{(H3)},   the operator $F(L)$ satisfies
 \begin{eqnarray}\label{e5.9}
 \|F(L)f\|_p \leq C p\sup_{t>0}\|\eta\delta_tF\|_{W^q_\beta(\RR)} \|f\|_p,
 \end{eqnarray}
as $p\to \infty.$
\end{theorem}

We will break up the proof into two parts in which we deal with
the ${\mathcal G}_{\lambda}^{\ast}$ function:
Let $\phi\in C^{\infty}_0(\mathbb R)$ be
even function with $\int \phi =1$, $\mbox{supp}\,\phi \subset (-1/10, 1/10)$.
Let $\varphi$ denote the Fourier transform of
$\phi$ and let $\psi(s)=s^{2n+2}\varphi^3(s)$.
We define the ${\mathcal G}_{\lambda}^{\ast}$ function by

$${\mathcal G}_{\lambda}^{\ast}(f)(x)=\left(\int^\infty_0\int_{X}\Big({t\over t+ d(x,y)}\Big)^{n\lambda}
|\psi(t\sqrt{L})f(y)|^2{d\mu(y) \over \mu(B(x,t))  }{ dt \over t }\right)^{1/2}, \ \ \ \lambda>1.
$$
It is known that under the assumptions \textbf{(H1)} and \textbf{(H2)} of the operator
$L$,   ${\mathcal G}_{\lambda}^{\ast}(f)$ is bounded on $L^p(X)$ for all $1<p<\infty$.
Following the method of  Rubio de Francia and Garc\'ia-Cuerva (see pp. 356 -
357, \cite{FP97}), we have
\begin{eqnarray}\label{e5.10}
\|{\mathcal G}_{\lambda}^{\ast}(f)\|_{L^p(X)}\leq Cp^{1/2}\|f\|_{L^p(X)}.
\end{eqnarray}
as $p\rightarrow \infty$. See \cite[Theorems 1.2 and  3.3]{GX} and \cite[Theorem 1.3]{GY}.

 \begin{proposition}\label{prop5.5}
 	 Under the assumptions of Proposition~\ref{prop5.3},
  there
exist a constant $C=C_{n,\lambda,\Psi}$ and some $\lambda>1$ such that  for every $f\in L^2(X)$,
\begin{eqnarray}\label{e5.11}
S_L (F(L)f)(x) \leq C\sup_{t>0}\|\eta\delta_tF\|_{W^q_\beta(\RR)} {\mathcal G}_{\lambda}^{\ast}( f)(x).
\end{eqnarray}
\end{proposition}

\begin{proof}
By the spectral theory (\cite{Yo}),
for every $f\in L^2(X)$ and every $\kappa >3n$,
\begin{eqnarray*}
f =C_\psi\int^\infty_0
(t^2L)^{\kappa}e^{-t \sqrt{L}}\psi(t\sqrt{L}) f  {dt\over t}
\end{eqnarray*}
with $C^{-1}_{\psi}=\int^\infty_0 t^{2\kappa}
 e^{-t} \psi(t ){dt/t}$, and the integral converges in $L^2(X)$.
\begin{eqnarray*}
  s^2L e^{-s \sqrt{L}}F(L)f(y)
&=&  C_{\psi}  \int^{\infty}_{0} \int_{X} {s^2t^{2\kappa}\over
(s^2+t^2)^{\kappa+1}}K_{(s^2+t^2)^{\kappa+1} L^{\kappa+1}e^{-(s+t){\sqrt L}}F(L)}(y, z)
 \psi(t\sqrt{L})f (z){dt\, d\mu(z)\over t}\\
&\leq&C A\cdot B,
\end{eqnarray*}
where
\begin{eqnarray*}
A^2 &= &   \int^{\infty}_0 \int_{X}|\psi(t\SL)f(z)|^2\left(1+{d(y,z)\over s+t}\right)^{-2\beta_0} {
 {s^2t^{2\kappa}\over
(s^2+t^2)^{\kappa+1}}} {d\mu(z)dt\over t   \mu(B(y, s+t))}
\end{eqnarray*}
and
\begin{eqnarray*}
B^2
&=&  \int^{\infty}_0 \int_{X}
 {s^2t^{2\kappa}\over
(s^2+t^2)^{\kappa+1} \mu(B(y, s+t))}  \left(1+{d(y,z)\over s+t}\right)^{2\beta_0}
|K_{(s^2+t^2)^{\kappa+1} L^{\kappa+1}e^{-(s+t)\sqrt{L}}F(L)}(y, z)|^2
{d\mu(z)dt\over t}\\  &\leq& C\sup_{t>0}\|\eta\delta_tF\|^2_{W^q_\beta(\RR)},
\end{eqnarray*}
where we used the fact (see \cite[Lemma 4.6]{GY2}, \cite[Lemma 4.3]{DOS}) that  for any $n/2<\beta_0<\beta$,
 \begin{eqnarray*}
 \int_{X}\Big|\left(1+{d(y,z)\over t}\right)^{\beta_0}K_{(t^2L)^{\kappa}e^{-t\sqrt{L}}F(L)}(y,z)\Big|^2d\mu(z)
 \leq {C\over \mu(B(y,t)) }\sup_{t>0}\|\eta\delta_tF\|^2_{W^q_\beta(\RR)}.
 \end{eqnarray*}
Therefore, we put it  into the definition of $S_L$ to obtain
\begin{eqnarray*}
& & S^2_L(F(L)f)(x)\\
& = &\int^{\infty}_0\int_{d(x,y)<s}|s^2L e^{-s\sqrt{L}}F(L)f(y)|^2{d\mu(y)ds\over s \mu(B(x, s))}\\
&\leq& C\sup_{t>0}\|\eta\delta_tF\|^2_{W^q_\beta(\RR)} \int_0^{\infty}\int_X |\psi(t\SL)f(z)|^2\\
&&\hspace{0.2cm}\times\Bigg(\int^{\infty}_0
\int_{d(x,y)<s}\left(1+{d(y,z)\over s+t}\right)^{-\beta_0} {
 {s^2t^{2\kappa}\over
(s^2+t^2)^{\kappa+1}}} \left(\mu(B(x, t))\over \mu(B(x, s)) \mu(B(y, s+t))\right)
{d\mu(y)ds \over s}\Bigg)
{d\mu(z)dt\over t \mu(B(x,t))}.
\end{eqnarray*}
We will be done if we show that
\begin{eqnarray}\label{e5.12}\hspace{1cm}
& & \int^{\infty}_0
\int_{d(x,y)<s}\left(1+{d(y,z)\over s+t}\right)^{-2\beta_0} {
 {s^2t^{2\kappa}\over
(s^2+t^2)^{\kappa+1}}} \left(\mu(B(x, t))\over \mu(B(x, s)) \mu(B(y, s+t))\right)
{d\mu(y)ds \over s}\notag\\
& \leq&  C\left({t\over t+d(x,z)}\right)^{n\lambda}
\end{eqnarray}
for some $\lambda>1$. Then we will prove estimate (\ref{e5.12}) by considering the following two cases.

\medskip

\noindent
{\it Case (1).}  $d(x, z)\leq t.$ \ In this case,  observe that
if $d(x,y)\leq s$, then $\mu(B(x, s+t))\leq C \mu(B(y, s+t))$. From this, we know that
\begin{eqnarray*}
{\rm LHS\ of \ (\ref{e5.12})}
 \leq C \int^{\infty}_0  {s^2t^{2\kappa}\over
(s^2+t^2)^{\kappa+1}}
{ ds \over s } \leq C.
\end{eqnarray*}

\noindent But $d(x, z)\leq t$, so

$$\left({t\over t+d(x, z)}\right)^{n\lambda}\geq C_{n,\lambda}.$$

\noindent This implies that   \eqref{e5.12}  holds  when $d(x, z)\leq t.$

\medskip

\noindent
{\it Case (2).}   $d(x, z)> t$. In this  case, we break the integral into two pieces:

\begin{eqnarray*}
 \int^{d(x, z)/2}_0\int_{d(x, y)<s}\cdots + \int^{\infty}_{d(x, z)/2}\int_{d(x,y)<s}\cdots =:I+{\it II}.
\end{eqnarray*}

\noindent For the first term, note that $d(x,y)<s<d(x, z)/2$, and so $d(y, z)\geq d(x, z)-d(y,x)>d(x, z)/2.$ This yields

\begin{eqnarray*}
I&\leq& \int^{d(x, z)/2}_0\int_{d(x, y)<s}\left(1+{d(y,z)\over s+t}\right)^{-2\beta_0} {
 {s^2t^{2\kappa}\over
(s^2+t^2)^{\kappa+1}}} \left(\mu(B(x, t))\over \mu(B(x, s)) \mu(B(y, s+t))\right)
{d\mu(y)ds \over s}\\
&\leq&C\int^{\infty}_0
 \left({s+t\over d(x,z)}\right)^{2\beta_0}  {
 {s^2t^{2\kappa}\over
(s^2+t^2)^{\kappa+1}}}{ ds \over s}\\
&\leq&C\left({t\over d(x,z)}\right)^{n\lambda}
\end{eqnarray*}
with $\lambda=2\beta_0/n>1$. For the term $\it{II}$, we have that $s>d(x,z)/2>t/2$ and $d(y, z)\leq d(x,y)+d(x,z)\leq 3s$.
 It yields that

\begin{eqnarray*}
\it{II}
&\leq&C \int_{d(x,z)/2}^{\infty}\int_{d(x, y)<s}\left(1+{d(y,z)\over s+t}\right)^{-2\beta_0} {
 {s^2t^{2\kappa}\over
(s^2+t^2)^{\kappa+1}}} \left(\mu(B(x, t))\over \mu(B(x, s)) \mu(B(y, s+t))\right)
{d\mu(y)ds \over s}\\
&\leq&C\left({t\over d(x,z)}\right)^{2\kappa} \int^{\infty}_0 {s^2 \over
(s^2+t^2) }
{ ds \over s }\\\\
&\leq&C\left({t\over d(x,z)}\right)^{n\lambda}
\end{eqnarray*}
with $\lambda=2\kappa/n>1$.
From the above {\it Cases (1)} and {\it  (2)}, we have obtained estimate (\ref{e5.12}), and then   the proof of
estimate \eqref{e5.11} is complete.
\end{proof}

\medskip

\noindent
{\it Proof of Theorem~\ref{th5.4}.}   Theorem~\ref{th5.2} and Proposition~\ref{prop5.5}, in combination with
\eqref{e5.10}, yield  that
\begin{eqnarray*}
\|F(L)f\|_p \leq  C p^{1/2}\|S_L(F(L)f)\|_p
 &\leq&  C p^{1/2} \sup_{t>0}\|\eta\delta_tF\|_{W^q_\beta(\RR)}\|{\mathcal G}_{\lambda}^{\ast}( f)\|_p\\
&\leq&  Cp \sup_{t>0}\|\eta\delta_tF\|_{W^q_\beta(\RR)}\|f\|_p
\end{eqnarray*}
as $p\to \infty.$  This finishes the proof of Theorem~\ref{th5.4}.
 \hfill{}$\Box$

\medskip
\noindent
{\bf Remark}.  Let $M$ be a complete non-compact connected Riemannian manifold, $\mu$ the Riemannian
measure, $\nabla$ the Riemannian gradient. One defines $\Delta$, the Laplace-Beltrami operator, as a
self-adjoint positive operator on $L^2(M)$.
Assume that $M$ satisfies
the volume doubling condition \eqref{e2.1} and
  the (scaled) Poincare inequalities if there exists $C>0$ such that, for
every ball $B=B(x,r)$  and every $f$ with $f, \nabla f$ locally square integrable,
\begin{eqnarray*}
\int_B|f-f_B|^2 d\mu\leq C r^2\int_B|\nabla f|^2 d\mu.
\end{eqnarray*}
It is known (see \cite[Theorem 1.3]{ACDH}) that if
   the gradient of kernel $p_{t}(x,y)$ of the semigroup $e^{-t\Delta}$ satisfies
$$
	\big|\nabla p_{t}(x,y)\big|\leq C {C\over \sqrt{t} {\mu(B(x,\sqrt{t}))}} \exp\left(-c{d^2(x,y)\over t}\right), \ \ \ \
	\forall t>0, x, y\in M,
$$
then the  Riesz transform
 $
 \nabla \Delta^{-1/2}
 $
 is bounded on $L^p(M)$ for every $1<p<\infty.$ Applying  the method in the proof of  Theorem~\ref{th5.4},
 we can  show  that
 \begin{eqnarray}\label{e8.1}
 \| \nabla \Delta^{-1/2}f\|_p \leq C p  \|f\|_p,
 \end{eqnarray}
as $p\to \infty$. For the detail, we refer it to \cite{Wu}.

\medskip

\subsection{Two-weighted norm inequalities for the square function}
Next we apply Corollary~\ref{cor4.6} to  show the following analogous result to Theorem~11.3 in \cite{Wi08}.

\begin{theorem}
\label{th5.6}
 Assume that $L$ is a
densely-defined operator
on $L^2(X)$
satisfying \textbf{(H1)}, \textbf{(H2)} and \textbf{(H3)}.
Let $0<p<\infty$, $\tau>p/2$, and let $V$ and $W$ be two weights such that
\begin{equation}
\label{e5.13}
\int_Q V(x)\left(\log \left(e+\f{V(x)}{V_Q}\right)\right)^{\tau} d\mu(x)\leq \int_Q W(x)d\mu(x)
\end{equation}
for all $Q\in \bigcup_{b=1}^K \mathcal{D}^b$. {Then  there exists a positive constant $C=C(p,n,\tau,\rho)$  such that
for every $f\in L_0^\infty(X)$,}
\begin{equation}
\label{e5.14}
\int_X |f(x)|^p V(x)d\mu(x)\leq C\int_X   \big( \mathscr{A}_L (f)(x)\big)^p W(x)d\mu(x).
\end{equation}
As a consequence, {for every $\alpha>0$, there exists a positive constant $C=C(p,n,\alpha)$
  such that for every  $f\in L_0^\infty(X)$,}
\begin{equation}
\label{e5.15}
\int_X |f(x)|^p V(x)d\mu(x)\leq C\int_X   \big( S_{L,\alpha} (f)(x)\big)^p W(x)d\mu(x).
\end{equation}
\end{theorem}

\begin{proof}
Let $F_{b,m}(x)=\sum_{
E_2(Q)\in \bigcup_{-m\leq s\leq m}\mathcal{D}_s^{b}} a_Q(f)(x)$, it
follows from the proof of Theorem~\ref{th4.7} that there exists $\{m_i\}$ such that
\begin{eqnarray}
\label{e5.16}
	\int_X |f(x)|^p V(x)d\mu(x)
	&\leq&  \underset{i\to \infty}{\underline{\lim}} \int_X
	\Big|\sum\nolimits_{b=1}^K  F_{b,m_i}(x)\Big|^p V(x)d\mu(x)  \notag\\
	&\leq & C(K,p) \sum\nolimits_{b=1}^K  \underset{i\to \infty}{\underline{\lim}}
	\int_X |F_{b,m_i}(x)|^p V(x)d\mu(x).
\end{eqnarray}
This, together with  Corollary~\ref{cor4.6}, the condition \eqref{e5.13} and the
 similar argument to that of Theorem~11.3 in \cite{Wi08}, deduces that for
 every $0<p<\infty$,  there exists a constant $C=C(p,n,\tau,\rho)$ such that
\[\int_X |F_{b,m}(x)|^p V(x)d\mu(x) \leq C \int_X
\left(\mathscr{A}_L(F_{b,m})(x)\right)^p W(x)d\mu(x)\]
for every $b\in \{1,2,\cdots, K\}$ and $m>0$.

Notice that $\mathscr{A}_L (F_{b,m})(x)\leq \mathscr{A}_L(f)(x)$, a.e. $x\in X$,
 arguments above implies that \eqref{e5.14} holds. {This, in combination with Lemma~\ref{le4.3}, deduces \eqref{e5.15} readily.} The proof of Theorem~\ref{th5.6} is end.
\end{proof}

{For every $P\in \mathscr{D}$, we denote by $\ell(P)=\delta^{k_P}$ where $k_P$ is the generational label of $P$ in $\mathscr{D}$, that is, $P\in \mathscr{D}_{k_P}$. }
Corresponding to definitions \eqref{e4.66} and \eqref{e4.7}, for a given $f\in L^2(X)$
 and $f_\mathcal{F}=\underset{E_2(P)\in \mathcal{F}}{\sum}  a_P(f)$ with $\mathcal{F}\subset \mathcal{D}^b$
 for some $b=1, 2, \cdots, K$, define
 \begin{eqnarray}
\label{e5.17}
\widetilde{{\mathscr A}}_L (f_\mathcal{F})(x):=\Bigg({\sum_{   P\in \mathscr{D},\, E_2(P)\in \mathcal{F}  }
\f{|\lambda_{P }|^2}{(\ell(P))^2 \mu(P )}\chi_{E_2(P)}(x)}\Bigg)^{1/2},
\end{eqnarray}
and
\begin{eqnarray}
\label{e5.18}
\widetilde{{\mathscr A}}_L (f )(x):=\Bigg({\sum_{   P\in \mathscr{D} }
\f{|\lambda_{P }|^2}{(\ell(P))^2\mu(P )}\chi_{E_2(P)}(x)}\Bigg)^{1/2}.
\end{eqnarray}

The following result is a straightforward result from Corollary~\ref{cor4.6}.

\begin{lemma}
\label{le5.7}
 Assume that $L$ is a
densely-defined operator
on $L^2(X)$
satisfying \textbf{(H1)}, \textbf{(H2)} and \textbf{(H3)}.
For $f_{\mathcal{G}}$ and ${Q}{\in \mathcal{D}^b}$    defined in
Theorem~\ref{prop4.5}. Assume $V$ is a weight, $\int_{Q} V(x)d\mu(x)\neq  0$,
and $0<p<\infty$. {Then there exists a positive constant $C=C(p,n,\rho)$ independent of $\mathcal{G}$, $K_\mathcal{G}$, $Q$, $f$ and $V$
  such that}
\[\int_{Q}  |f_\mathcal{G}(x)|^p V(x)d\mu(x) \leq C \, {(\ell(Q))^p}   \,  \big\|\widetilde{\mathscr {A}}_{L}
(f_{\mathcal{G}})\big\|_{L^\infty
(Q)}^p       \int_{Q} V(x)
\left(\log\left(e +\f{V(x)}{V_{Q}}\right)\right)^{p/2}d\mu(x).\]
\end{lemma}
\begin{proof}
We  write
\[\int_{Q}  |f_\mathcal{G}(x)|^p V(x)d\mu(x) =(\ell(Q))^p
 \int_{Q}  \left|\f{f_\mathcal{G}(x)}{\ell(Q)}\right|^p V(x)d\mu(x). \]
{For every $E_2(P)\in \mathcal{G}$ with $P\in \mathscr{D}$, one can apply Proposition~\ref{prop4.2} and  $E_2(P)\subset Q$ to see    $\ell(P)\leq \ell(Q)$. Then }the fact $\mathscr{A}_L(C f_{\mathcal{G}})=|C| \mathscr{A}_L(f_\mathcal{G})$, $C\in \mathbb{R}$,  implies that
\begin{eqnarray*}
\mathscr{A}_L\Big(\f{f_\mathcal{G}(x)}{\ell(Q)}\Big) (x)
&=&\Bigg(\sum_{
 E_2(P)\in \mathcal{G} }   \f{|\lambda_P|^2}{ (\ell(Q))^2 \mu(P)}\chi_{E_2(P)}(x)\Bigg)^{1/2} \\
&\leq &  \Bigg(\sum_{
 E_2(P)\in \mathcal{G}}  \f{|\lambda_P|^2}{(\ell(P))^2 \mu(P)}\chi_{E_2(P)}(x)\Bigg)^{1/2}\\
&=& \widetilde{\mathscr{A}}_L(f_\mathcal{G})(x).
\end{eqnarray*}
Then Lemma~\ref{le5.7} follows from Corollary~\ref{cor4.6}.
\end{proof}

\smallskip

Next we aim to prove  the following  two-weighted norm inequalities,
which is analogous  to Theorems~4.1 and 4.2 in \cite{Wi}. {In the sequel, for every $b\in \{1,2,\cdots, K\}$, similarly as before,  {for every $Q\in \mathcal{D}^b$, we denote by $\ell(Q)=\delta^{k_Q}$ where $k_Q$ is the generational label of $Q$ in $\mathcal{D}^b$, that is, $Q\in \mathcal{D}^b_{k_Q}$. }}

\begin{proposition}
\label{th5.8}
 Assume that $L$ is a
densely-defined operator
on $L^2(X)$
satisfying \textbf{(H1)}, \textbf{(H2)} and \textbf{(H3)}.
Let $0<p\leq 2$, $\tau>p/2$, and let $V$ and $W$ be two weights such that
\begin{equation}
\label{e5.19}
\left(\ell(Q)\right)^p \int_Q V(x)
\left(\log \left(e+\f{V(x)}{V_Q}\right)\right)^{\tau} d\mu(x)\leq \int_Q W(x)d\mu(x)
\end{equation}
for all $Q\in \bigcup_{b=1}^K \mathcal{D}^b$.
 There exists a constant $C=C(p,n,\tau,\rho)$, such that for every $f\in L_0^\infty(X)$,
\begin{equation}
\label{e5.20}
\int_X |f(x)|^p V(x)d\mu(x)\leq C\int_X   \left( \widetilde{\mathscr{A}}_L (f)(x)\right)^p W(x)d\mu(x).
\end{equation}
\end{proposition}

We will break up the proof into two cases: $0<p\leq 1$, and $1<p\leq 2$. When $1<p\leq 1$,
the proof is based  on a modification of the proof of Theorem~11.3 in \cite{Wi08}.
When $1<p\leq 2$, the argument is
 more technical, we shall combine techniques from proofs of \cite[Theorem 2.7]{Wi},
 \cite[Theorem~4.1]{Wi} and \cite[Theorem~11.3]{Wi08}, and  a crucial ingredient for
 the proof is Corollary~\ref{cor4.6}. In the sequel, for every function $V\in L^1_{\rm loc}(X)$
 and set $E$, we write $V(E)=\int_E V(x)  d\mu(x).$

\medskip

\noindent
{\it Proof of Proposition~\ref{th5.8} for $0<p\leq 1$.}    For every $b\in \{1,2,\cdots, K\}$
and $m\geq 1$, define $F_{b,m}=\underset{
E_2(P)\in \bigcup_{-m\leq s\leq m}\mathcal{D}_s^{b}}{\sum} a_P(f)(x)$ as we did in
the proof of Theorem~\ref{th5.6}. With some notations  as in  the proof
of Theorem~\ref{th5.1}, for $ j\in \mathbb{Z}$ we set
 \[\widetilde{E}_{b,m}^j=\left\{x\in X:\
\widetilde{ \mathscr A}_L\left(F_{b,m}\right)(x)
>2^j\right\},\]
\[\widetilde{D}_{b,m}^{j}=\left\{E_2(Q):\ \ Q\in \mathscr{D},  E_2(Q)\in \bigcup_{-m\leq s
\leq m}\mathcal{D}_s^{b} , \  \  E_2(Q)\subset \widetilde{E}_{b,m}^j,\  E_2(Q)\not\subset
\widetilde{E}_{b,m}^{j+1}\right\},\]
\[\widetilde{F}_{b,m}^j(x) =\sum_{E_2(Q)\in \widetilde{D}_{b,m}^j} a_Q(f)(x),\ \  \  j\in \mathbb{Z},\]
and
\[ \widetilde{\mathcal{B}}_{b,m}^j=\left\{Q_{b,m}^{j}: \ \ Q_{b,m}^{j} \ \mbox{is the maximal
dyadic cube in}\  \widetilde{D}_{b,m}^j\right\}.\]

From the   discussion in the proof of Theorem~\ref{th5.1}, we have the following conclusions:
\begin{enumerate}[(1)]
  \item Sets $\widetilde{D}_{b,m}^{j}$ for different $j$
are disjoint
(some of them might be empty) and $\bigcup_{-m\leq s \leq m}\mathcal{D}_s^{b}=\cup_j \widetilde{D}_{b,m}^{j}$.
  \item $F_{b,m}= \sum_{j} \widetilde{F}_{b,m}^j$.
  \item $\widetilde{\mathscr{A}}_L(F_{b,m}^j)\leq 2^{j+1}$ everywhere.
  \item Any two cubes in $\widetilde{\mathcal{B}}_{b,m}^j$ are disjoint.
Besides it is easy to check
 $$
 {\rm{supp}}\ \widetilde{F}_{b,m}^j\subset \bigcup_{Q_{b,m}^{j}\in \widetilde{\mathcal{B}}_{b,
 m}^j}Q_{b,m}^{j} \subset  \widetilde{E}_{b,m}^j.
 $$
\end{enumerate}

Without  loss of generality, we  assume that $V_{Q_{b,m}^{j}}
\neq 0$
for all $Q_{b,m}^{j}\in \widetilde{\mathcal{B}}_{b,m}^j$. For $0<p\leq 1$,
\begin{eqnarray}
\label{e5.21}
\int_X  \left|F_{b,m}(x)\right|^p     V(x)d\mu(x) &\leq &
  \sum_{j\in \mathbb{Z}}   \sum_{Q_{b,m}^{j}\in \widetilde{\mathcal{B}}_{b,m}^j}
\int_{Q_{b,m}^{j}}  \left| \widetilde{F}_{b,m}^j (x)\right|^p \ V(x)
d\mu(x) \notag\\
&\leq & \sum_{j\in \mathbb{Z}}   \sum_{Q_{b,m}^{j}\in \widetilde{\mathcal{B}}_{b,m}^j}
(V(Q_{b,m}^{j}))^{1-p/(2\tau)} \Bigg(\int_{Q_{b,m}^{j}}
|\widetilde{F}_{b,m}^j (x)|^{2\tau} \ V(x)
d\mu(x) \Bigg)^{p/(2\tau)}.
\end{eqnarray}

Notice that  $
\big\|\widetilde{\mathscr A}_L (F_{b,m}^j)\big\|_{L^\infty}\leq 2^{j+1}$, one can apply
Lemma~\ref{le5.7} and \eqref{e5.19}  to see
\begin{eqnarray}
\label{e5.22}
\int_{Q_{b,m}^{j}} \left|\widetilde{F}_{b,m}^j (x)\right|^{2\tau} \ V(x)
d\mu(x)
&\leq  & C 2^{2\tau j} {\big(\ell({Q_{b,m}^{j}})\big)^{2\tau}}
 \int_{Q_{b,m}^{j}} V(x)
 \Bigg(\log\Bigg(e +\f{V(x)}{V_{Q_{b,m}^{j}}}\Bigg)\Bigg)^{\tau}d\mu(x) \notag \\
&\leq & C 2^{2\tau j} {\big(\ell({Q_{b,m}^{j}})\big)^{2\tau -p} } W(Q_{b,m}^{j}).
\end{eqnarray}
Put it into \eqref{e5.21}, we obtain
\begin{eqnarray*}
\int_X  \left|F_{b,m}(x)\right|^p     V(x)d\mu(x) &\leq &
C  \sum_{j\in \mathbb{Z}}  2^{pj} \sum_{Q_{b,m}^{j}\in \widetilde{\mathcal{B}}_{b,m}^j}
\left[\big[\ell({Q_{b,m}^{j}})\big]^p  V(Q_{b,m}^{j})\right]^{1-p/(2\tau)}
\left(W(Q_{b,m}^{j})\right)^{p/{(2\tau)}}\\
&\leq &  C  \sum_{j\in \mathbb{Z}}  2^{pj} \sum_{Q_{b,m}^{j}\in \widetilde{\mathcal{B}}_{b,m}^j}
  W(Q_{b,m}^{j})\\
&= &  C  \sum_{j\in \mathbb{Z}}  2^{pj} \sum_{Q_{b,m}^{j}\in \widetilde{\mathcal{B}}_{b,m}^j}
 \int_{Q_{b,m}^{j}} W(x)d\mu(x)\\
 &= &  C  \sum_{j\in \mathbb{Z}}   \sum_{Q_{b,m}^{j}\in \widetilde{\mathcal{B}}_{b,m}^j}
 \int_{Q_{b,m}^{j}}\left(\widetilde{\mathscr{A}}_L (F_{b,m})(x)\right)^p W(x)d\mu(x)\\
&\leq &C\int_X   \left(\widetilde{\mathscr{A}}_L (F_{b,m})(x)\right)^p W(x)d\mu(x),
\end{eqnarray*}
where the second inequality follows from the fact that \eqref{e5.19} implies
$\big[\ell({Q_{b,m}^{j}})\big]^p  V(Q_{b,m}^{j}) \leq W(Q_{b,m}^{j})$.

This, together with \eqref{e5.16}, implies that  Proposition~\ref{th5.8} holds for $0<p\leq 1$.
 \hfill{}$\Box$

\bigskip

To show Proposition~\ref{th5.8} for  $1<p\leq 2$, we need the following lemma.

\begin{lemma}
\label{le5.9}
Let $0<p<\infty$, $\tau>p/2$, and let $V$ be a weight.  Assume
 $\mathcal{F}\subset \mathcal{D}^b$ for some $b\in \{1,2,\cdots, K\}$,
 $\#\mathcal{F}<\infty$ and  there exists some ${k\in \mathbb{N}} $ such that
\begin{equation}
\label{e5.23}
2^{k }\leq \f{1}{\int_{Q} V(x)d\mu(x)} \int_{Q}
 V(x)\Bigg(\log\Bigg(e+\f{V(x)}{V_{Q}}\Bigg)\Bigg)^{\tau} d\mu(x)<2^{k +1},\  \   \forall
  Q\in \mathcal{F}.
\end{equation}
For any $f\in L_0^\infty(X)$, define $f_\mathcal{F}
=\underset{E_2(P)\in \mathcal{F}}{\sum} a_P (f)$. Then there exists a positive constant
 $C(p,n,\tau,\rho)$ independent of $\mathcal{F}$ and $f$, such that
\begin{equation}
\label{e5.24}
\int_X \left|f_{\mathcal{F}}(x)\right|^p V(x)d\mu(x)
\leq C 2^{[p /(2\tau)]k } \int_X \left|  \mathscr{A}_L(f_{\mathcal{F}})(x)\right|^p V(x)d\mu(x).
\end{equation}
\end{lemma}
\begin{proof}
Similar to notations in the proof of Proposition~\ref{th5.1}, set
 \[E_j=\left\{x\in X:\   \mathscr A _L\left(f_\mathcal{F}\right)(x)
>2^j\right\},\]
\[D_{j}=\left\{E_2(Q)\in \mathcal{F}:\ \   E_2(Q)\subset E_j,\  E_2(Q)\not\subset
E_{j+1}\right\}, \  \   j\in \mathbb{Z},\]
\[f_{\mathcal{F},\, j}(x) =\sum_{E_2(Q)\in D_j} a_Q(f)(x),\ \  \  j\in \mathbb{Z},\]
and
\[ \mathcal{B}_j=\left\{E_2(Q): \ \ E_2(Q) \ \mbox{is the maximal
dyadic cube in}\  D_j\right\}.\]

A key observation is that using a similar argument to that in \cite[pp. 192]{Wi08},
 we obtain $f_{\mathcal{F}}(x)=\underset{i:\ i\leq j}{\sum} f_{\mathcal{F},\, i}(x)$
  when $x\in E_j\setminus E_{j+1}$.

Let $\varepsilon>0$ to be chosen later. Recall that
$\{x\in X:\  f_{\mathcal{F}}(x)\neq 0\}\subset \bigcup_j E_j$, we follow \cite[(11.4)]{Wi08} to see
\begin{eqnarray}
\label{e5.25}
\int_{X} \left|f_{\mathcal{F}}(x)\right|^p V(x)d\mu(x) & =&\sum_{j\in \mathbb{Z}}
\int_{E_j\setminus E_{j+1} }  \left|f_{\mathcal{F}}(x)\right|^p V(x)d\mu(x)  \notag\\
&= &\sum_{j\in \mathbb{Z}} \int_{E_j\setminus E_{j+1} }
\left|\underset{i:\ i\leq j}{\sum} f_{\mathcal{F},\,i}(x)\right|^p V(x)d\mu(x) \notag\\
&\leq &  C  \sum_{j\in \mathbb{Z}} \int_{E_j\setminus E_{j+1} }
\sum_{i:\  i\leq j}  2^{\varepsilon(j-i)} \left|f_{\mathcal{F},\,i}(x)\right|^p V(x)d\mu(x)  \notag\\
&\leq & C   \sum_{i,j:\  i\leq j}  2^{\varepsilon(j-i)}  \int_{E_j}
\left|f_{\mathcal{F},\, i}(x)\right|^p V(x)d\mu(x) \notag\\
&\leq & C   \sum_{i,j:\  i\leq j}  2^{\varepsilon(j-i)}
 \left(V(E_j)\right)^{1-p/(2\tau)} \left(\int_{E_i}
 \left|f_{\mathcal{F},\,i}(x)\right|^{2\tau}  V(x)d\mu(x) \right)^{p/(2\tau)},
\end{eqnarray}
where the  last inequality follows from H\"older's inequality and $E_j \subset E_i$ for $i\leq j$.

Notice that $Q_{i}\in \mathcal{F}$ for every $Q_{i}\in  \mathcal{B}_i$.
  It follows from Corollary~\ref{cor4.6} and \eqref{e5.23}  that
\begin{eqnarray*}
\int_{E_i}\left|f_{\mathcal{F},\, k}(x)\right|^{2\tau}  V(x)d\mu(x)&=&  \sum_{Q_{ i}\in  \mathcal{B} _i}
\int_{Q_{ i}}\left|f_{\mathcal{F},\, i}(x)\right|^{2\tau}  V(x)d\mu(x)\\
 &\leq&  C 2^{2i\tau} \sum_{Q_{ i}\in \mathcal{B}_i}\int_{Q_{ i}}
 V(x)\left(\log\left(e +\f{V(x)}{V_{Q_{ i}}}\right)\right)^{\tau}d\mu(x)\\
 &\leq  & C 2^{2i\tau} 2^{k } \sum_{Q_{ i}\in  \mathcal{B}_i}  V(Q_{ i})\\
 &\leq & C 2^{2i\tau} 2^{k } V(E_k).
\end{eqnarray*}

Put it into \eqref{e5.25}, and we follow estimates in \cite[(11.5)]{Wi08} to obtain
\begin{eqnarray*}
\int_{X} \left|f_{\mathcal{F}}(x)\right|^p V(x)d\mu(x) & \leq &
C   \sum_{i,j:\  i\leq j}  2^{\varepsilon(j-i)}
\left(V(E_j)\right)^{1-p/(2\tau)}  2^{ip} 2^{p k /(2\tau)} \left(V(E_k)\right)^{p/(2\tau)} \\
&\leq & C  2^{p k /(2\tau)} \sum_{i,j:\  i\leq j}
2^{-\left\{p[1-p/(2\tau)]-\varepsilon\right\}(j-i)}
\left[2^{jp }V(E_j)\right]^{1-p/(2\tau)} \left[2^{ip} V(E_k)\right]^{p/(2\tau)} \\
&\leq &  C  2^{p k /(2\tau)} \sum_{i,j:\  i\leq j}
2^{-\{p[1-p/(2\tau)]-\varepsilon\}(j-i)}   \left[\left(1-p/(2\tau)\right) 2^{jp }V(E_j)\right] \\
& &+   C  2^{p k /(2\tau)} \sum_{i,j:\  i\leq j}
2^{-\{p[1-p/(2\tau)]-\varepsilon\}(j-i)}   \left[\left(p/(2\tau)\right)2^{kp} V(E_i)\right]
\end{eqnarray*}

Choose $\displaystyle \varepsilon=\f{p\left[1-p/(2\tau)\right]}{2}>0$,
one can apply the last part in the proof of \cite[Theorem~11.3]{Wi08} to
obtain \eqref{e5.24} readily and we omit it.
\end{proof}

\medskip

\noindent
{\it Proof of Proposition~\ref{th5.8} for $1<p\leq 2$. }
For every $b\in \{1,2,\cdots, K\}$ and $m\geq 1$, define $F_{b,m}=\underset{
E_2(P)\in \bigcup_{-m\leq s\leq m}\mathcal{D}_s^{b}}{\sum} a_P(f)(x)$.
Following the line of \cite[Theorem~2.7]{Wi}, for every $k\in \mathbb{N}$, define
\[\mathcal{G}_{b,m}^{k}=\left\{E_2(P)\in \bigcup_{-m\leq s\leq m}
\mathcal{D}_s^{b}:\  \ 2^{k}\leq \f{1}{\int_{E_2(P)} V(x)d\mu(x)} \int_{E_2(P)}
V(x)\left(\log\left(e+\f{V(x)}{V_{E_2(P)}}\right)\right)^{\tau} d\mu(x)<2^{k+1}  \right\},\]
and
\[G_{b,m}^k(x)=\sum_{E_2(P)\in \mathcal{G}_{b,m}^k } a_P(f)(x).\]
It's cleat that $\mathcal{G}_{b,m}^{i}\cap \mathcal{G}_{b,m}^{j}=\emptyset$ for arbitrary $i\neq j$, and
\[ \bigcup_{-m\leq s\leq m} \mathcal{D}_s^b =\bigcup_{k\in \mathbb{N}} \mathcal{G}_{b,m}^{k}.\]

For every $k\in \mathbb{N}$, define
\[\mathcal{E}_{E_2(R)}(G_{b,m}^k)\equiv  \Bigg(\sum_{E_2(R)\subset E_2(P)
\in \mathcal{G}_{b,m}^k } \f{|\lambda_R|^2}{\mu(R)}\Bigg)^{p/2}
- \Bigg(\sum_{E_2(R)\subsetneqq E_2(P)\in \mathcal{G}_{b,m}^k } \f{|\lambda_R|^2}{\mu(R)}\Bigg)^{p/2} \]
for every $E_2(R)\in \bigcup_{-m\leq s\leq m}\mathcal{D}_s^{b}$. As shown in the proof of \cite[Theorem~2.7]{Wi},
\[F_{b,m}=\sum_{i\in \mathbb{N}} G_{b,m}^k,\]
\[\mathcal{E}_{E_2(R)}(G_{b,m}^k)=0\ \  \  \text{if} \  \ \  E_2(R)\notin \mathcal{G}_{b,m}^k,\]
and
\begin{equation*}
\left(\mathscr{A}_L(G_{b,m}^k)\right)^p (x)=\sum_{x\in E_2(R)\in\mathcal{G}_{b,m}^i }
 \mathcal{E}_{E_2(R)}(G_{b,m}^k).
\end{equation*}

\smallskip

The left of the proof  is to follow   the technique as  in the proof of \cite[Theorem~2.7]{Wi} and \cite[Theorem~4.1]{Wi}.
In details, for every  $E_2(P)\in \mathcal{G}_{b,m}^{k}$, $k\in \mathbb{N}$, it follows from the
definition of $\mathcal{G}_{b,m}^{k}$ and \eqref{e5.19} to see
\[
2^{k}V(E_2(P))
\leq    \int_{E_2(P)}  V(x)\left(\log\left(e+\f{V(x)}{V_{E_2(P)}}\right)\right)^{\tau} d\mu(x)
\leq   \f{W(E_2(P))}{{\left[\ell({E_2(P)})\right]^p}},
\]
and the proof of \cite[Theorem~4.1]{Wi} shows that since $p/2\leq 1$,
\begin{eqnarray*}
\f{\mathcal{E}_{E_2(P)}(G_{b,m}^k)}{{\left[\ell({E_2(P)})\right]^p}}
&\leq  & \Bigg(\sum_{E_2(R)\subset E_2(P)\in \mathcal{G}_{b,m}^k }
\f{|\lambda_R|^2}{{(\ell({R}))^2 }\mu(R)}\Bigg)^{p/2}
  - \Bigg(\sum_{E_2(R)\subsetneqq E_2(P)\in \mathcal{G}_{b,m}^k }
  \f{|\lambda_R|^2}{{(\ell({R}))^2}\mu(R)}\Bigg)^{p/2}\notag\\
&=:& \widetilde{\mathcal{E}}_{E_2(P)}(G_{b,m}^k).
\end{eqnarray*}
Moreover,
\[\left(\widetilde{\mathscr{A}}_L(G_{b,m}^k)\right)^p (x)
=\sum_{x\in E_2(P)\in \mathcal{G}_{b,m}^k} \widetilde{\mathcal{E}}_{E_2(P)}(G_{b,m}^k). \]

Observe that for any increasing function $\aleph:\   [0,\infty]\to [1,\infty)$
satisfying $\sum_{k=0}^\infty \aleph(k)^{-1/(p-1)} \leq 1$, there exists a constant
 $C=C(\aleph)$ such that
\[\int_{X} \left|F_{b,m}(x)\right|^p V(x)d\mu(x)  \leq   C(\aleph)
 \sum_{k\in \mathbb{N}}  \aleph(k)  \int_X   \left|G_{b,m}^k(x)\right|^p V(x)d\mu(x). \]

 Since $f\in L_0^\infty(X)$, we have  $\#\mathcal{G}_{b,m}^k <\infty$
 for every $k\in \mathbb{N}$. Hence one can apply Lemma~\ref{le5.9} and
 take $\aleph(k)=2^{[1/2-p/(4\tau)]k}$, $k\in \mathbb{N}$,  to obtain when $1<p\leq 2$,
\begin{eqnarray}
\label{e5.26}
\int_{X} \left|F_{b,m}(x)\right|^p V(x)d\mu(x) & \leq &  C  \sum_{k\in \mathbb{N}}
2^{[1/2-p/(4\tau)]k}  \int_X   \left|G_{b,m}^k(x)\right|^p V(x)d\mu(x)  \notag\\
&\leq & C \sum_{k\in \mathbb{N}}     2^{[1/2-p/(4\tau)]k}   2^{[p /(2\tau)]k} \int_X
\left|  \mathscr{A}_L(G_{b,m}^k)(x)\right|^p V(x)d\mu(x) \notag \\
&=&   C \sum_{k\in \mathbb{N}}    2^{-[1/2-p/(4\tau)]k}  2^{k}  \sum_{E_2(R)
\in \mathcal{G}_{b,m}^{k}} \mathcal{E}_{E_2(R)}(G_{b,m}^k) V(E_2(R)) \notag \\
&\leq &  C  \sum_{k\in \mathbb{N}}2^{-[1/2-p/(4\tau)]k} \sum_{E_2(R)\in \mathcal{G}_{b,m}^{k}}
 \f{\mathcal{E}_{E_2(R)}(G_{b,m}^k)}{{\left[\ell({E_2(R)})\right]^p}}  W(E_2(R)) \notag\\
&\leq &  C  \sum_{k\in \mathbb{N}}  2^{-[1/2-p/(4\tau)]k} \sum_{E_2(R)\in \mathcal{G}_{b,m}^{k}}
  \widetilde{\mathcal{E}}_{E_2(R)}(G_{b,m}^k)  W(E_2(R)) \notag\\
&\leq & C \int_X     \left(\widetilde{\mathscr{A}}_L(F_{b,m})\right)^p  (x) W(x)d\mu(x),
\end{eqnarray}
where the last inequality follows from $1/2-p/(4\tau)>0$.

This, together with \eqref{e5.16}, implies that  \eqref{e5.20} holds for $1<p\leq 2$.
The proof of Proposition~\ref{th5.8} is complete.
\hfill{}$\Box$

\bigskip

As a consequence of Proposition~\ref{th5.8}, we have the following result.

\begin{theorem}
\label{cor5.10}
 Assume that $L$ is a
densely-defined operator
on $L^2(X)$
satisfying \textbf{(H1)}, \textbf{(H2)} and \textbf{(H3)}.
Let $0<p\leq 2$, $\tau>p/2$, and let $V$ and $W$ be two weights such that
\begin{equation*}
\left({\ell({Q})}\right)^p \int_Q V(x)\left(\log \left(e+\f{V(x)}{V_Q}
\right)\right)^{\tau} d\mu(x)\leq \int_Q W(x)d\mu(x)
\end{equation*}
for all $Q\in \bigcup_{b=1}^K \mathcal{D}^b$.
 For every   $\alpha>0$,  there exists a positive constant $C=C(p,n,\tau,\alpha)$, such that for every $f\in L_0^\infty(X)$,
\begin{equation}
\label{e5.27}
\int_X |f(x)|^p V(x)d\mu(x)\leq C \int_X
 \left( \widetilde{S}_{L,\alpha} (\sqrt{L}f)(x)\right)^p W(x)d\mu(x).
\end{equation}
where
\[\widetilde{S}_{L,\alpha}(g)(x):=\left(\int_0^\infty \int_{d(y,x)<\alpha t}
\left|t\sqrt{L}e^{-t\sqrt{L}}g(y)\right|^2\f{dy}{\mu(B(x,t))}\f{dt}{t}\right)^{1/2},\ \  g\in L^2(X), \  \ \alpha>0.\]
In particular, when $1< p\leq 2$ and $W\in A_p$, {there exists a positive constant $C=C(p,n,\tau,\|W\|_{A_p})$} such that for every $f\in L_0^\infty(X)$,
  \begin{equation}
\label{e5.28}
\int_X |f(x)|^p V(x)d\mu(x)\leq C \int_X   \big| \sqrt{L}f(x)\big|^p W(x)d\mu(x).
\end{equation}
\end{theorem}

\begin{proof}
{Similar to Lemma~\ref{le4.3}, one can verify that there exists a  constant $C=C(n,\alpha)>0$ such that}
\[\widetilde{\mathscr{A}}_L(f)(x)\leq C \widetilde{S}_{L,{\alpha}}(\sqrt{L}f)(x), \  \  a.e. \ x\in X.\]
The proof of \eqref{e5.27} is similar to that of Theorem~\ref{th5.2} and we omit it.
Furthermore,   one can apply   \cite[ Proposition~3.3]{GY} and  \cite[Proposition~3.4]{GY2}
to deduce that $\widetilde{S}_{L, 1}$ is bounded on $L_w^q(X)$ for $w\in A_q$, $1<q<\infty$.
Hence, when $1<p\leq 2$ and $W\in A_p$,
\[\int_X \left(\widetilde{S}_{L, 1}(\sqrt{L}f)(x)\right)^p W(x)d\mu(x)  \leq C \int_X
\big|\sqrt{L}f(x)\big|^p W(x)d\mu(x) \]
for some constant {$C=C(p, n,\|W\|_{A_p})$},  which deduces that \eqref{e5.28} holds. The proof of Theorem~\ref{cor5.10} is finished.
\end{proof}

\bigskip

\section{Eigenvalue estimates for Schr\"odinger operators}
\setcounter{equation}{0}

 Recall that $\Omega=\{(X, y)\in {\mathbb R}^{n-1}\times {\mathbb R}: y>\Phi(X)\} $ is
a  special Lipschitz domain of ${\mathbb R}^n$ where
  $\Phi: {\mathbb R}^{n-1}\mapsto {\mathbb R}$ is a Lipschitz function,
i.e., $||\nabla\Phi||_{\infty}\leq M<\infty$  for some constant $M$.
Let $A(x)$ be an $n\times n$
matrix function with real symmetric, bounded measurable  entries  on
${\mathbb R}^n$ satisfying the ellipticity condition
\begin{eqnarray}\label{e6.1}
\|A\|_{\infty}\leq \lambda^{-1}\ \ \ {\rm and}\ \ \
\  A(x)\xi\cdot\xi\geq \lambda |\xi|^2
\end{eqnarray}
for some constant $\lambda\in (0,1)$,  for all $\xi\in{\mathbb R}^n$
and for almost all $x\in {\mathbb R}^n$.
Let $V\in L^1_{\rm loc}(\Omega)$ be nonnegative. Let
$H$ be  a closed subspace of
\begin{eqnarray}\label{e6.2}
\big\{u\in W^{1,2}(\Omega): \int_{\Omega} V|u|^2dx <\infty \big\}
\end{eqnarray}
containing $W^{1,2}_0(\Omega),$
and
${\mathcal L}$ be a second order elliptic operator
in divergence form $ {\mathcal L}=-{\rm div} (A\nabla)-V $ on $L^2(\Omega)$
with largest domain ${\bf Dom}({\mathcal L})\subset H$ such that
\begin{eqnarray}\label{e6.3}
\langle  {\mathcal L}f, g\rangle=\int_{\Omega} A\nabla f\cdot\nabla g dx - \int_{\Omega} V f\cdot  g dx, \ \ \ \ \forall f\in
{\bf Dom}( {\mathcal L}),\ \ \forall g\in H.
\end{eqnarray}
 We use the notation $ {\mathcal L}=(A, \Omega, H)$ to denote any operator defined as above.
 We say that $ {\mathcal L}$ satisfies the Dirichlet boundary condition (DBC) when $H=W^{1,2}_0(\Omega)$,
 the Neumann  boundary condition (NBC) when $H=W^{1,2}(\Omega)$.

 In the sequel, we denote   ${\mathcal L}$ by $L$   when $V=0$.
 It is known that when $\Omega={\mathbb R}^n$ or $\Omega$ be a special Lipschitz domain of ${\mathbb R}^n$
 and $ L=-{\rm div} (A\nabla)$ is    real symmetric operators
(under DBC or NBC) as above,
 properties  \textbf{(H1)}, \textbf{(H2)} and \textbf{(H3)}
  for operator $L$  are always satisfied (see \cite{AR, Da}).
Furthermore, it follows by \eqref{e6.1} that there exists a constant $C>0$ such that for every $f\in L^2(\Omega)$
\begin{eqnarray}\label{e6.4}
\lambda \|\nabla f\|_2 \leq  \|L^{1/2} f\|_2 \leq \lambda^{-1} \|\nabla f\|_2.
\end{eqnarray}

 Now, let $\mathcal{L}=(A, \Omega, V)$ be as above. In studying the spectrum of $\mathcal{L}$,
 a central question
is to find conditions for $\mathcal{L}\geq 0$ (see \cite{CW, CWW, Fe, KS}).
 Observe that integration by parts,
together with \eqref{e6.2}, shows that $\mathcal{L}\geq 0$ follows from the inequality
\begin{eqnarray}\label{e6.5}
\int_{\Omega} |f(x)|^2 V(x)dx\leq  c\int_{\Omega}\left|\nabla  f(x)\right|^2  dx
\end{eqnarray}
for all $f\in {\bf Dom}( {\mathcal  L})$ and $c>0$ sufficiently small.

To show \eqref{e6.5},    we apply the previous result in the setting $X=\Omega$ and $L$
to  give a sufficient condition
    for $\mathcal{L}\geq 0$, by adapting the argument as in  Theorem~5.1 in \cite{Wi} to the present
	situation.

\begin{theorem}
\label{th6.1}
When $\Omega={\mathbb R}^n$ or $\Omega$ be a special Lipschitz domain of ${\mathbb R}^n$
 and $ \mathcal{L}=-{\rm div} (A\nabla)-V$ be    real symmetric operators
(under DBC or NBC) with nonnegative $V\in L^1_{\rm loc}(\Omega)$ as  above.
 Let $\sigma(\mathcal{L})$ be the lowest nonpositive eigenvalue of $\mathcal{L}$ on $\Omega$.
Fix a  $\tau>1$. For  every ball $B=B(x_B, r_B)$  in $\Omega$,
\begin{equation}
\label{e6.66}
\Lambda(B,V)=\int_B V(x)\left(\log \left(e+\f{V(x)}{V_B}\right)\right)^{\tau} d\mu(x).
\end{equation}
There exist two positive constants $c_1$, $c_2$ which depend on $n$, $\tau$ such that
\begin{equation}
\label{e6.77}
\sigma(\mathcal{L}) \geq -\sup_{B}
c_1\left[|B|^{-1}\Lambda(B,V)-c_2  r_B^{-2}|B|\right],
\end{equation}
where $|B|$ denotes the Lebesgue measure of $B$ in $\Omega$.

As a consequence, if
\[ r_B^2 \Lambda(B,V) \leq c_2 |B| \]
for all balls $B$, then $\mathcal{L}\geq 0$.
\end{theorem}

To show Theorem~\ref{th6.1},
for every $Q\in {\mathcal D}=\bigcup_{b=1}^{K} {\mathcal D}^b$ we define
\begin{equation}
\label{e6.6}
\Lambda(Q,V)=\int_Q V(x)\left(\log \left(e+\f{V(x)}{V_Q}\right)\right)^{\tau} d\mu(x),\ \ \  \tau>1.
\end{equation}
Let $\ell(Q)$ denote the side length of the cube $Q.$
 Then the following result holds.

\begin{lemma}\label{le6.1}
Fix  $\tau>1$.
For every $Q\in {\mathcal D}$,  there exists ball $B$ such that
$Q\subset B$ and $|B|\leq C|Q|$ and $r_B\sim \ell(Q)$ such that
\begin{eqnarray}\label{ef1}
\Lambda(Q,V)
\leq C\Lambda(B,V).
\end{eqnarray}
\end{lemma}

\begin{proof}
Recall that from \cite[Theorem~11.2]{Wi08}, if $V_Q\neq 0$,
for all $\gamma>0$,
\begin{equation}
\label{e4.300}
\int_Q   V(x)\left(\log\left(C_6 +\f{V(x)}{V_Q}\right)\right)^{\gamma}d\mu(x)
 \sim \sup\left\{\int_Q V(x)\phi(x)d\mu(x):\  \  \phi\in
 {\rm{Exp}}_{e}(Q,1/\gamma)
  \right\},
\end{equation}
where $C_6$ is the constant in \eqref{e4.19}, and
\[{\rm{Exp}}_{e}(Q,1/\gamma)=\left\{\psi:\    \f{1}{|Q|} \int_Q
\exp\left(\psi^{1/\gamma}(x)\right)d\mu(x) \leq e+1 \right\}.\]
Note that \eqref{e4.300} still holds if we replace dyadic cube $Q$ by
ball $B$.

For $\phi\in  {\rm{Exp}}_{e}(Q,1/\tau)$, define ${\tilde \phi}(x)=\phi(x)$ for $x\in Q$
and ${\tilde \phi}(x)=0$ for $x\notin Q$.
Then
\begin{eqnarray*}
\f{1}{|B|} \int_B
\exp\left({\tilde \phi}^{1/\tau}(x)\right)d\mu(x) &=&
\f{1}{|B|} \int_Q
\exp\left({\tilde \phi}^{1/\tau}(x)\right)d\mu(x)+\f{1}{|B|} \int_{B\backslash Q}
\exp\left({\tilde \phi}^{1/\tau}(x)\right)d\mu(x)\\
&\leq& \f{1}{|Q|} \int_Q
\exp\left(\phi^{1/\tau}(x)\right)d\mu(x)+1\\
&\leq& e+1+1,
\end{eqnarray*}
which implies that ${\tilde \phi}\in {\rm{Exp}}_{e+1}(B,1/\tau)$.
Then for all $\phi\in  {\rm{Exp}}_{e}(Q,1/\tau)$
\begin{eqnarray*}
\int_Q V(x)\phi(x)d\mu(x)&=& \int_B V(x){\tilde \phi}(x)d\mu(x)\\
&\leq & \sup\left\{\int_B V(x)\psi(x)d\mu(x):\  \  \psi\in
 {\rm{Exp}}_{e+1}(B,1/\tau)\right\}.
\end{eqnarray*}
Therefore,
\begin{eqnarray*}
\Lambda(Q,V)&=&\int_Q   V(x)\left(\log\left(e +\f{V(x)}{V_Q}\right)\right)^{\tau}d\mu(x)\\
&\leq& C\sup\left\{\int_Q V(x)\psi(x)d\mu(x):\  \  \psi\in
 {\rm{Exp}}_{e}(Q,1/\tau)
  \right\}\\
&\leq&C \sup\left\{\int_B V(x)\psi(x)d\mu(x):\  \  \psi\in
 {\rm{Exp}}_{e+1}(B,1/\tau)\right\}\\
&\leq&C \int_B   V(x)\left(\log\left(e+1 +\f{V(x)}{V_B}\right)\right)^{\tau}d\mu(x)\\
&\leq&C \int_B   V(x)\left(\log\left(e +\f{V(x)}{V_B}\right)\right)^{\tau}d\mu(x)\\
&\leq&C \Lambda(B,V).
\end{eqnarray*}
This proves Lemma~\ref{le6.1}.
\end{proof}

\noindent
{\it Proof of Theorem~\ref{th6.1}}.\
First, we assume that for all $Q\in \mathcal{D}$,
\begin{equation}
\label{e6.8}
|Q|^{-1}\Lambda(Q,V) -c_3 \big(\ell(Q)\big)^{-2}|Q|\leq {   C}_{11},
\end{equation}
where $c_3$ will be chosen later. We will show that $\mathcal{L}\geq -cC_{11}$ for some positive constant $c$.

For arbitrary $f\in C_0^\infty (\Omega)$, let $F_{b,m}$, $\mathcal{G}_{b,m}^k$, $G_{b,m}^k$,
$\mathcal{E}_{E_2(P)}(G_{b,m}^k)$, $\widetilde{\mathcal{E}}_{E_2(P)}(G_{b,m}^k)$
be as in  the proof of Proposition~\ref{th5.8} for $1<p\leq 2$, with $X=\Omega$ and the
corresponding self-adjoint operator is $ {L}=-{\rm div} (A\nabla)$.
It
follows from the proof of Theorem~\ref{th4.7} that there exists $\{m_i\}$ such that
\begin{eqnarray}
\label{e8.16}
	\int_X |f(x)|^p V(x)d\mu(x)
	&\leq&  \underset{i\to \infty}{\underline{\lim}} \int_X
	\Big|\sum\nolimits_{b=1}^K  F_{b,m_i}(x)\Big|^p V(x)d\mu(x)  \notag\\
	&\leq & C(K,p) \sum\nolimits_{b=1}^K  \underset{i\to \infty}{\underline{\lim}}
	\int_X |F_{b,m_i}(x)|^p V(x)d\mu(x).
\end{eqnarray}
An argument as in \eqref{e5.26} shows that
\begin{eqnarray*}
\int_{X} \big|F_{b,m}(x)\big|^p V(x)d\mu(x) &\leq & C \sum_{k\in \mathbb{N}}    2^{k}
 \sum_{Q\in \mathcal{G}_{b,m}^{k}} \mathcal{E}_{Q}(G_{b,m}^k)\, V(Q)  \\
&\leq &  C  \sum_{k\in \mathbb{N}} \sum_{Q\in \mathcal{G}_{b,m}^{k}}
  \mathcal{E}_{Q}(G_{b,m}^k)\,  \Lambda(Q,V) \\
&\leq & C  \sum_{k\in \mathbb{N}} \sum_{Q\in \mathcal{G}_{b,m}^{k}}
 \Bigg[   c_3\, \frac{\mathcal{E}_{Q}(G_{b,m}^k)}{\big(\ell(Q)\big)^2}  |Q|
   +C_{11}\, \mathcal{E}_{Q}(G_{b,m}^k)\,|Q| \Bigg]\\
&\leq & C \sum_{k\in \mathbb{N}} \sum_{Q\in \mathcal{G}_{b,m}^{k}}
\left[  c_3 \, \widetilde{\mathcal{E}}_Q(G_{b,m}^k)\, |Q|  +C_{11}\, \mathcal{E}_{Q}(G_{b,m}^k)\,|Q|   \right]\\
&=&  C \sum_{k\in \mathbb{N}} \left\{ c_3
\int_{\Omega}\left(\widetilde{\mathscr{A}_{ {L}}}(G_{b,m}^k)\right)^2 (x) dx
+ C_{11} \int_{\Omega} \left(\mathscr{A}_{ {L}}(G_{b,m}^k)\right)^2(x)dx \right\}\\
&=& C \left\{ c_3 \int_{\Omega}\left(\widetilde{\mathscr{A}_{ {L}}}(F_{b,m})\right)^2 (x)  dx
 + C_{11} \int_{\Omega} \left(\mathscr{A}_{ {L}}(F_{b,m})\right)^2(x)dx\right\}.
\end{eqnarray*}
There exists a constant $C>0$   such that
$$
\int_{\Omega}\left(\widetilde{\mathscr{A}_{ {L}}}(F_{b,m})\right)^2 (x)
dx \leq \int_{\Omega}\left(\widetilde{\mathscr{A}_{ {L}}}(f)\right)^2 (x)
dx \leq  C  \int_{\Omega} \left| {L}^{1/2}(f)(x)\right|^2  dx.
$$
Notice that $S_{L,\alpha}$ is bounded on $L^p$ for $1<p<\infty$ and $\alpha>0$ (see \cite[Proposition 3.3]{GY}), we have
$$
\int_{\Omega} \left(\mathscr{A}_{ {L}}(F_{b,m})\right)^2(x)dx
\leq \int_{\Omega} \left(\mathscr{A}_{ {L}}(f)\right)^2(x)dx \leq
 \int_{\Omega} \left(S_{ {L,5\delta^{-2}}}(f)\right)^2(x)dx   \leq C \int_{\Omega} |f(x)|^2 dx.
$$

These, combined with \eqref{e8.16} and \eqref{e6.4}, deduce that
\begin{eqnarray}\label{kkk}
\int_\Omega  |f(x)|^2 V(x)dx&\leq& C_{12} c_3  \int_\Omega \left| {L}^{1/2}f(x)\right|^2  dx
+CC_{11}\int_\Omega |f(x)|^2 dx\nonumber\\
&\leq&   C_{12}c_3 \lambda^{-1}\int_\Omega \left| \nabla f(x)\right|^2  dx+C  C_{11}\int_\Omega |f(x)|^2 dx
\end{eqnarray}
for some constant $C_{12}$ and  $\lambda$ is the constant in \eqref{e6.1}.
We take $c_3=\lambda^2/(2C_{12})$ in \eqref{kkk} and use  \eqref{e6.1} and \eqref{e6.3} to obtain
\[\langle \mathcal{L} f,f\rangle \geq    C_{12}c_3 \lambda^{-1}\int_\Omega \left| \nabla f(x)\right|^2  dx
-\int_\Omega |f(x)|^2 V(x)dx \geq  -C  C_{11} \langle f,f \rangle, \  \  \   f\in {\bf Dom}( {\mathcal L}).   \]

Finally,  it follows by Lemma~\ref{le6.1} that
\begin{eqnarray*}
\sigma(\mathcal{L}) &\geq& -\sup_{Q}
c\left[|Q|^{-1}\Lambda(Q,V)-c_3  \ell(Q)^{-2}|Q|\right]\\
&\geq& -\sup_{B}
c_1\left[|B|^{-1}\Lambda(B,V)-c_2   r_B^{-2}|B|\right]
\end{eqnarray*}
for some $c_1, c_2>0$. The proof of Theorem~\ref{th6.1} is complete.
   \hfill{}$\Box$

\bigskip

\bigskip

 \noindent
{\bf Acknowledgments.}
P. Chen was partially  supported by NNSF of China 11501583, Guangdong Natural
Science Foundation 2016A030313351 and the Fundamental Research Funds for the Central
Universities 161gpy45.  X. T. Duong  is supported by
Australian Research Council  Discovery Grant DP 140100649.
 L.C. Wu and L.X.~Yan are supported by the NNSF
of China, Grant Nos.~11371378 and  ~11521101).

 \end{document}